\documentclass[pdflatex,sn-mathphys-num]{sn-jnl}

\usepackage{graphicx}%
\usepackage{multirow}%
\usepackage{amsmath,amssymb,amsfonts}%
\usepackage{amsthm}%
\usepackage{mathrsfs}%
\usepackage[title]{appendix}%
\usepackage{xcolor}%
\usepackage{textcomp}%
\usepackage{manyfoot}%
\usepackage{booktabs}%
\usepackage{algorithm}%
\usepackage{algorithmicx}%
\usepackage{algpseudocode}%
\usepackage{listings}%
\usepackage{csquotes}
\usepackage{todonotes}
\usepackage{commath}
\usepackage{mathtools}
\usepackage{bbm}
\usepackage{nicefrac}
\usepackage{subdepth}
\usepackage{bm}
\usepackage{siunitx}
\usepackage{subcaption} 
\usepackage{enumitem}
\usepackage{trimclip}
\usepackage{thmtools}
\usepackage[utf8]{luainputenc}

\theoremstyle{thmstyleone}%
\newtheorem{theorem}{Theorem}
\theoremstyle{thmstyletwo}%
\newtheorem{example}{Example}%
\newtheorem{remark}{Remark}%

\theoremstyle{thmstylethree}%
\newtheorem{corollary}{Corollary}%
\newtheorem{lemma}{Lemma}%
\renewcommand{\b}[1]{\mathbf{#1}}
\newcommand{\w}[1]{\widetilde{#1}}

\def\code#1{\texttt{#1}}

\newcommand{\bsigma}{\boldsymbol{\sigma}}

\renewcommand{\O}{\mathcal{O}}

\newcommand{\dd}{\mathrm{d}}

\newcommand{\qta}{\quad\text{and}\quad}

\newcommand{\tend}{t_\mathrm{end}}

\newcommand{\bui}{\b u^{(i)}}
\newcommand{\buj}{\b u^{(j)}}

\newcommand{\ui}{u^{(i)}}

\newcommand{\fnu}{\mathbf{f}^{[\nu]}}

\newcommand{\tn}{t_n}
\newcommand{\told}{t_\text{old}}
\newcommand{\tnew}{t_\text{new}}
\newcommand{\uold}{\b u_\text{old}}
\newcommand{\unew}{\b u_\text{new}}
\newcommand{\etaold}{\eta_\text{old}}
\newcommand{\etanew}{\eta_\text{new}}
\newcommand{\tgamman}{t_\gamma^n}
\newcommand{\ugamman}{\b u_\gamma^n}
\newcommand{\unplusgamma}{u^{n+\gamma}}
\newcommand{\bunplusgamma}{\b u^{n+\gamma}}
\newcommand{\etaugamman}{\eta(\ugamman)}

\usepackage{xparse}

\let\epsilon\varepsilon
\let\phi\varphi
\let\rho\varrho

\usepackage{xspace}

\newcommand{\eg}[0]{{e.g.\@}\xspace}
\newcommand{\ie}[0]{{i.e.\@}\xspace}

\renewcommand{\O}{\mathcal{O}}

\newcommand{\1}{\mathbbm{1}}

\providecommand\R{}
\renewcommand{\R}{\mathbb{R}}

\newcommand{\dt}{\Delta t}
\newcommand{\dx}{\Delta x}

\newcommand{\diag}{\operatorname{diag}}

\renewcommand{\Vec}[1]{\renewcommand*{\arraystretch}{1.2}\begin{pmatrix*}[r]#1\end{pmatrix*}}
\newcommand{\cVec}[1]{\renewcommand*{\arraystretch}{1.2}\begin{pmatrix*}#1\end{pmatrix*}}

\makeatletter
\newcommand*\dashline{\rotatebox[origin=c]{90}{$\dabar@\dabar@\dabar@$}}
\makeatother
\begin{document}

	\title[Positivity-Preserving Relaxation]{A Positivity-Preserving Relaxation Algorithm}

	\author*[1,3]{\fnm{Thomas} \sur{Izgin}}\email{izgin@mathematik.uni-kassel.de}
	\author[2]{\fnm{Hendrik} \sur{Ranocha}}\email{hendrik.ranocha@uni-mainz.de}
	\author[3]{\fnm{Chi-Wang} \sur{Shu}}\email{chi-wang\_shu@brown.edu}

	\affil*[1]{\orgdiv{Department of Mathematics}, \orgname{University of Kassel}, \orgaddress{\street{Heinrich-Plett-Str. 40}, \postcode{34132}, \city{Kassel}, \country{Germany}}}
	\affil[2]{\orgdiv{Institute of Mathematics}, \orgname{Johannes Gutenberg University Mainz}, \orgaddress{\street{Staudingerweg 9}, \postcode{55128}, \city{Mainz}, \country{Germany}}}
	\affil[3]{\orgdiv{Division of Applied Mathematics}, \orgaddress{\orgname{Brown University}, \city{Providence}, 
\state{Rhode Island} \postcode{02906}, \country{USA}}}


	\abstract{We combine Patankar-type methods with suitable relaxation procedures that are capable of ensuring correct dissipation or conservation of functionals such as entropy or energy while producing unconditionally positive and conservative approximations. To that end, we adapt the relaxation algorithm to enforce positivity by using either ideas from the dense output framework when a linear invariant must be preserved, or simply a geometric mean if the only constraint is positivity preservation. The latter merely requires the solution of a scalar nonlinear equation while former results in a coupled linear-nonlinear system of equations. We present sufficient conditions for the solvability of the respective equations. Several applications in the context of ordinary and partial differential equations are presented, and the theoretical findings are validated numerically.}




	\maketitle

	\section{Introduction}

	We consider initial-value problems (IVPs)
	\begin{align}\label{eq:ivp}
		\b u'(t)  =\b f(\b u(t)), \quad  \b u(t_0) & = \b u^0\in \R^d,
	\end{align}
	either as classical ordinary differential equation (ODE) model on its own, or more
typically obtained after discretizing a partial differential equation (PDE) in space.
	We are interested in two types of structures of the IVP \eqref{eq:ivp}.
	First, many applications require positive solutions, i.e., $\b u(t)>\b 0$ for all $t\geq t_0$ if $\b u^0>\b 0$, where inequalities are understood component-wise.
	This occurs, for example, when modeling chemical reactions, population dynamics, or the density of fluids.
	Second, many problems are equipped with additional functionals of interest, such as Lyapunov functionals, energy, or entropy.
	We say that the IVP \eqref{eq:ivp} is \emph{dissipative} with respect to a smooth functional $\eta$, if $\eta'(\bf u) \bf f(\bf u) \le 0$, i.e.,
	\begin{equation*}
		\frac{\dd}{\dd t}\eta(\b u(t))\leq 0
	\end{equation*}
	for all solutions $\b u(t)$ of \eqref{eq:ivp}.
	Similarly, \eqref{eq:ivp} is \emph{conservative} with respect to $\eta$, if $\eta'(\bf u) \bf f(\bf u) = 0$, i.e.,
	\begin{equation*}
		\frac{\dd}{\dd t}\eta(\b u(t))= 0.
	\end{equation*}
	In addition to the (typically nonlinear) functional $\eta$, many problems also conserve additional linear invariants, such as mass or momentum, which we also want to preserve on the discrete level.

	When discretizing \eqref{eq:ivp} in time using a one-step method, we would like to preserve these properties, i.e., we would like to have an unconditionally positive method satisfying
	\begin{equation*}
		\b u^0 > \b 0 \implies \b u^n>\b 0 \quad \text{for all } n\geq 0.
	\end{equation*}
	For many positive ODEs/PDEs, avoiding negative approximations is critical; such artifacts can lead to qualitatively incorrect solutions or the total failure of the numerical method \cite{BBKS2007, sandu2001positive, STKB2005, SSPMPRK2}.
	Moreover, for dissipative problems, we would like to use a dissipative method that satisfies
	\begin{equation}\label{eq:discrete_entropy_dissipative}
		\eta(\b u^n)\leq \eta(\b u^{n-1})\leq\dotsc \leq \eta(\b u^0).
	\end{equation}
	Similarly, a conservative method applied to a conservative problem should satisfy
	\begin{equation}\label{eq:discrete_entropy_conservative}
		\eta(\b u^n)= \eta(\b u^{n-1})=\dotsc = \eta(\b u^0).
	\end{equation}

	For convex $\eta$, the implicit Euler method is a well-known example of an unconditionally positive and dissipative method.
	However, this analysis neglects possible positivity issues that can arise while solving the implicit equations as well as remaining errors of the nonlinear iterative solver.

	Concerning positivity, the implicit Euler method is essentially the best method one can use in the class of general linear methods, since any unconditionally positive method can be at most first-order accurate \cite{bolley1978conservation}.
	To address this challenge, several strategies have been proposed:
		\begin{enumerate}
		\item \textit{Clipping techniques}, which forcibly set negative values to zero, either result in a mass-shifting optimization problem or otherwise compromise conservation of linear invariants, and, to date, lack a proof of stability \cite{BIM2022}.
		\item \textit{Projection techniques} \cite{sandu2001positive,nusslein2021positivity} can be positive and conserve linear invariants, but they may result in step size constraints and/or reduced accuracy.
		\item \textit{Fully implicit, nonlinear methods} \cite{HR2020,ricchiuto2011habilitation} can enforce positivity but require costly iterative solvers, which may fail to converge (to a positive solution), and thus, still produce nonphysical results.
		\item \textit{Diagonally split Runge--Kutta (DSRK) methods} \cite{horvath_positivity_1998} can be unconditionally positive and with order higher than one. However, they are typically less accurate than the implicit Euler method in practice \cite{macdonald2007}.
		\item \textit{Adaptive methods} \cite{STKB2005} use root-finding procedures and adapt the time step size. This can be effective, but the resulting schemes are only conditionally positive.
		\item \textit{Strong stability preserving (SSP) methods} \cite{GKS2011} are positive if the explicit Euler method is positive under a certain time step restriction. However, only the implicit Euler method leads to unconditional positivity, and thus, all other SSP methods are only conditionally positive.
		\item \textit{Patankar-type methods} represent a family of explicit or linearly implicit yet nonlinear schemes, which are unconditionally positive and can preserve certain linear invariants \cite{Patankar1980,BDM2003,MCD2020,KM18,AKM2020}.
	\end{enumerate}
	In this work, we focus on Patankar-type schemes.
	The main idea behind them is to modify an existing time-stepping method by introducing nonlinear weights in such a way that the resulting numerical scheme becomes unconditionally positive.
	The primary challenge lies in designing these weights so that the modified scheme preserves the accuracy of the original (baseline) method.
	This nonlinear modification is achieved using the so-called \emph{Patankar-trick} \cite{Patankar1980}, which gives this family of methods its name.
	A notable example is the incorporation of modified Patankar (MP) weights into classical Runge–Kutta (RK) schemes, leading to the development of modified Patankar–Runge–Kutta (MPRK) methods \cite{BDM2003,KM18,KM18Order3}, which in addition to being unconditionally positive, are also conservative.
	Motivated by their strong numerical performance, the Patankar-trick has since been successfully extended to a variety of time integration frameworks, including SSP Runge–Kutta (SSPRK) methods \cite{SSPMPRK2,SSPMPRK3}, arbitrary high-order Deferred Correction (DeC) schemes \cite{MPDeC}, generalized BBKS methods \cite{AKM2020}, GeCo schemes \cite{MCD2020}, and linear multistep methods \cite{IMPV2025}.
	The resulting modified schemes all belong to the broader Patankar-type family, which can themselves be recast as non-standard additive Runge--Kutta (NSARK) methods, see \cite{NSARK, IzginThesis}.

	Concerning the preservation (conservation/dissipation) of functionals $\eta$, several results are available for linear schemes such as RK methods applied to linear problems \cite{tadmor2002semidiscrete,ranocha2018L2stability,sun2017stability,sun2019strong,achleitner2024necessary,tadmor2025stability,sun2022energy} and fully-implicit methods \cite{lefloch2002fully,friedrich2019entropy,burrage1979stability,burrage1980nonlinear,dahlby2011preserving}.
	There are also positive results on dissipative schemes if the problem is sufficiently dissipative \cite{higueras2005monotonicity,jungel2015entropy,jungel2017entropy}.
	In the general case including conservative problems, however, results are restrictive and include many negative results \cite{ranocha2021strong,ranocha2020energy}.
	Similarly to positivity preservation, postprocessing/projection methods can be used to enforce the desired conservation/dissipation properties of time integration methods \cite{Shampine1986,grimm2005geometric,calvo2006preservation,calvo2010projection,laburta2015numerical}.
	In this work, we focus on the relaxation approach \cite{ketcheson2019relaxation,ranocha2020relaxation,ranocha2020general}, which can be used to enforce conservation/dissipation of functionals while preserving all linear invariants.
	The basic idea of relaxation methods goes back to \cite{sanzserna1982explicit} and \cite[pp.~265--266]{dekker1984stability}.

	Thus, there are several studies and methods devoted to either positivity preservation or the preservation of functionals such as entropy, but to the best of our knowledge, there is no high-order method that can guarantee both properties simultaneously.
	The main contribution of this work is to design a modified relaxation algorithm capable of simultaneously preserving positivity and conservation/dissipation of functionals.
	To that end, we first equip unconditionally positivity-preserving NSARK methods with suitable estimates for dissipative entropies by applying the relaxation framework from \cite{ranocha2020general}.
	While relaxation can be rendered positivity-preserving for dissipative problems with minor adjustments (see Remark~\ref{rem:pos_relax_diss}), entropy-conservative problems require more sophisticated treatment.
	Leveraging dense output formulae for MPRK methods \cite{izgin2024}, we propose a modified relaxation step that ensures unconditional positivity.
	Furthermore, we introduce a bootstrapping technique to achieve arbitrarily high-order accuracy in time for MPRK schemes.

	The remainder of the paper is structured as follows. We recall the relaxation technique from \cite{ranocha2020general} in Section~\ref{sec:relax}. In Section~\ref{sec:NSARK} we give a brief introduction to NSARK methods.  After that, we explain in Section~\ref{sec:dissipative} how to apply the relaxation algorithm for NSARK schemes and entropy dissipative problems. The main result is given for the entropy-conservative case, see Section~\ref{sec:cons_entropy}, where we equip different families of MP schemes with a positivity-preserving relaxation algorithm and present a bootstrapping technique to obtain arbitrary high order (in time) for MPRK schemes. Finally, we present several examples of ordinary and partial differential equations and validate our findings for second- and third-order MP schemes.

	\section{Preliminaries}\label{sec:preliminaries}

	In this section, we briefly review relaxation methods to preserve functionals $\eta$ and non-standard additive Runge--Kutta (NSARK) methods, which includes Patankar-type methods as a special case.

	\subsection{Classical Relaxation}\label{sec:relax}
	One way to guarantee dissipation \eqref{eq:discrete_entropy_dissipative} or conservation \eqref{eq:discrete_entropy_conservative} of functionals $\eta$ is the relaxation procedure explained in \cite{ranocha2020general}. We are given a numerical one-step method of order $p\geq 2$ generating approximations $\b u^n$ to $\b u(t_n)$ with a time step size of $\dt$. We then have to repeat the following steps, starting with $n=0$.
	\begin{enumerate}
		\item Define the quantities $(\told,\uold,\etaold)\coloneqq(t_n,\b u^n, \eta(\b u^n))$ as well as $(\tnew,\unew)\coloneqq(t_{n+1},\b u^{n+1})$.
		\item
		\begin{itemize}
			\item For dissipative problems \eqref{eq:ivp} compute a suitable estimate \[\etanew=\eta(\unew)+\O(\dt^{p+1}), \quad \dt\to 0.\]
			\item For conservative problems we can simply set $\etanew\coloneqq\etaold$, since we arrive at $\etanew=\eta(\uold)=\eta(\b u(t_{n+1}))=\eta(\unew)+\O(\dt^{p+1})$ by means of an induction over $n$.
		\end{itemize}
		\item Solve the system
		\begin{equation}\label{eq:System_Relaxation_step}
			\cVec{\tgamman\\
				\ugamman\\
				\etaugamman} = \Vec{\told\\ \uold\\ \etaold}+ \gamma \Vec{\tnew -\told\\ \unew - \uold \\ \etanew- \etaold}
		\end{equation}
		by inserting $\ugamman$ into the last equation and solving for $\gamma\approx 1$, and then computing $\tgamman$ and $\ugamman$ according to the remaining equations.
		\item Proceed with the numerical scheme using $\tgamman$ and $\ugamman$ instead of $t_{n+1}$ and $\b u^{n+1}$.
	\end{enumerate}
	For dissipative problems, the \enquote{suitable estimate $\etanew$} must guarantee the discrete dissipativity \eqref{eq:discrete_entropy_dissipative} for the approximations from the relaxation procedure. We will introduce such a suitable estimate for NSARK methods that are based on ARK methods with a non-negative extended Butcher tableau in Section~\ref{sec:dissipative}. For now, let us proceed by revisiting the main results from \cite{ranocha2020general}, assuming we have such an $\etanew$ at hand.
	\begin{theorem}[{\cite[Theorem~2.13, Theorem~2.14]{ranocha2020general}}]\label{thm:relax}
		Consider the relaxation procedure \eqref{eq:System_Relaxation_step} with a numerical method of order $p$ and $\dt>0$ sufficiently small. If
		\begin{subequations}\label{eq:cond_eta}
			\begin{equation}\label{eq:cond_eta_conv}
				\eta \text{ is convex and } \eta''(\uold)(\b f(\uold),\b f(\uold))\neq 0\quad \text{ or }
			\end{equation}
			\begin{equation}\label{eq:cond_eta_degen}
				\eta'(\unew)\frac{\unew-\uold}{\Vert \unew-\uold\Vert}=c(\uold)\dt+\O(\dt^2) \text{ with } c\neq 0,
			\end{equation}
		\end{subequations}
		then there exists a unique $\gamma=1+\O(\dt^{p-1})$ that satisfies \eqref{eq:System_Relaxation_step}. Additionally, the relaxation method is of order $p$, that is $\ugamman=\b u(\tgamman)+\O(\dt^{p+1})$.
		In particular, there exist $\gamma_1,\gamma_2>0$ such that
		\begin{equation}\label{eq:r}
			r(\gamma)\coloneqq\eta(\uold +\gamma (\unew-\uold))-\left(\etaold +\gamma(\etanew-\etaold)\right)
		\end{equation}
		satisfies $r(\gamma_1)r(\gamma_2)<0$ for $\dt>0$ small enough.
	\end{theorem}
	This theorem is the theoretical basis for the existence and uniqueness of the solution of the relaxation procedure \eqref{eq:System_Relaxation_step}. Unfortunately, the theorem does not give bounds on $\dt$ for the existence of the solution, so that computations may be rejected due to $\dt$ being too large.

	\begin{remark}[Issue with positivity]\label{rem:issue_pos_relax}
		The main issue of positivity-preservation with the above relaxation algorithm is that the update
		\[\b u^n_\gamma=\b u^n+\gamma(\b u^{n+1}-\b u^n)= \gamma\b u^{n+1}+(1-\gamma)\b u^n\]
		is not necessarily positivity-preserving for $\gamma>1$, even if the baseline method is positive.
	\end{remark}
	Nevertheless, to overcome this issue, we propose to use \textit{unconditionally positive}\footnote{That is, $\b u^n>\b 0$ component-wise implies $\b u^{n+1}>\b 0$ for all $\dt>0$.} time integrators. In the upcoming sections, we introduce the methods of interest for this work, all of which may be recast as so-called non-standard additive Runge--Kutta schemes.
	\subsection{Non-standard Additive Runge--Kutta Methods}\label{sec:NSARK}
	Non-standard additive Runge--Kutta methods (NSARK) methods are applied to an IVP \eqref{eq:ivp}, where the right-hand side is split into a sum, that is
	\begin{align}\label{eq:ivp_NSARK}
		\b u'(t)  =\b f(\b u(t))= \sum_{\substack{\nu=1}}^N  \fnu(\b u(t)), \quad  \b u(t_0) & = \b u^0\in \R^d.
	\end{align}
	Already for traditional additive Runge--Kutta (ARK) methods, including Implicit-Explicit (IMEX) Runge--Kutta (RK) methods \cite{Crouzeix1980,ARS1997}, the main idea is to apply very
	different RK schemes determined by  $\b A^{[\nu]}=(a_{ij}^{[\nu]})_{i,j=1,\dotsc,s}$, $\b b^{[\nu]}=(b_1^{[\nu]},\dotsc, b_s^{[\nu]})$, $\b c^{[\nu]}=(c_1^{[\nu]},\dotsc, c_s^{[\nu]})^T$ to the different addends $\fnu$. For internal consistency, we require that the different RK schemes actually do not differ in the abscissa, \ie
	\begin{equation}
		c_i=c_i^{[\nu]}=\sum_{j=1}^sa_{ij}^{[\nu]}\label{eq:sum_aij_ARK}
	\end{equation}
	for $i=1,\dotsc,s$ and $\nu=1,\dotsc, N$, see \cite{SG2015}. However, for autonomous IVPs \eqref{eq:ivp_NSARK}, this has no effect on the resulting ARK method, which in this case reads
	\begin{equation}\label{eq:ark}
		\begin{aligned}
			\bui & = \b u^n + \dt \sum_{j=1}^s  \sum_{\substack{\nu=1}}^N a^{[\nu]}_{ij}\fnu(\buj), \quad i=1,\dotsc,s,\\
			\b u^{n+1} & = \b u^n + \dt \sum_{j=1}^s \sum_{\substack{\nu=1}}^N  b^{[\nu]}_j\fnu(\buj),
		\end{aligned}
	\end{equation}  and the corresponding extended Butcher tableau is given by
	\[\arraycolsep=1.4pt\def\arraystretch{1.5}
	\begin{array}{c|c|c|c|c}
		\b c &	\b A^{[1]}        &    \b A^{[2]} & \cdots & \b A^{[N]}\\ \hline
		&	\b b^{[1]}        &    \b b^{[2]} & \cdots & \b b^{[N]}
	\end{array}
	\]
	with $\b c=(c_1,\dotsc,c_s)^T$.

	NSARK methods now differ from ARK schemes \eqref{eq:ark} in that their extended Butcher tableau is allowed to also depend on the step size and the solution. In particular, NSARK methods applied to \eqref{eq:ivp_NSARK} are of the form
	\begin{equation}\label{eq:nsark}
		\begin{aligned}
			\bui & = \b u^n + \dt \sum_{j=1}^s  \sum_{\substack{\nu=1}}^N a^{[\nu]}_{ij}(\b U^n,\tn,\dt)  \fnu(\buj), \quad i=1,\dotsc,s,\\
			\b u^{n+1} & = \b u^n + \dt \sum_{j=1}^s \sum_{\substack{\nu=1}}^N b^{[\nu]}_j(\b U^n,\tn,\dt) \fnu(\buj),
		\end{aligned}
	\end{equation}
	where $\b U^n=(\b u^n \dashline\b u^{(1)}\dashline\dotsc\dashline\b u^{(s)}\dashline \b u^{n+1})\in \R^{d\times s+2}$.

	In the case of gBBKS \cite{AKM2020}, Geometric Conservative (GeCo) \cite{MCD2020}, both of which may be interpreted as NSRK schemes, as well as modified Patankar--Runge--Kutta (MPRK) \cite{KM18,KM18Order3} methods, the same RK scheme is used for the treatment of the different addends in \eqref{eq:ivp_NSARK} and only the solution-dependent terms vary. For MP strong-stability-preserving RK (MPSSPRK) schemes, the situation is different, see Section~\ref{sec:MPSSPRK}. In this work we focus on modified Patankar (MP) schemes in the entropy-conservative case and leave gBBKS and GeCo methods for future works.
	\subsubsection{Production-Destruction-Rest Systems}
	The application of modified Patankar (MP) schemes is restricted to production-destruction-rest (PDRS) systems
	\[u_k'(t)=r^P_k(\b u(t))-r^D_k(\b u(t)) + \sum_{\nu=1}^{d}(p_{k\nu}(\b u(t))-d_{k\nu}(\b u(t))),\quad k=1,\dotsc,d\]
	with $p_{k\nu}=d_{\nu k}$ and $r^P_k,r^D_k,p_{k\nu},d_{k\nu}\geq 0$ on $\R^d_{>0}$. We note that this is only a formal restriction since every autonomous system with real-valued right-hand sides can be rewritten as such a PDRS \cite{IR2023}. Now, one can recover the function $\b f^{[\nu]}$ in \eqref{eq:ivp_NSARK} and specify the solution-dependent Butcher coefficients. Indeed, according to \cite[Remark~2.25]{IzginThesis}, a PDRS can be written in terms of \eqref{eq:ivp_NSARK} by using the convention $p_{kk}=d_{kk}=0$ and choosing $N=d+1$ as well as
	\begin{equation}\label{eq:RHS_MPRK}
		\begin{aligned}
			\b f^{[N]}(\b u(t))&=(r_1^P(\b u(t)),\dotsc,r_d^P(\b u(t)))^T\\
			f^{[\nu]}_k(\b u(t))&=\begin{cases}p_{k\nu}(\b u(t)), &k\neq \nu, \\-\left(r_k^D(\b u(t)) + \sum_{\mu=1}^d d_{k\mu}(\b u(t))\right), &k=\nu\end{cases}
		\end{aligned}
	\end{equation}
	for $k,\nu =1,\dotsc,d$.
	\subsubsection{Modified Patankar--Runge--Kutta Schemes}
	With \eqref{eq:RHS_MPRK}, every MPRK scheme \cite{BDM2003,KM18,KM18Order3} that is based on a single explicit RK method with a non-negative Butcher array can be expressed in terms of an NSARK scheme using

	\begin{equation}\label{eq:NSweights}
		\begin{aligned}
			a^{[\nu]}_{ij}(\b U^n,\tn,\dt)&=a_{ij}\gamma_\nu^{[i]}(\b U^n,\tn,\dt),\\
			b^{[\nu]}_j(\b U^n,\tn,\dt)&=b_j\delta_\nu(\b U^n,\tn,\dt),
		\end{aligned}
	\end{equation}
	where
	\begin{equation}\label{eq:NSW_MPRK}
		\begin{aligned}
			\gamma_\nu^{[i]}(\b U^n,\tn,\dt)&=\begin{cases}
				\frac{\ui_\nu}{\pi_{\nu}^{(i)}(\b u^n,\b u^{(1)},\dotsc,\b u^{(i-1)}) }, & \nu < N \\1, & \nu = N\end{cases} \,\, \text{ and }\\
			\delta_\nu(\b U^n,\tn,\dt)&=\begin{cases} \frac{u^{n+1}_\nu}{\sigma_\nu(\b u^n,\b u^{(1)},\dotsc,\b u^{(s)})}, & \nu < N \\1, & \nu = N\end{cases}
		\end{aligned}
	\end{equation}
	are the so-called \emph{non-standard weights} (NSWs). Here,  $\pi_{\nu}^{(i)}$ and $\sigma_\nu$ denote the so-called \textit{Patankar-weight denominators} (PWDs) and can be chosen for the particular MPRK method to ensure stability and accuracy, see \cite{KM18, IzginThesis} for more insights. If the Butcher array contains negative entries, more care is needed when defining the MPRK method, see e.g. \cite{MPDeC}.
	\begin{example}[Second-order Family]
		The second-order family of MPRK schemes, denoted by MPRK22($\alpha$), is given by
		\begin{align}
			u_k^{(1)} =&\, u_k^n,\nonumber\\
			u_k^{(2)} =&\, u_k^n + \alpha\dt \left(r^P_k(\b u^{(1)})  + \sum_{\nu=1}^dp_{k\nu}(\b u^{(1)})\frac{u_\nu^{(2)}}{u_\nu^n}- \left( r^D_k(\b u^{(1)}) + \sum_{\nu=1}^dd_{k\nu}(\b u^{(1)})\right)\frac{u_k^{(2)}}{u_k^n}\right),\nonumber\\
			u_k^{n+1} =&\, u_k^n + \dt\sum_{j=1}^2b_j\left(r_k^P(\b u^{(j)}) + \sum_{\nu=1}^dp_{k\nu}(\b u^{(j)}) \frac{u_\nu^{n+1}}{(u_\nu^{(2)})^{\frac{1}{\alpha}}(u_\nu^n)^{1-\frac{1}{\alpha}}}\right.\nonumber\\
			&\hphantom{u_k^n + \dt\sum_{j=1}^2b_j}
			-\left.\left( r^D_k(\b u^{(j)}) + \sum_{\nu=1}^dd_{k\nu}(\b u^{(j)})\right)\frac{u_k^{n+1}}{(u_k^{(2)})^{\frac{1}{\alpha}}(u_k^n)^{1-\frac{1}{\alpha}}}\right)\nonumber
		\end{align}
		with $k=1,\dotsc,d$, $\alpha\geq \frac12$ and $b_2=\tfrac{1}{2\alpha}$ as well as $b_1=1-b_2$. In terms of the previous notation, we are using the Butcher array
		\begin{equation*}
			\begin{aligned}
				\def\arraystretch{1.2}
				\begin{array}{c|cc}
					0 &  & \\
					\alpha & \alpha  &\\
					\hline
					& 1-\frac1{2\alpha} &\frac1{2\alpha}
			\end{array}	\end{aligned}
		\end{equation*}
		and the PWDs
		\[\pi_k^{(2)} =u_k^n,\quad \sigma_k=(u_k^{(2)})^{\frac{1}{\alpha}}(u_k^n)^{1-\frac{1}{\alpha}}.\]
		In this work, we will focus on $\alpha=1$ as suggested by \cite{IssuesMPRK}.
	\end{example}
	\begin{example}[Third-order Family]\label{example:MPRK43I}
		There are two third-order families of MPRK schemes, see \cite{KM18Order3}. One of them is based on the Butcher array
		\begin{equation}\label{array:MPRK43alphabeta}
			\begin{aligned}
				\def\arraystretch{1.4}
				\begin{array}{c|ccc}
					0 &  & & \\
					\alpha & \alpha & & \\
					\beta & \frac{3\alpha\beta (1-\alpha)-\beta^2}{\alpha(2-3\alpha)}& \frac{\beta (\beta-\alpha)}{\alpha(2-3\alpha)}& \\
					\hline
					& 1+\frac{2-3(\alpha+\beta)}{6 \alpha \beta } &\frac{3 \beta-2}{6\alpha (\beta-\alpha)} & \frac{2-3\alpha}{6\beta(\beta-\alpha)}
				\end{array}
			\end{aligned}
		\end{equation}
		see \cite{KM18Order3} for more details on the domain of $\alpha$ and $\beta$. The PWDs are given by
		\begin{equation}\label{eq:PWDsMPRK43}
			\begin{aligned}
				\pi_\nu^{(2)} =\, &u_\nu^n,\qquad
				\pi_\nu^{(3)}  =\, (u_\nu^{(2)})^{\frac{1}{p}}(u_\nu^n)^{1-\frac{1}{p}},\\
				\sigma_k=\, &u_k^n+ \dt \sum_{j=1}^2\beta_j\left(r_k^P(\buj) + \sum_{\nu=1}^d  p_{k\nu}(\buj)\frac{\sigma_\nu}{(u_\nu^{(2)})^{\frac{1}{a_{21}}}(u_\nu^n)^{1-\frac{1}{a_{21}}}}\right. \\
				& \hspace{2cm}- \left.\left(r_k^D(\buj) + \sum_{\nu=1}^d d_{k\nu}(\buj)\right)\frac{\sigma_k}{(u_k^{(2)})^{\frac{1}{a_{21}}}(u_k^n)^{1-\frac{1}{a_{21}}}}\right)
			\end{aligned}
		\end{equation}
		for $\nu,k=1,\dotsc,d$, where $\beta_1=1-\beta_2$,  $\beta_2=\frac{1}{2a_{21}}$,
		and $p=3 a_{21}\left(a_{31}+a_{32} \right)b_3$.
		Note that $\bsigma$ requires the solution of another linear system, which is why this family is denoted by MPRK43I($\alpha,\beta$). In this work we focus on $\alpha=0.5$ and $\beta=0.75$.

	\end{example}

	In any case we point out that the schemes are implicit due to the numerators in \eqref{eq:NSW_MPRK}. Indeed, they are linearly implicit as the PWDs $\pi_{\nu}^{(i)}$ and $\sigma_\nu$ are required to be independent of the numerator \cite{BDM2003,IzginThesis}. Consequently, an MPRK scheme can be written in matrix-vector notation as follows.
	\begin{equation}\label{eq:MPRK_PDRS_matrix}
		\begin{aligned}
			\b M^{(i)}\bui&= \b u^n+\dt \sum_{j=1}^{i-1}a_{ij} \b r^P(\b u^{(j)}),\quad i=1,\dotsc,s, \\
			\b M\b u^{n+1}&= \b u^n + \dt\sum_{j=1}^{s}b_{j} \b r^P(\b u^{(j)}),
		\end{aligned}
	\end{equation}
	where $\b r^P=(r_1^P,\dotsc, r_d^P)^T$ and $\b M^{(i)}=(m^{(i)}_{k\nu})_{1\leq k,\nu\leq d}$ with
	\begin{equation} \label{eq:M_stage}
		\begin{aligned}
			m^{(i)}_{kk}&=1+ \dt \sum_{j=1}^{i-1}a_{ij}\left(r_k^D(\b u^{(j)}) +\sum_{\nu=1}^dd_{k\nu}(\b u^{(j)}) \right)\frac{1}{\pi_k^{(i)}}, \\
			m^{(i)}_{k\nu}&= -\dt \sum_{j=1}^{i-1}a_{ij}p_{k\nu}(\b u^{(j)})\frac{1}{\pi_\nu^{(i)}}, \quad k\neq \nu
		\end{aligned}
	\end{equation}
	as well as, using $\b M=(m_{k\nu})_{1\leq k,\nu\leq d}$,
	\begin{equation}\label{eq:M_update}
		\begin{aligned}
			m_{kk}&= 1+\dt \sum_{j=1}^sb_j\left(r_k^D(\b u^{(j)}) +\sum_{\nu=1}^dd_{k\nu}(\b u^{(j)}) \right)\frac{1}{\sigma_k}, \\
			m_{k\nu}&= -\dt \sum_{j=1}^sb_jp_{k\nu}(\b u^{(j)})\frac{1}{\sigma_\nu}, \quad k\neq \nu.
		\end{aligned}
	\end{equation}

	\subsubsection{MP Strong-Stability-Preserving-Runge--Kutta Schemes}\label{sec:MPSSPRK}
	Although there exist second- and third-order MPSSPRK schemes, see \cite{SSPMPRK2,SSPMPRK3}, we focus for simplicity on the second-order method and the conservative PDS case. The generalization to non-conservative PDRS is straightforward but complicates the formulae. Also, the consideration of third-order MPSSPRK schemes will be left out for future work.
	The two-parameter family of second-order MPSSPRK schemes from \cite{SSPMPRK2} is given by
	\begin{equation}\label{eq:SSPMPRK2}
		\begin{aligned}
			\b u^{(1)}=&\b u^n,\\
			u_i^{(2)}=& u_i^n+\beta \dt\left(\sum_{j=1}^dp_{ij}(\b u^n)\frac{u_j^{(2)}}{u_j^n}- \sum_{j=1}^dd_{ij}(\b u^n)\frac{u_i^{(2)}}{u_i^n}\right),\\
			u_i^{n+1}={}& (1-\alpha)u_i^n+\alpha u_i^{(2)}+\dt\Biggl(\sum_{j=1}^d\left(\beta_{20}p_{ij}(\b u^n)+\beta_{21}p_{ij}(\b u^{(2)})\right)\frac{u_j^{n+1}}{(u_j^n)^{1-s}(u_j^{(2)})^s} \\& - \sum_{j=1}^d\left(\beta_{20}d_{ij}(\b u^n)+\beta_{21}d_{ij}(\b u^{(2)})\right)\frac{u_i^{n+1}}{(u_i^n)^{1-s}(u_i^{(2)})^s}\Biggr),
		\end{aligned}
	\end{equation}
	where $\beta_{20}=1-\frac{1}{2\beta}-\alpha\beta$, $\beta_{21}=\frac{1}{2\beta}$ and $s=\frac{1-\alpha\beta+\alpha\beta^2}{\beta(1-\alpha\beta)}$. There, the free parameters $\alpha$ and $\beta$ are subject to
	\begin{equation}
		0\leq \alpha\leq 1,\quad \beta>0,\quad \alpha\beta+\frac{1}{2\beta}\leq 1.\label{eq:alphabeta_conditions}
	\end{equation}
	We refer to the above scheme as MPSSPRK2($\alpha,\beta$). For numerical experiments we use $\alpha=\frac12$ and $\beta=1$ \cite{SSPMPRK2}.

	Substituting the second stage into the update, we can collect production and destruction terms. Hence, in the notation of \eqref{eq:RHS_MPRK}, the solution-dependent coefficients for the conservative PDS case are
	\begin{equation}\label{eq:NSW_MPSSPRK}
		\begin{aligned}
			a_{21}^{[\nu]}(\b U^n,t_n,\dt)&=\beta\frac{u_\nu^{(2)}}{u^n_\nu}\\
			b_1^{[\nu]}(\b U^n,t_n,\dt)&=\alpha\beta\frac{u_\nu^{(2)}}{u^n_\nu} + \beta_{20}\frac{u_\nu^{n+1}}{\sigma_\nu},\quad b_2^{[\nu]}(\b U^n,t_n,\dt)=\beta_{21}\frac{u_\nu^{n+1}}{\sigma_\nu},
		\end{aligned}
	\end{equation}
	where $\sigma_\nu = (u_\nu^n)^{1-s}(u_\nu^{(2)})^s$.
	\section{Positivity-Preserving Relaxation Technique}
	In what follows, we adapt the classical relaxation algorithm from Section~\ref{sec:relax} such that it becomes positivity-preserving.
	\subsection{Entropy Dissipative Case}\label{sec:dissipative}
	First of all, we present a suitable estimate $\etanew$ for a general NSARK scheme.
	In order to minimize the computational effort, we propose to re-use the computed stage values of the NSARK scheme satisfying \eqref{eq:NSweights}, that is
	\begin{equation}\label{eq:nsark_and_eta}
		\begin{aligned}
			\bui & = \b u^n + \dt \sum_{j=1}^s  \sum_{\substack{\nu=1}}^N a_{ij}\gamma^{[i]}_{\nu}(\b U^n,\tn,\dt)  \fnu(\buj), \quad i=1,\dotsc,s,\\
			\b u^{n+1} & = \b u^n + \dt \sum_{j=1}^s \sum_{\substack{\nu=1}}^N b_j\delta_{\nu}(\b U^n,\tn,\dt) \fnu(\buj),\\
			\etanew &= \eta(\b u^n)+ \dt \sum_{j=1}^s b_j (\eta'\b f)(\buj),
		\end{aligned}
	\end{equation}
	which can be interpreted as computing the numerical approximation of the augmented system
	\[\frac{\dd}{\dd t}\cVec{\b u(t)\\\eta(\b u(t))} =\sum_{\substack{\nu=1}}^N\underbrace{\cVec{  \fnu(\b u(t))\\0}}_{\eqqcolon \hat{\b f}^{[\nu]}(\b u(t))} + \underbrace{\cVec{\b 0\\ (\eta'\b f)(\b u(t))}}_{\eqqcolon\hat{\b f}^{[N+1]}(\b u(t))}\]
	using an NSARK method with the extended Butcher tableau
	\begin{equation}\label{eq:NSARK_augmented_eta}
		\arraycolsep=1.4pt\def\arraystretch{1.5}
		\begin{array}{c|c|c|c|c|c}
			\b c &\bm\Gamma_1(\b U^n,t_n,\dt) \b A         &    \bm\Gamma_2(\b U^n,t_n,\dt) \b A  & \cdots & \bm\Gamma_N(\b U^n,t_n,\dt) \b A &	\b A\\ \hline
			&	\delta_1(\b U^n,t_n,\dt) \b b         &    \delta_2(\b U^n,t_n,\dt) \b b  & \cdots & \delta_N(\b U^n,t_n,\dt) \b b 	&	\b b
		\end{array},
	\end{equation}
	where $\bm\Gamma_\nu\coloneqq\diag(\gamma^{[1]}_\nu,\dotsc, \gamma^{[s]}_\nu)$. Assuming that the two corresponding base methods described by the Butcher tableaux
	\[\arraycolsep=1.4pt\def\arraystretch{1.5}
	\begin{array}{c|c|c|c|c}
		\b c &	\bm\Gamma_1(\b U^n,t_n,\dt) \b A        &    \bm \Gamma_2(\b U^n,t_n,\dt)\b A & \cdots & \bm \Gamma_N(\b U^n,t_n,\dt) \b A\\ \hline
		&	\delta_1(\b U^n,t_n,\dt) \b b        &    	\delta_2(\b U^n,t_n,\dt)\b b  & \cdots & 	\delta_N(\b U^n,t_n,\dt)\b b
	\end{array}\quad  \text{ and } \quad \arraycolsep=1.4pt\def\arraystretch{1.5}
	\begin{array}{c|c}
		\b c &		\b A      \\ \hline
		&	\b b
	\end{array}
	\]
	both are of $p$-th order for some $p\in\{2,3,4\}$, it can be seen from \cite[Theorem~18,Lemma~25,Lemma~26]{NSARK} that the overall scheme \eqref{eq:nsark_and_eta} is of order $p$, since the respective order conditions are decoupled with respect to the columns of the tableau \eqref{eq:NSARK_augmented_eta}.
	\begin{remark}
		The NSWs from MPSSPRK schemes, see \eqref{eq:NSW_MPSSPRK}, also satisfy \eqref{eq:NSweights} after multiplying and dividing by the respective Butcher coefficient.
	\end{remark}
	As a result, we indeed obtain that $\etanew=\eta(\unew)+\O(\dt^{p+1})$, and additionally,
	\begin{equation}\label{eq:etanew<=etaold}
		\etanew\leq \eta(\b u^n)
	\end{equation} whenever $b_j\geq 0$ for $j=1,\dotsc,s$. If $b_j<0$ for some $j=1,\dotsc,s$, one can still use Gau\ss\, quadrature, as suggested in \cite[Page 866]{ranocha2020general} together with the unconditionally positive dense output formulae derived in \cite{izgin2024} to obtain the approximations needed for the quadrature formula.

	In view of Remark~\ref{rem:issue_pos_relax}, the relaxation technique is in danger of not being positivity-preserving for $\gamma>1$. The following corollary gives a work around for dissipative problems as we will discuss in the upcoming Remark~\ref{rem:pos_relax_diss}.
	\begin{corollary}[{\cite[Pages 882-883]{ranocha2020general}}]\label{cor:relax_diss}
		If $\eta$ is convex with $\eta''(\uold)(\b f(\uold),\b f(\uold))\neq 0$
		then $r$ from \eqref{eq:r}
		is convex and satisfies $r(0)=0$, $r'(0)<0$ and $r'(\gamma)>0$ for all $\gamma\geq 1$ and $\dt>0$ small enough.
	\end{corollary}
	\begin{remark}[Positivity-preserving relaxation for convex $\eta$]\label{rem:pos_relax_diss}
		Suppose that $\gamma^*>1$ is the solution to \eqref{eq:System_Relaxation_step}, \ie $r(\gamma^*)=0$, so that the positivity of the relaxation update $\ugamman$ is not guaranteed any longer. Because of Corollary~\ref{cor:relax_diss} we know that $r(\gamma)\leq r(\gamma^*)=0$ for all $\gamma\in [1,\gamma^*]$. In particular $r(1)\leq 0$ follows, i.\,e., for $\gamma=1$ we obtain from \eqref{eq:r} the relation \[\eta(\ugamman)=\eta(\unew)\leq \etanew.\]
		This means, that due to \eqref{eq:etanew<=etaold} only more dissipation will be introduced by using \[\gamma=\min\{\gamma^*,1\}\in(0,1],\] where $\gamma^*$ is the solution to \eqref{eq:System_Relaxation_step}. But with that choice, $\ugamman$ is again a convex combination of positive data, and hence, positivity preservation is guaranteed for dissipative problems with a convex $\eta$.
	\end{remark}

	\subsection{Entropy-Conservative Case}\label{sec:cons_entropy}
	As $\ugamman$ in \eqref{eq:System_Relaxation_step} is not guaranteed to be positivity-preserving for $\gamma>1$, see Remark~\ref{rem:issue_pos_relax}, we propose to replace the update formula by a positivity-preserving variant. To indicate this difference in our notation we will write $\bunplusgamma$ for a positivity-preserving approximation rather than $\ugamman$.
	\subsubsection{Explicit Positivity-Preserving Procedure}
	If we are only interested in preserving positivity, a single nonlinear invariant but no further linear invariants, we may apply a Patankar--Runge--Kutta method to guarantee the positivity of the update and combine it with the geometric mean
	\begin{equation}\label{eq:geom_mean}
		\bunplusgamma=(\b u^{n+1})^{\gamma}(\b u^n)^{1-\gamma}.
	\end{equation}
	In logarithmic variables, this reduces to
	\[	\ln(\bunplusgamma)=\ln(\b u^n) +\gamma (\ln(\b u^{n+1})-\ln(\b u^n)), \]
	where we can find a solution to the relaxation problem in logarithmic variables according to the classical theory. Now, if $\eta$ is convex and non-decreasing in each argument the composition $\eta\circ \exp$ is also convex \cite[Section~2.3.4]{boyd2004convex}, and hence, also
	\[\eta(\bunplusgamma)=\etaold\]
	possesses a positive solution $\gamma=1+\O(\dt^{p-1})$.

	\subsubsection{Implicit Positivity-Preserving Procedure for Conservative PDS}
	One possible candidate for computing $\bunplusgamma$ is to use dense output formulae recently developed in \cite{izgin2024}, which we briefly recall in the upcoming section.
	\paragraph{Positivity-Preserving Dense Output}
	We first focus on Runge--Kutta methods and the MP variant, but the ideas can be carried out for MPSSPRK schemes in a straightforward manner as we will see.

	The main idea is to replace $b_j\in \R$ by a function $\bar b_j\colon [0,1]\to \R$ such that
	\begin{equation*}\label{eq:RKdense}
		u^{n+\gamma}_k = u^n_k + \dt \sum_{j=1}^s \bar b_j(\gamma) \sum_{\nu=1}^d \left(r_k^P(\b u^{(j)}) - r_k^D(\b u^{(j)}) + p_{k\nu}(\b u^{(j)}) - d_{k\nu}(\b u^{(j)})  \right)
	\end{equation*}
	approximates $u_k(t^n+\gamma\dt)$. We impose \[\bar b_j(0)=0 \qta \bar b_j(1)=b_j\] to recover
	\begin{equation*}\label{eq:inner_consistency}
		\b u^{n+\gamma}=\begin{cases}
			\b u^n,& \gamma = 0,\\
			\b u^{n+1}, & \gamma = 1.
		\end{cases}
	\end{equation*}
	\begin{example}[Second-order dense output for MPRK22($\alpha$)]\label{example:DO_MPRK2}
		Using $\bar b_j(\gamma)=\gamma b_j$ and \begin{equation*}\label{eq:DO_ansatz1}
			\unplusgamma_k=u_k^n + \dt\sum_{j=1}^s \bar b_j(\gamma)\left(r_k^P(\b u^{(j)}) + \sum_{\nu=1}^dp_{k\nu}(\b u^{(j)}) \frac{u_\nu^{n+1}}{\sigma_\nu}-\left( r^D_k(\b u^{(j)}) + \sum_{\nu=1}^dd_{k\nu}(\b u^{(j)})\right)\frac{u_k^{n+1}}{\sigma_k}\right)
		\end{equation*} yields a positivity-preserving dense output. Indeed, we find  $\b u^{n+\gamma}=(1-\gamma)\b u^n+\gamma \b u^{n+1}$ in this case, which coincides with the relaxation update. For $\gamma>1$ this is not necessarily positivity-preserving.  For our purposes, we want to ensure positivity even for $\gamma >1$, which can be done using
		\begin{equation}\label{eq:DO_ansatz2}
			\unplusgamma_k=u_k^n + \dt\sum_{j=1}^s \bar b_j(\gamma)\left(r_k^P(\b u^{(j)}) + \sum_{\nu=1}^dp_{k\nu}(\b u^{(j)}) \frac{u_\nu^{n+\gamma}}{\bar\sigma_\nu(\gamma)}-\left( r^D_k(\b u^{(j)}) + \sum_{\nu=1}^dd_{k\nu}(\b u^{(j)})\right)\frac{u_k^{n+\gamma}}{\bar\sigma_k(\gamma)}\right)
		\end{equation}
		together with
		\begin{equation}\label{eq:sigma_MPRK22}
			\bar \sigma_k(\gamma) = \sigma_k=(u_k^{(2)})^{\frac{1}{\alpha}}(u_k^n)^{1-\frac{1}{\alpha}}
		\end{equation}
		or
		\begin{equation}\label{eq:sigma_denseMPRK2}
			\bar \sigma_k(\gamma) =(u_k^{(2)})^{\frac{\gamma}{\alpha}}(u_k^n)^{1-\frac{\gamma}{\alpha}}.
		\end{equation}
	\end{example}
	\begin{example}[Higher order positive dense output for MPRK]\label{example:highorder_dense}
		In general, one may use $\bar b_j(\gamma)$ from the classical dense output formula paired with the update \eqref{eq:DO_ansatz2}. Then, the only quantity to define is $\bar\bsigma(\gamma)$. According to \cite{izgin2024}, it is sufficient to use a lower order dense output MPRK scheme for the computation of $\bar\bsigma(\gamma)$. For instance, we note that third-order MPRK schemes are equipped with \[\bar b_1(\gamma)=\gamma -(1-b_1)\gamma^2,\quad \bar b_j(\gamma)=\gamma^2b_j, \quad j=2,\dotsc,s\]
		for the dense output. However, we will discuss a different approach in this work, and thus, omit to also recall $\bar\bsigma$ from \cite{izgin2024}.
	\end{example}
	\begin{example}[Second-order dense output for MPSSPRK]\label{example:MPSSPRK2_dense}
		Let us incorporate the $\gamma$-dependency in \eqref{eq:NSW_MPSSPRK}, which gives
		\begin{equation*}
			\begin{aligned}
				b_1^{[\nu]}(\b U^n,t_n,\dt,\gamma)&=\gamma\left(\alpha\beta\frac{u_\nu^{(2)}}{u^n_\nu} + \beta_{20}\frac{u_\nu^{n+\gamma}}{\bar\sigma_\nu(\gamma)}\right),\quad b_2^{[\nu]}(\b U^n,t_n,\dt,\gamma)=\gamma\beta_{21}\frac{u_\nu^{n+\gamma}}{\bar\sigma_\nu(\gamma)},
			\end{aligned}
		\end{equation*}
		where we restrict to the choice $\bar\sigma_\nu(\gamma) = (u_\nu^n)^{1-\gamma s}(u_\nu^{(2)})^{\gamma s}$. Then
		\[	\b u^{n+\gamma}  = \b u^n + \dt \sum_{j=1}^s \sum_{\substack{\nu=1}}^d b_j(\b U^n,\tn,\dt,\gamma) \fnu(\buj).\]
	\end{example}
	\begin{remark}[Use of dense output for relaxation]\label{remark:issue_denseoutput_relaxation}
		As we will show, we can use the dense output from Example~\ref{example:DO_MPRK2} for the relaxation algorithm. However, proving the existence of a solution to the relaxation equation for \eqref{eq:sigma_MPRK22} is more complex than for \eqref{eq:sigma_denseMPRK2}, which is due to the respective truncation errors.
		Moreover, as illustrated in Example~\ref{example:highorder_dense}, the bootstrapping for higher order positive dense output involves higher degree polynomials for $\bar b_j$, which may not be positive for $\gamma>1$. This is crucial since the solvability for the linear systems \eqref{eq:MPRK_PDRS_matrix} relies on positive Butcher coefficients. While we could implement a trick \cite{MPDeC} to overcome this issue, we rather focus on a different bootstrapping approach to keep the overall algorithm simple.
	\end{remark}
	\paragraph{Preparatory Results for MPRK22($\alpha$)}
	We proceed to develop a relaxation technique using \eqref{eq:DO_ansatz2}-\eqref{eq:sigma_MPRK22}, which is more complicated than using \eqref{eq:sigma_denseMPRK2} but on the other hand motivates us to derive more general results.
	To that end, we note that the scheme with \eqref{eq:DO_ansatz2}-\eqref{eq:sigma_MPRK22} can be written in matrix-vector notation as
	\begin{equation}\label{eq:MPRK22alpha_DO}
		\begin{aligned}
			\b u^{(1)}&=\b u^n\\
			\b M^{(2)}(\b u^n)\b u^{(2)}&=\b u^n + \alpha\dt \b r^P(\b u^n)\\
			\b M_\gamma(\b u^n)	\bunplusgamma &= \b u^n + \gamma\dt\sum_{j=1}^s b_j\b r^P(\b u^{(j)}), \quad \gamma >0,
		\end{aligned}
	\end{equation}
	where  $\b M^{(2)}$ can be obtained from \eqref{eq:M_stage}
	and
	\begin{equation}
		\b M_\gamma= \gamma(\b M-\b I) +\b I\label{eq:Mgamma}
	\end{equation}
	with $\b M$ from \eqref{eq:M_update}.

	Finally, the relaxation step \eqref{eq:System_Relaxation_step} for entropy-conservative problems is now updated to
	\begin{equation}\label{eq:System_Relaxation_MPRK_step}
		\cVec{\tgamman\\
			\b M_\gamma(\uold)	\bunplusgamma\\
			\eta(\bunplusgamma)} = \Vec{\told\\ \uold\\ \etaold}+ \gamma \Vec{\tnew -\told\\ \dt\sum_{j=1}^2 b_j\b r^P(\b u^{(j)})\\ 0} ,
	\end{equation}
	resulting in a coupled linear-nonlinear system for the simultaneous computation of $\gamma$ and $\bunplusgamma$. Note that if such a $\gamma>0$ exists, the relaxation method for MPRK22($\alpha$) naturally is of the correct order for all $\gamma$ as we are using an appropriate dense output formula.

	Since we allow for a truncation error of $\O(\dt^3)$ for the second-order MPRK22$(\alpha$) scheme, it is beneficial to prove the following
	\begin{lemma}\label{lem:dgamma} If $\bar b_j(\gamma)=\gamma b_j$ and $\bar\sigma_k(\gamma)=\sigma_k$, then the MPRK22($\alpha$) scheme \eqref{eq:DO_ansatz2} satisfies \begin{equation}
			\bunplusgamma = \b u^{n+1} + (\gamma-1)\dt \b d^n_\gamma +\O(\dt^3),
		\end{equation}
		where
		\begin{equation}\label{eq:dgamma}
			\begin{aligned}
				d^n_{\gamma,k}&=\frac{u_k^{n+1}-u_k^n}{\dt} +\dt\gamma\sum_{j=1}^sb_j\left(\sum_{\nu=1}^dp_{k\nu}(\buj)\frac{f_\nu(\b u^n)}{u_\nu^n}-\left(r_k^D(\buj)+\sum_{\nu=1}^dd_{k\nu}(\buj)\right)\frac{f_k(\b u^n)}{u_k^n} \right)\\&=f_k(\b u^n) + \O(\dt).
			\end{aligned}
		\end{equation}
	\end{lemma}
	\begin{proof}
		Utilizing \cite[Lemma~2, Lemma~3]{KM2019}, we observe
		\begin{equation}\label{eq:u:sigma}
			\frac{u_\nu^{n+\gamma}}{\bar \sigma_\nu(\gamma)}=\frac{u_\nu^{n+\gamma}}{(u_\nu^{(2)})^{\frac{1}{\alpha}}(u_\nu^n)^{1-\frac{1}{\alpha}}} = 1+(\gamma-1)\dt \frac{f_\nu(\b u^n)}{u_\nu^n} +\O(\dt^2)=\frac{u_\nu^{n+1}}{\sigma_\nu}+(\gamma-1)\dt \frac{f_\nu(\b u^n)}{u_\nu^n} +\O(\dt^2)
		\end{equation}
		as $\frac{u_\nu^{n+1}}{\sigma_\nu}=1+\O(\dt^2)$ \cite{NSARK}.
		Substituting this into \eqref{eq:DO_ansatz2} we receive
		\begin{align*}
			\unplusgamma_k&=u_k^n + \gamma\dt\sum_{j=1}^s b_j\left(r_k^P(\b u^{(j)}) + \sum_{\nu=1}^dp_{k\nu}(\b u^{(j)}) \frac{u_\nu^{n+\gamma}}{\bar\sigma_\nu(\gamma)}-\left( r^D_k(\b u^{(j)}) + \sum_{\nu=1}^dd_{k\nu}(\b u^{(j)})\right)\frac{u_k^{n+\gamma}}{\bar\sigma_k(\gamma)}\right)\\
			&= u_k^n+\gamma(u_k^{n+1}-u_k^n) \\
			&+ \gamma(\gamma-1)\dt^2\sum_{j=1}^s b_j\left(\sum_{\nu=1}^dp_{k\nu}(\b u^{(j)}) \frac{f_\nu(\b u^n)}{u_\nu^n}-\left( r^D_k(\b u^{(j)}) + \sum_{\nu=1}^dd_{k\nu}(\b u^{(j)})\right)\frac{f_k(\b u^n)}{u_k^n}\right) +\O(\dt^3)\\
			&=u_k^{n+1} +(\gamma-1)(u_k^{n+1}-u_k^n)\\
			&+ \gamma(\gamma-1)\dt^2\sum_{j=1}^s b_j\left(\sum_{\nu=1}^dp_{k\nu}(\b u^{(j)}) \frac{f_\nu(\b u^n)}{u_\nu^n}-\left( r^D_k(\b u^{(j)}) + \sum_{\nu=1}^dd_{k\nu}(\b u^{(j)})\right)\frac{f_k(\b u^n)}{u_k^n}\right) +\O(\dt^3)\\
			&= u_k^{n+1} +(\gamma-1)\dt	d^n_{\gamma,k} +\O(\dt^3).
		\end{align*}
	\end{proof}
	\begin{remark}[Influence of $\bar \bsigma$ and application to MPSSPRK] \label{rem:dgamma_using_sigma(gamma)}
		In the situation of Lemma~\ref{lem:dgamma}, if we use \eqref{eq:sigma_denseMPRK2} rather than \eqref{eq:sigma_MPRK22}, then instead of \eqref{eq:u:sigma} we obtain
		\[	\frac{u_\nu^{n+\gamma}}{\bar \sigma_\nu(\gamma)}= 1+\O(\dt^2)=\frac{u_\nu^{n+1}}{\sigma_\nu}+\O(\dt^2) \]
		using the same technique, and finally
		\begin{equation}
			\bunplusgamma=\b u^{n+1}+(\gamma-1)\dt \underbrace{\frac{\b u^{n+1}-\b u^n}{\dt}}_{\eqqcolon \b d^n} + \O(\dt^{3}).\label{eq:ugamman_bootstrap}
		\end{equation}
		Here $\b d^n_\gamma=\b d^n$ is independent of $\gamma$. 	We note that the PWDs for MPSSPRK2 are similar to the MPRK case, see Section~\ref{sec:MPSSPRK}. Indeed, one can show that \eqref{eq:ugamman_bootstrap} also holds for the second-order MPSSPRK family.
	\end{remark}
	While the scheme using \eqref{eq:DO_ansatz2}-\eqref{eq:sigma_MPRK22} is the motivation for the general formulation of our main result in the upcoming section, the scheme \eqref{eq:DO_ansatz2},\eqref{eq:sigma_denseMPRK2} will be the basis for the bootstrapping algorithm to obtain higher order, see Section~\ref{sec:bootstrap}.
	\paragraph{Main Result for Entropy Conservation and Positivity Preservation}
	We assume that the method can be written in the form
	\begin{equation*}
		\bunplusgamma = \b u^{n+1} + (\gamma-1)\dt \b d^n_\gamma(\dt) +\O(\dt^{p+1})
	\end{equation*}
	with $p\geq 2$ being the order of the baseline method and some suitable $\b d^n_\gamma(\dt)$ depending on the method. Now, since $\b d$ generally also depends on $\gamma$, we need the preparatory
	\begin{lemma}\label{lem:wgamma}
		Let $h\colon U\times V\to \R$, \begin{equation}
			h(\lambda,\gamma)\coloneqq \lambda-(\gamma-1)\dt \lVert \b d_\gamma^n(\dt)\rVert_2\label{eq:lambda}
		\end{equation} where $U\times V$ is an open neighborhood of $(0,1)$.

		If $\b d_\gamma^n(\dt)$ is a $\mathcal C^1$ map on $V$ w.r.t.\ $\gamma$ with $\lVert \b d_1^n(\dt)\rVert_2\neq 0$ for $\dt$ small enough, then there exist a neighborhood $\w U$ of $0$ and a continuous function $\w \gamma\colon \w U\to \R$ such that $h(\lambda,\w \gamma(\lambda))=0$ for all $\lambda\in \w U$ and $\dt>0$ small enough.
	\end{lemma}
	\begin{proof}
		We have $h(0,1)=0$ and
		\[\partial_\gamma h(\lambda,\gamma)=-\dt\lVert \b d_\gamma^n(\dt)\rVert_2-(\gamma-1)\dt \frac{(\partial_\gamma \b d_\gamma^n(\dt))^T \b d_\gamma^n(\dt)}{\lVert \b d_\gamma^n(\dt)\rVert_2}.\]
		Thus, $\partial_\gamma h(0,1)=-\dt \lVert \b d_1^n(\dt)\rVert_2\neq 0$ for $\dt$ small enough. The assertion then follows from the implicit function theorem.
	\end{proof}
	With that, we are positioned to prove the main theorem for the relaxation technique.
	\begin{theorem}\label{thm:pos_relax_1}
		In the situation of Lemma~\ref{lem:wgamma}, let
		\begin{equation*}
			\bunplusgamma = \b u^{n+1} + (\gamma-1)\dt \b d^n_\gamma(\dt) +\O(\dt^{p+1})
		\end{equation*} with $p\geq 2$ being the order of the baseline method and suppose $\eta\in \mathcal C^2$ with
		\[\eta'(\b u^{n+1})\cdot\frac{\b d^n_{\w \gamma(\dt \mu)}(\dt)}{\lVert \b d^n_{\w \gamma(\dt \mu)}(\dt)\rVert_2}= c(\b u^n,\mu) \dt +\O(\dt^2)\]
		for $\abs \mu$ small enough and \[\lim_{\mu\to 0}c(\b u^n,\mu)\eqqcolon\w c(\b u^n)\neq 0.\]
		Then the equation \[\eta(\bunplusgamma)-\etaold=0\]
		possesses a positive solution $\gamma$. Furthermore, if $\lVert\b d_\gamma^n(\dt)\rVert_2=\O(\dt^q)$, then there exists a unique positive solution satisfying $\gamma=1+\O(\dt^{p-1-q})$. 
	\end{theorem}
	\begin{proof}
		We set \[\b w_\gamma^n(\dt)\coloneqq \frac{\b d_\gamma^n(\dt)}{\lVert \b d_\gamma^n(\dt)\rVert_2}\]
		and follow the proof of \cite[Theorem~2]{CLM2010} by analyzing the function
		\begin{equation}\label{eq:z_def}
			z(\dt,\mu)\coloneqq\dt^{-2}\left(\eta\left(\b u^{n+\w \gamma(\dt \mu)}\right)-\etaold\right),\quad \dt\neq 0.
		\end{equation}
		The idea is to deduce that
		\begin{equation}\label{eq:z}
			z(\dt,\mu)
			=\mu c(\b u^n,\mu) +\frac{\mu^2}{2}(\b w^n_{\w \gamma(\dt\mu)}(\dt))^T H_\eta(\b u^{n+1})\b w^n_{\w \gamma(\dt\mu)}(\dt) +\O(\dt),
		\end{equation}
		where $H_\eta$ denotes the Hessian. Then, since $\w \gamma(0)=1$, we have
		\[z(0,\mu)\coloneqq \lim_{\dt\to 0}z(\dt,\mu)=\mu \w c(\b u^n)+\frac{\mu^2}{2}(\b w^n_1(0))^T H_\eta(\b u^n)\b w^n_1(0)\qta \partial_\mu z(0,0)=\w c(\b u^n)\neq 0.\]
		According to the proof of \cite[Theorem~2]{CLM2010}, there exist $\dt^*>0$ and a unique function $\mu:[0,\dt^*]\to \R$ such that $z(\dt,\mu(\dt))=0$ for all $0\leq\dt\leq \dt^*$. Indeed, because of
		\[z(\dt,0)=\dt^{-2}\left(\eta(\b u^{n+1})-\etaold\right)
		\] it can be deduced along the same lines that $\mu=\O(\dt^{p-1})$.

		To prove \eqref{eq:z} we first note that
		\begin{equation}
			\bunplusgamma=\b u^{n+1}+\lambda\b w_\gamma^n(\dt) +\O(\dt^{p+1}),\label{eq:scheme_w}
		\end{equation}
		where
		\begin{equation}\label{eq:lambda_proof}
			\lambda\coloneqq(\gamma-1)\dt \lVert \b d_\gamma^n(\dt)\rVert_2.
		\end{equation}
		Next, as $h(\lambda,\gamma)=0$ we can use Lemma~\ref{lem:wgamma} to solve \eqref{eq:lambda_proof} for $\gamma$ and plug it into \eqref{eq:scheme_w} resulting in a function of $\lambda$ only, \ie
		\[\b u^{n+\w \gamma(\lambda) }=\b u^{n+1}+\lambda\b w_{\w \gamma(\lambda)}^n(\dt) +\O(\dt^{p+1}).\]  In the following, we denote by $\b u(t)$ the exact local solution at $t$ satisfying $\b u(\tn)=\b u^n$ and recall that we are considering an entropy-conservative problem. Hence, with $\lambda\eqqcolon\dt\mu$ and the assumptions on $\eta'$ we receive
		\begin{equation*}
			\begin{aligned}
				z(\dt,\mu)&=\dt^{-2}\left(\eta(\b u^{n+\w \gamma(\dt\mu)}) - \eta(\b u(\tn+\dt)))\right)\\
				&=\dt^{-2}\left(\eta(\b u^{n+1}+\dt\mu\b w_{\w \gamma(\dt \mu)}^n(\dt)) - \eta(\b u(\tn+\dt)))\right) + \O(\dt^{p-1})\\
				&=\frac{\eta(\b u^{n+1})- \eta(\b u(\tn+\dt)) + \dt\mu \eta'(\b u^{n+1})\b w_{\w \gamma(\dt \mu)}^n(\dt)}{\dt^2} \\
				&\hphantom{==}+  \frac{\mu^2}{2}(\b w^n_{\w \gamma(\dt\mu)}(\dt))^T H_\eta(\b u^{n+1})\b w^n_{\w \gamma(\dt\mu)}(\dt)  +\O(\dt)\\
				&=	\mu c(\b u^n,\mu) + \frac{\mu^2}{2}(\b w^n_{\w \gamma(\dt\mu)}(\dt))^T H_\eta(\b u^{n+1})\b w^n_{\w \gamma(\dt\mu)}(\dt) +\O(\dt).
			\end{aligned}
		\end{equation*}
		Finally, since $\O(\dt^{p-1})=\mu=\frac{\lambda}{\dt}=(\gamma-1)\lVert \b d_\gamma^n(\dt)\rVert_2$ we deduce that $\gamma=1+\O(\dt^{p-1-q})$.
	\end{proof}
	Now if we are given such a solution $\gamma=1+\O(\dt^{p-1})$ we can deduce that the relaxation update is of order $p$ as the following lemma shows.
	\begin{lemma}\label{lem:order}
		If $\bunplusgamma=\b u^{n+1}+(\gamma-1)(\b u^{n+1}-\b u^n) + \O(\dt^{p+1})$ with a $p$-th order baseline method and $\gamma=1+\O(\dt^{p-1})$, then \[\bunplusgamma = \b u(\tn+\gamma\dt) +\O(\dt^{p+1}).\]
	\end{lemma}
	\begin{proof}
		This is just Lemma~\cite[Lemma~2.7]{ranocha2020general}, where the relaxation method is perturbed by an additive error of $\O(\dt^{p+1})$.
	\end{proof}
	\paragraph{Bootstrapping Algorithm for Positivity-Preserving Relaxation}\label{sec:bootstrap}
	The main idea for generalizing the relaxation technique to higher order is to use the observation from Remark~\ref{rem:dgamma_using_sigma(gamma)}, where
	\[\bunplusgamma=\b u^{n+1}+(\gamma-1)\dt \underbrace{\frac{\b u^{n+1}-\b u^n}{\dt}}_{\eqqcolon \b d^n} + \O(\dt^{p+1}).\]
	Since $\b d^n$ now is independent of $\gamma$, there is no need for Lemma~\ref{lem:wgamma} as we have an explicit expression for $\w \gamma$, \ie $\w \gamma(\lambda)=\frac{\lambda +\dt \lVert \b d^n\rVert_2}{\dt \lVert \b d^n\rVert_2}$ and we can apply Theorem~\ref{thm:pos_relax_1} giving us $\gamma=1+\O(\dt^{p-1})$. Also, Remark~\ref{remark:issue_denseoutput_relaxation} motivates us to start a bootstrapping algorithm using the functions $\bar b_j(\gamma)=\gamma b_j$ also for higher order. This seems to contradict our results from the theory on dense output, but in combination with relaxation, the issue can be resolved, see Remark~\ref{rem:bootstrap_vs_dense} below.


	Now in view of the following lemma, we can bootstrap the relaxation technique to higher orders.
	\begin{lemma}\label{lem:bootstrap}
		Consider a scheme of the form \eqref{eq:DO_ansatz2} with $\bar b_j(\gamma)=\gamma b_j$ and
		\begin{equation}\label{eq:bsigma_proof}
			\begin{aligned}
				\bar\bsigma(\gamma)=\b u(\tn+\dt) +(\gamma-1)(\b u(\tn+\dt)-\b u(\tn)) +\O(\dt^p)
			\end{aligned}
		\end{equation} for all $\gamma$ in a neighborhood $V$ of $1$.
		Then \[	\frac{u_\nu^{n+\gamma}}{\bar \sigma_\nu(\gamma)}= 1+\O(\dt^p)=\frac{u_\nu^{n+1}}{\sigma_\nu}+\O(\dt^p), \]
		and in particular,
		\[\bunplusgamma=\b u^{n+1}+(\gamma-1)(\b u^{n+1}-\b u^n) + \O(\dt^{p+1})\]
		for all $\gamma\in V$.

	\end{lemma}
	\begin{remark}\label{rem:bootstrap_vs_dense}
		Before we prove this lemma, we want to stress two things. First, using \eqref{eq:bsigma_proof} with $\bar b_j(\gamma)=\gamma b_j$ does not result in a higher-order dense output formula, but only guarantees to obtain the desired order at the root $\gamma$ of $\eta(\bunplusgamma)=\etaold$, see Theorem~\ref{thm:pos_relax_1} and Lemma~\ref{lem:order}.

		Secondly, the bootstrapping process consists of using $\bunplusgamma$ from the scheme of order $p-1$ as the new $\bar \bsigma(\gamma)$ resulting in a new method of order $p$. We can start the bootstrapping process by using the second-order MPRK22($\alpha$) scheme as a baseline method together with \eqref{eq:sigma_denseMPRK2}. Note that this naturally results in nested functions that depend on $\gamma$, which should be kept in mind when implementing the Newton iteration.
	\end{remark}

	\begin{proof}[Proof of Lemma~\ref{lem:bootstrap}]
		We prove this claim by induction over $q=1,\dotsc,p$ and exploit \cite[Lemma~4.6,Lemma~4.8]{IzginThesis} to justify the implication
		\[\bunplusgamma=\bar \bsigma(\gamma)+\O(\dt^q) \Longrightarrow 	\frac{u_\nu^{n+\gamma}}{\bar \sigma_\nu(\gamma)}= 1+\O(\dt^{q}),\quad \nu=1,\dotsc,N.\]

		For $q=1$ we find $\bunplusgamma=\b u^n+\O(\dt)$ and $\bar\bsigma(\gamma)=\b u^n+\O(\dt)$ since $	\frac{u_\nu^{n+\gamma}}{\bar \sigma_\nu(\gamma)}= \O(1)$ due to \cite[Lemma~4.6]{IzginThesis}.

		Now suppose that \eqref{eq:bsigma_proof} holds with $\O(\dt^{q})$ and some $q\geq 2$. By the induction hypothesis we have
		\[	\frac{u_\nu^{n+\gamma}}{\bar \sigma_\nu(\gamma)}= 1+\O(\dt^{q-1})=\frac{u_\nu^{n+1}}{\sigma_\nu}+\O(\dt^{q-1}) \] where we used
		$\frac{u_\nu^{n+1}}{\sigma_\nu}=1+\O(\dt^p)$ and $p\geq q$ for the last equality, see \eg \cite[Lemma~16]{NSARK}. Substituting this into \eqref{eq:DO_ansatz2}, we see
		\[\bunplusgamma=\b u^{n+1}+(\gamma-1)(\b u^{n+1}-\b u^n) + \O(\dt^{q}).\] Finally, since \eqref{eq:bsigma_proof} holds with $\O(\dt^{q})$, we deduce
		\[	\frac{u_\nu^{n+\gamma}}{\bar \sigma_\nu(\gamma)}= 1+\O(\dt^{q})=\frac{u_\nu^{n+1}}{\sigma_\nu}+\O(\dt^{q}).\]
	\end{proof}
	\begin{example}[Third-Order Relaxation for Conservative Problems using MPRK]
		Looking at  the third-order MPRK family from Example~\ref{example:MPRK43I}, the relaxation scheme is fully defined by setting
		\begin{equation}\label{eq:sigma_gamma_for_MPRK43I}
			\begin{aligned}
				\bar\sigma_k(\gamma)=\, &u_k^n+ \gamma\dt \sum_{j=1}^2\beta_j\left(r_k^P(\buj) + \sum_{\nu=1}^d  p_{k\nu}(\buj)\frac{\bar\sigma_\nu(\gamma)}{(u_\nu^{(2)})^{\frac{\gamma}{a_{21}}}(u_\nu^n)^{1-\frac{\gamma}{a_{21}}}}\right. \\
				& \hspace{2cm}- \left.\left(r_k^D(\buj) + \sum_{\nu=1}^d d_{k\nu}(\buj)\right)\frac{\bar\sigma_k(\gamma)}{(u_k^{(2)})^{\frac{\gamma}{a_{21}}}(u_k^n)^{1-\frac{\gamma}{a_{21}}}}\right)
			\end{aligned}
		\end{equation}
		for $k=1,\dotsc,d$, $\beta_1=1-\beta_2$, and $\beta_2=\frac{1}{2a_{21}}$.
	\end{example}

	\paragraph{Applying Newton's Method}
	As an illustrative example, we focus on \eqref{eq:DO_ansatz2} with $\bar b_j(\gamma)=\gamma b_j$, \ie
	\begin{equation}
		\begin{aligned}
			\unplusgamma_k=u_k^n &+ \dt\gamma\sum_{j=1}^s b_j\left(r_k^P(\b u^{(j)}) + \sum_{\nu=1}^dp_{k\nu}(\b u^{(j)}) \frac{u_\nu^{n+\gamma}}{\bar\sigma_\nu(\gamma)}-\left( r^D_k(\b u^{(j)}) + \sum_{\nu=1}^dd_{k\nu}(\b u^{(j)})\right)\frac{u_k^{n+\gamma}}{\bar \sigma_k(\gamma)}\right)\\
		\end{aligned},\label{eq:unplusgamma_bootstrap}
	\end{equation}
	and that
	\begin{equation}\label{eq:M_update_bootstrap}
		\begin{aligned}
			(\b M_\gamma)_{kk}&= 1+\gamma\dt \sum_{j=1}^sb_j\left(r_k^D(\b u^{(j)}) +\sum_{\nu=1}^dd_{k\nu}(\b u^{(j)}) \right)\frac{1}{\bar \sigma_k(\gamma)}, \\
			(\b M_\gamma)_{k\nu}&= -\gamma\dt \sum_{j=1}^sb_jp_{k\nu}(\b u^{(j)})\frac{1}{\bar\sigma_\nu(\gamma)}, \quad k\neq \nu.
		\end{aligned}
	\end{equation}
	in \eqref{eq:System_Relaxation_MPRK_step}.

		Starting with \eqref{eq:unplusgamma_bootstrap} we deduce
		\begin{equation*}
			\begin{aligned}
				\frac{\dd }{\dd \gamma} 	\unplusgamma_k= \dt\sum_{j=1}^s\bar b_j'(\gamma)&\left(r_k^P(\b u^{(j)}) + \sum_{\nu=1}^dp_{k\nu}(\b u^{(j)}) \frac{u_\nu^{n+\gamma}}{\bar{\sigma}_\nu(\gamma)}-\left( r^D_k(\b u^{(j)}) + \sum_{\nu=1}^dd_{k\nu}(\b u^{(j)})\right)\frac{u_k^{n+\gamma}}{\bar{\sigma}_k(\gamma)}\right)\\
				+ \dt\sum_{j=1}^s\bar b_j(\gamma&)\left(\sum_{\nu=1}^dp_{k\nu}(\b u^{(j)}) \left(\frac{\frac{\dd }{\dd \gamma}u_\nu^{n+\gamma}}{\bar{\sigma}_\nu(\gamma)}-\frac{u_\nu^{n+\gamma}\bar{\sigma}_\nu'(\gamma)}{(\bar{\sigma}_\nu(\gamma))^2}\right)\right.\\
				&\left.-\left( r^D_k(\b u^{(j)}) + \sum_{\nu=1}^dd_{k\nu}(\b u^{(j)})\right) \left(\frac{\frac{\dd }{\dd \gamma}u_k^{n+\gamma}}{\bar{\sigma}_k(\gamma)}-\frac{u_k^{n+\gamma}\bar{\sigma}_k '(\gamma)}{(\bar{\sigma}_k(\gamma))^2}\right)\right).
			\end{aligned}
		\end{equation*}
		To rewrite this in matrix-vector notation, we 
		denote by  \enquote{$\oslash$} and \enquote{$\odot$} the component-wise division and multiplication (Hadamard division and product), respectively. Then, using
		\[\b v(\gamma)\coloneqq \bunplusgamma \odot \bar \bsigma'(\gamma) \oslash (\bar \bsigma(\gamma)),\]
		we end up with
		\begin{equation}\label{eq:ungamma_prime_general}
			\begin{aligned}
				\b M_\gamma(\b u^n)\frac{\dd }{\dd \gamma} 	\bunplusgamma =& \frac{1}{\gamma}(\bunplusgamma-\b u^n)	+ (\b M_\gamma(\b u^n) - \b I)  \b v(\gamma).
			\end{aligned}
		\end{equation}
		Note that if $\bar \bsigma(\gamma)=\bsigma$, then $\b v(\gamma)=\b 0$. Also note that the derivative of $\bar\bsigma$ from \eqref{eq:sigma_gamma_for_MPRK43I} itself satisfies an analogue equation to \eqref{eq:ungamma_prime_general} as it represents an MPRK22($\alpha$) relaxation method of order $2$.

		Also, the system for MPSSPRK22 is the same; one only has to plug in the expressions $\b M_\gamma$ and $\bar\bsigma$ for that particular scheme.

		\section{Numerical Experiments}\label{sec:numexp}
		In this section we apply our new relaxation algorithm to dissipative and conservative problems to validate our theoretical findings and to experimentally test the constraints on $\dt$ for solving the system \eqref{eq:System_Relaxation_MPRK_step}. We note that we use Newton's method for the computation of $\gamma$, if not stated otherwise. Also, we use \eqref{eq:sigma_MPRK22} as default for MPRK22$(\alpha$). Also, we may use a PID controller with parameters $\beta_1=0.7$, $\beta_2=0.4$, and $\beta_3=0$, see \cite{IR2023} for more details. The resulting method is denoted by MPRK22adap. We note that our implementation of the relaxation algorithm, which can be found in our reproducibility repository \cite{IRS2026repository}, is also adaptive in the sense that successful relaxation steps increase the $\Delta t$ by 1\% while unsuccessful searches for $\gamma$ result in a 10\% decrease of the time step size. We refer to the repository \cite{IRS2026repository} for the implemented abortion criteria.

		\subsection{Lotka-Volterra System}

		The classical Lotka-Volterra system
		\begin{equation*}
			\begin{aligned}
				u_1'(t) &= 2u_1(t) - u_1(t) u_2(t), d\\
				u_2'(t) &= u_1(t) u_2(t) - u_2(t),\quad \b u(0)=(2,2)^T
			\end{aligned}
		\end{equation*}
		can be written as a non-conservative PDRS with
		\begin{equation*}
			r_1^P = 2 u_1, \quad p_{21} = u_1 u_2 = d_{12}, \quad r_2^D = u_2.
		\end{equation*}
		The entropy
		\begin{equation*}
			\eta(\b u) = \log(u_1) - u_1 + 2\log(u_2) - u_2
		\end{equation*}
		is conserved. Since the Lotka-Volterra system has periodic orbits,
		we expect improved numerical results using relaxation to conserve
		the entropy \cite{cano1997error,cano1998error,calvo2011error}. 
		\begin{figure}[!htbp]
			\begin{subfigure}[t]{0.5\textwidth}
				\includegraphics[width=\textwidth]{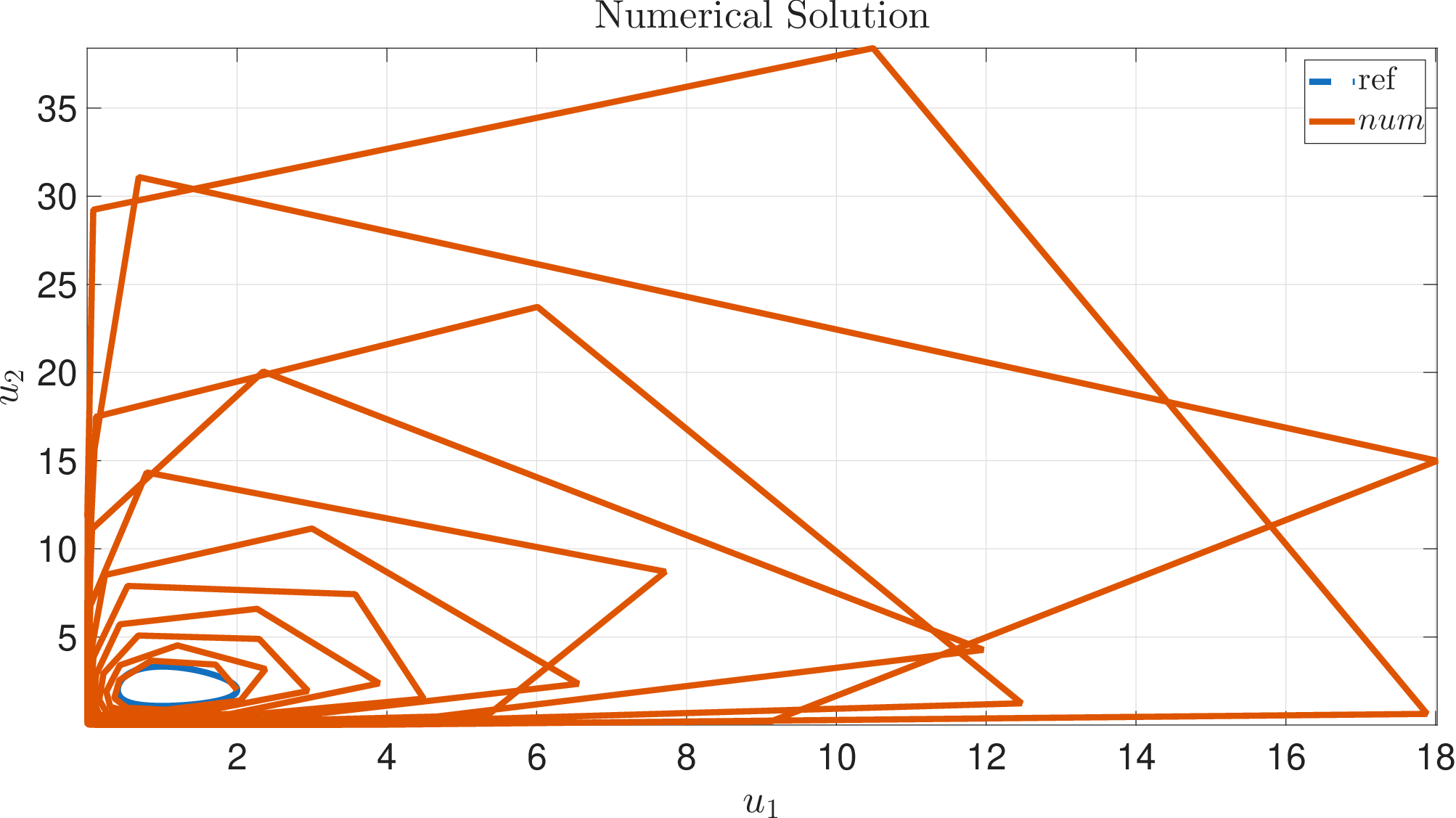}
			\end{subfigure}
			\begin{subfigure}[t]{0.5\textwidth}
				\includegraphics[width=\textwidth]{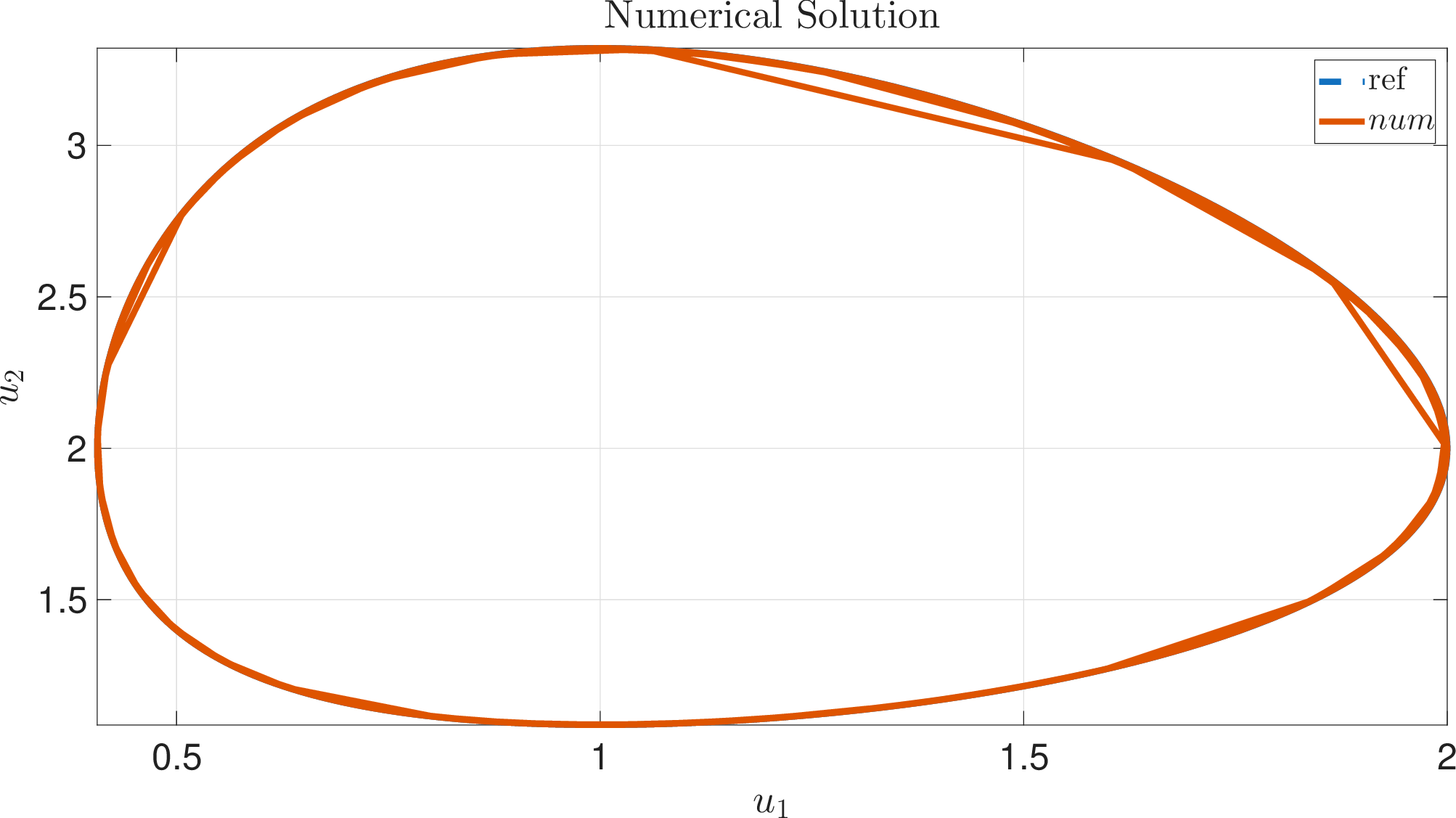}
			\end{subfigure}\\
			\begin{subfigure}[t]{0.48\textwidth}
				\includegraphics[width=\textwidth]{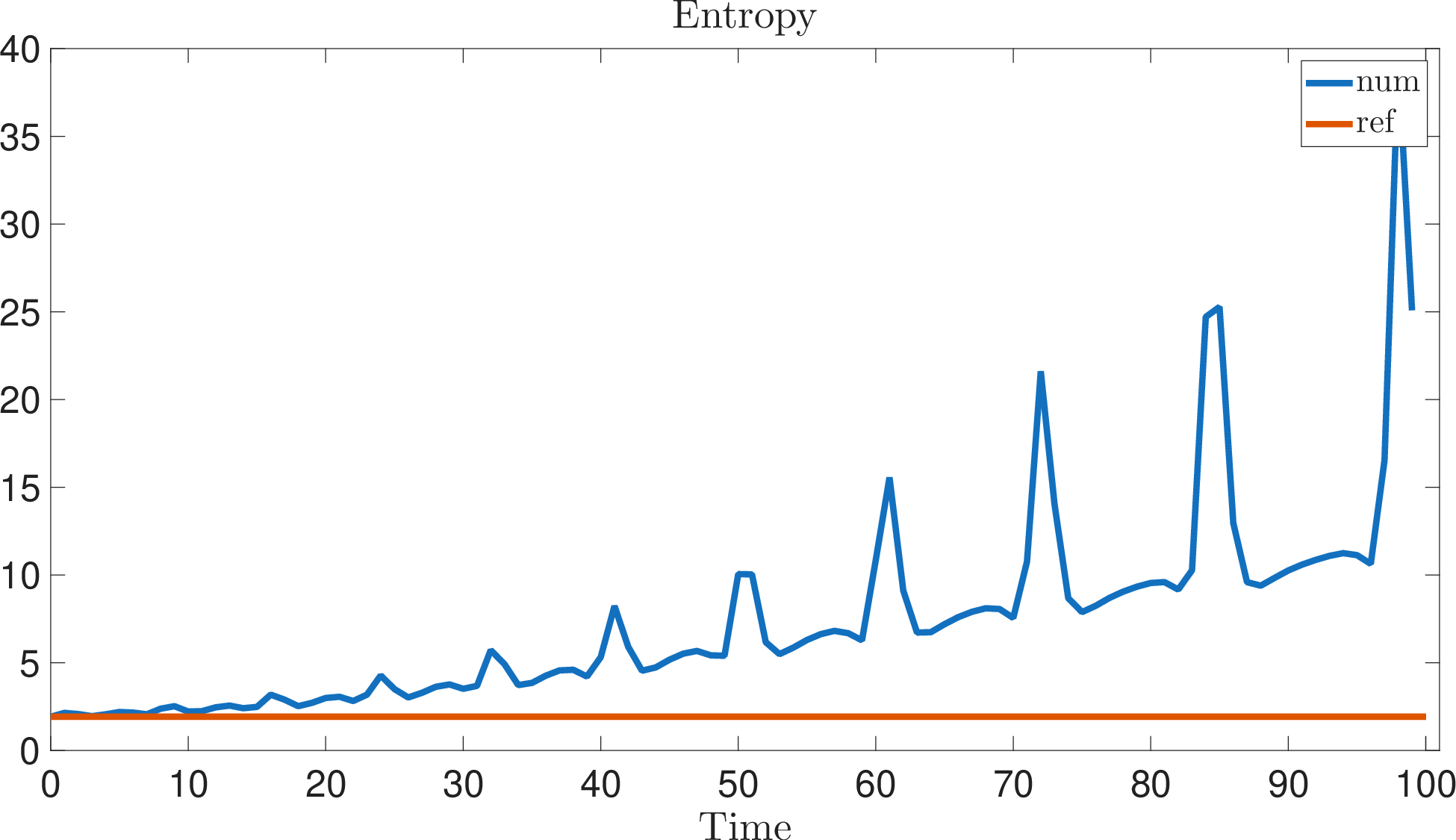}
			\end{subfigure}
			\begin{subfigure}[t]{0.495\textwidth}
				\includegraphics[width=\textwidth]{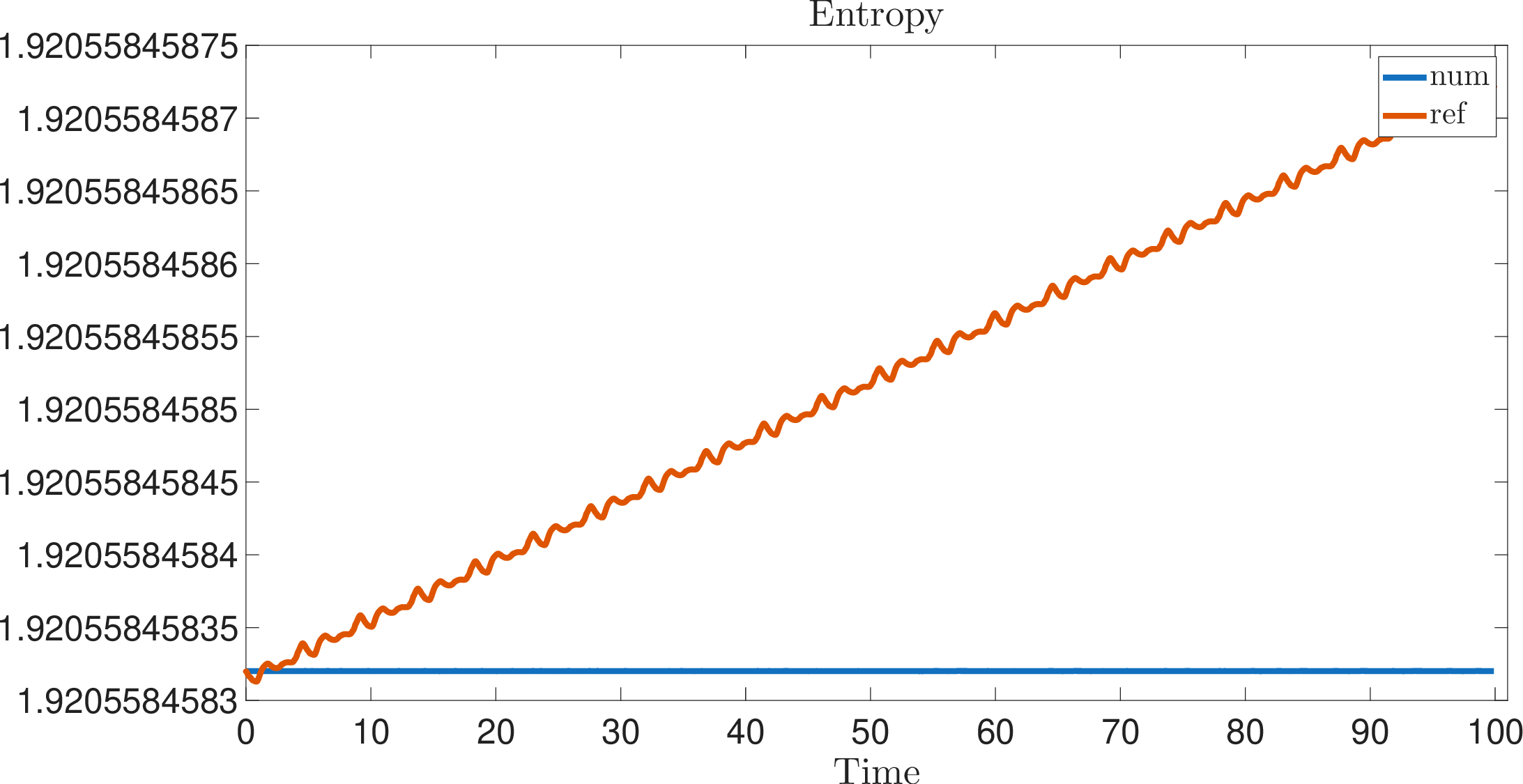}
			\end{subfigure}
			\begin{subfigure}[t]{0.495\textwidth}
				\includegraphics[width=\textwidth]{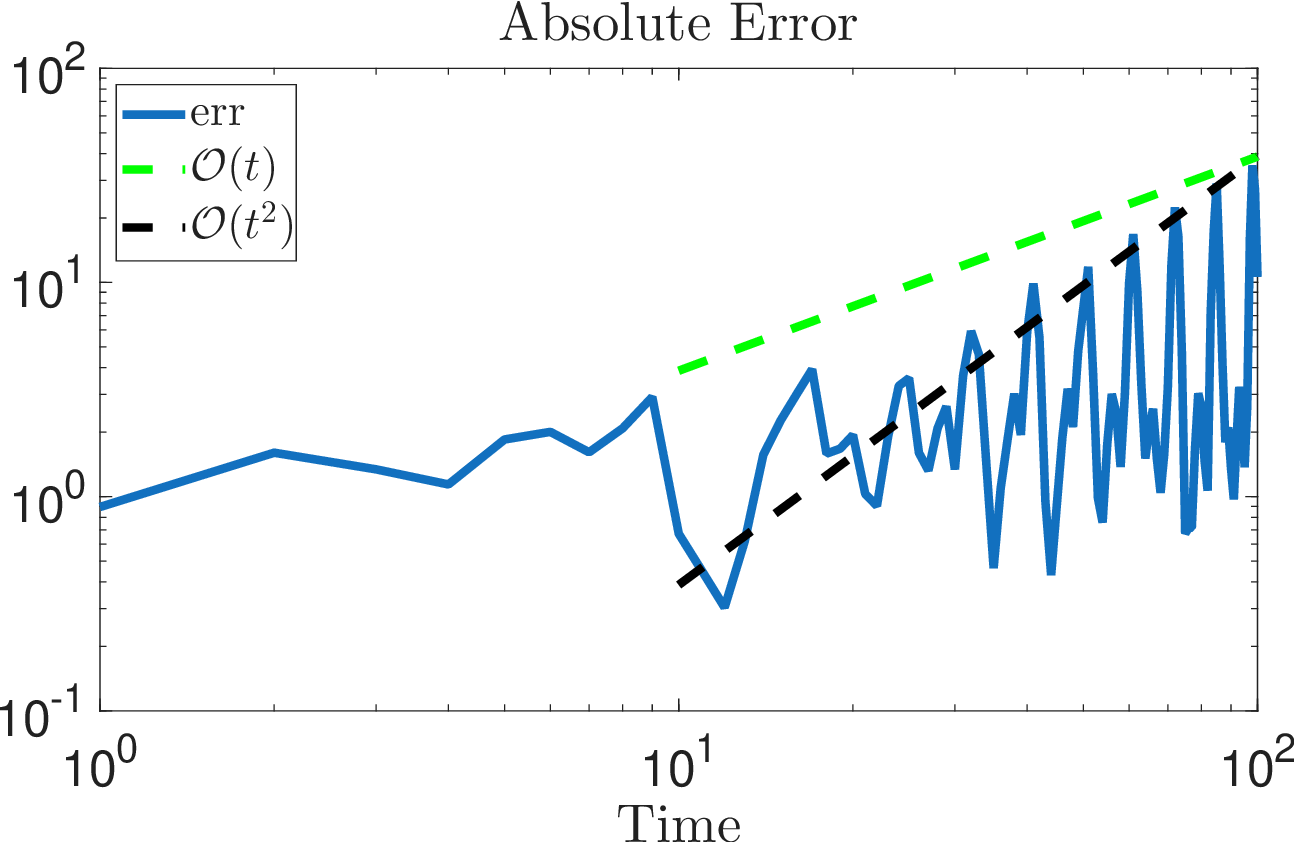}
			\end{subfigure}
			\begin{subfigure}[t]{0.495\textwidth}
				\includegraphics[width=\textwidth]{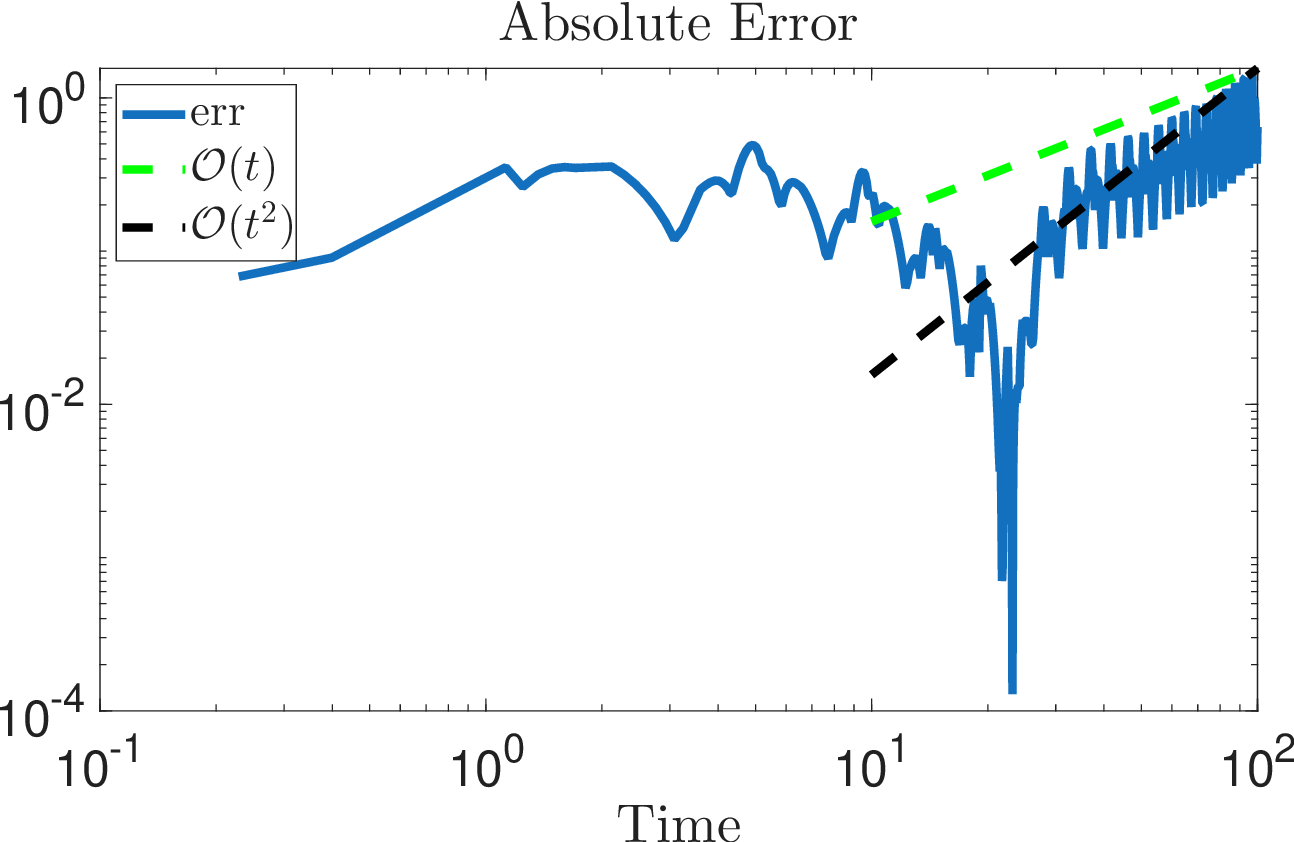}
			\end{subfigure}
			\caption{Numerical solution of Lotka Volterra problem using MPRK22(1) (top) and $\dt=1$. The error is $\text{err}^n=\max(\lvert u_1^n-u_1^{\text{ref},n}\rvert, \lvert u_2^n-u_2^{\text{ref},n}\rvert)$.  Left: without relaxation. Right: with relaxation. }\label{Fig:MPRK22_LV}
		\end{figure}Although there are only positivity constraints, $\eta$ is not non-decreasing for all $\b u>\b 0$ in this example, which is why we use the default relaxation algorithm.
	As expected, we observe that relaxation improves the error growth of the base method from quadratic to linear,
	see Figure~\ref{Fig:MPRK22_LV}.
		\subsection{Stratospheric Reaction Problem}
		The stratospheric reaction problem \cite{sandu2001positive} is a stiff system of ODEs $\b w'(t)=\b f(t,\b w(t))$ describing the interaction of the constituents $\b w=(w_1,\dotsc,w_6)=(O^{1D},O,O_3,O_2,NO,NO_2)$. This non-conservative PDS possesses two linear invariants determined by the vectors $\widetilde{\b n}_1=(1,1,3,2,1,2)^T$ and $\widetilde{\b n}_2=(0,0,0,0,1,1)^T$. In order to apply MPRK schemes to this problem, we scale the corresponding differential equations writing
		\[\diag(\widetilde{\b n}_1)\b w'(t)=\diag(\widetilde{\b n_1})\b f(t,\diag(\widetilde{\b n}_1)^{-1}\diag(\widetilde{\b n}_1)\b w(t)).\] Hence, introducing $\b u(t)=\diag(\widetilde{\b n}_1)\b w(t)$, the two linear invariants of the differential equations $\b u'(t)=\b f(t,\diag(\widetilde{\b n}_1)^{-1}\b u(t))$ are $\b n_1=(1,1,1,1,1,1)^T$ and $\b n_2=(0,0,0,0,1,\tfrac12)^T$. Moreover, the scaled system takes the form
		\begin{equation}\label{eq:strat}
			\begin{aligned}
				u_1'&=
				\tfrac13 r_5(t, \b u) - (r_6(\b u) + \tfrac13 r_7(\b u)), \\
				u_2'&= r_1(t, \b u) + \tfrac13 r_3(t, \b u) + r_6(\b u)+\tfrac12 r_{10}(t, \b u)  - (\tfrac12 r_2(\b u) + \tfrac13 r_4(\b u) + \tfrac12 r_9(\b u)+r_{11}(\b u)),\\
				u_3'&=\tfrac32 r_2(\b u) - (r_3(t, \b u) + r_4(\b u) + r_5(t, \b u)  + r_7(\b u) + r_8(\b u)),\\
				u_4'&=\tfrac23 r_3(t, \b u) + \tfrac43 r_4(\b u) + \tfrac23 r_5(t, \b u)  + \tfrac43 r_7(\b u) + \tfrac23 r_8(\b u) +r_9(\b u) - (r_1(t, \b u)+r_2(\b u)),\\
				u_5'&=\tfrac12 r_9(\b u)+\tfrac12 r_{10}(t, \b u) - (\tfrac13 r_8(\b u)r_{11}(\b u)),\\
				u_6'&=\tfrac23 r_8(\b u) + 2 r_{11}(\b u) - (r_9(\b u) + r_{10}(t, \b u)),
			\end{aligned}
		\end{equation}
		where
		\begin{equation*}
			\begin{aligned}
				r_1(t, \b u)&=k_1(t)u_4, &&& k_1(t)&=\sigma(T(t))^3\cdot 2.643\cdot 10^{-10}, \\
				r_2(t, \b u)&=k_2u_2u_4, &&& k_2&=8.018\cdot 10^{-17}, \\
				r_3(t, \b u)&=k_3(t)u_3, &&& k_3(t)&=\sigma(T(t))\cdot 6.120\cdot 10^{-4}, \\
				r_4(t, \b u)&=k_4u_2u_3, &&& k_4&=1.576\cdot 10^{-15}, \\
				r_5(t, \b u)&=k_5(t)u_3, &&& k_5(t)&=\sigma(T(t))^2\cdot 1.070\cdot 10^{-3}, \\
				r_6(t, \b u)&=k_6Mu_1, &&& k_6&=7.110\cdot 10^{-11}, &M&=8.120\cdot 10^{16},\\
				r_7(t, \b u)&=k_7u_1u_3, &&& k_7&=1.200\cdot 10^{-10}, \\
				r_8(t, \b u)&=k_8u_3u_5, &&& k_8&=6.062\cdot 10^{-15}, \\
				r_9(t, \b u)&=k_9u_2u_6, &&& k_9&=1.069\cdot 10^{-11}, \\
				r_{10}(t, \b u)&=k_{10}(t)u_6, &&& k_{10}(t)&=\sigma(T(t))\cdot 1.289\cdot 10^{-2}, \\
				r_{11}(t, \b u)&=k_{11}u_2u_5, &&& k_{11}&= 10^{-8}
			\end{aligned}
		\end{equation*}
		as well as
		\begin{equation*}
			\begin{aligned}
				\sigma(T(t))&= \begin{cases}
					0.5+0.5\cos\left(\pi \abs{\frac{2T(t)-T_r-T_2}{T_s-T_r}}\frac{2T(t)-T_r-T_2}{T_s-T_r}\right), & T_r\leq T(t)\leq T_s,\\
					0,& \text{otherwise},
				\end{cases}\\
				T(t)&=\frac{t}{3600}\mod 24, \quad T_r=4.5,\quad T_s=19.5.
			\end{aligned}
		\end{equation*}
		The non-zero production and destruction terms of the system \eqref{eq:strat} are given by
		\begin{equation*}
			\begin{aligned}
				d_{12}(t, \b u) &= r_6(\b u), &
				d_{14}(t, \b u) &= \tfrac13 r_7(\b u),&
				d_{23}(t, \b u) &= \tfrac12 r_2(\b u),\\
				d_{24}(t, \b u) &= \tfrac13 r_4(\b u),&
				d_{25}(t, \b u) &= \tfrac12 r_9(\b u),&
				d_{26}(t, \b u) &= r_{11}(\b u),\\
				d_{31}(t, \b u) &= \tfrac13 r_5(t, \b u),&
				d_{32}(t, \b u) &= \tfrac13  r_3(t, \b u),&
				d_{36}(t, \b u) &= \tfrac13 r_8(\b u),\\
				d_{34}(t, \b u) &= \tfrac23 r_3(t, \b u) + r_4(\b u)+ \tfrac23 r_5(t, \b u)  + r_7(\b u) + \tfrac23 r_8(\b u) ,\hspace{-6cm}\\
				d_{42}(t, \b u) &= r_1(t, \b u),&
				d_{43}(t, \b u) &= r_2(\b u),&
				d_{56}(t, \b u) &= r_{11}(\b u) +\tfrac13 r_8(\b u),\\
				d_{62}(t, \b u) &=\tfrac12 r_{10}(t, \b u),&
				d_{64}(t, \b u) &= r_9(\b u),&
				d_{65}(t, \b u) &= \tfrac12 r_{10}(t, \b u),
			\end{aligned}
		\end{equation*}
		and $p_{ij}=d_{ji}$.
		The solution to this problem will be approximated over the time interval $[12\cdot 3600,84\cdot3600]$, where a unit of time represents a second.  A reference solution of the scaled problem is depicted in Figure~\ref{fig:stratref}.
		\begin{figure}
			\centering
			\includegraphics[width=0.7\textwidth]{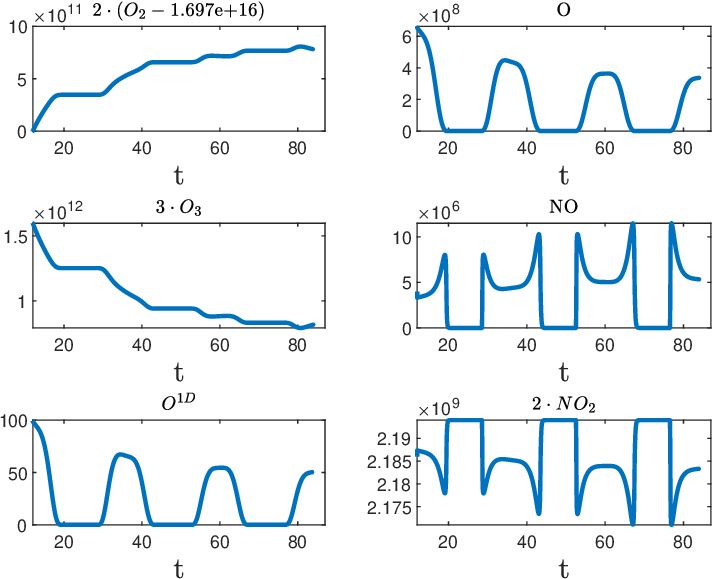}
			\caption{Reference solution of the scaled stratospheric problem \eqref{eq:strat} depicted over the time interval $[12,84]$ in hours.}\label{fig:stratref}
		\end{figure}
		As we will see, MPRK schemes do not conserve the second linear invariant, which is why \[
		\b n_2^T \frac{\b u^{n+1}-\b u^n}{\lVert\b u^{n+1}-\b u^n\rVert}=c(\b u^n)\dt+\O(\dt^2)\]
		with $c\neq 0$. Hence, we may use
		\[\eta(\b u)=\b n_2^T \b u\]
		as an entropy function satisfying \eqref{eq:cond_eta} to preserve also the second linear invariant with our relaxation technique for conservative problems. As can be seen in Figure~\ref{Fig:MPRK22_strat}, using relaxation improves the accuracy of the solution significantly. However, we note that Newton's method, while working in principle, sometimes fails at finding a solution $\gamma\approx1$ in our implementation. Indeed, we observed $\gamma\approx10^{-12}$ and thus decided to use Regula Falsi as a solver. 
		\begin{figure}[!htbp]
			\hspace{-0.7cm}\begin{subfigure}[t]{0.6\textwidth}
				\includegraphics[width=\textwidth]{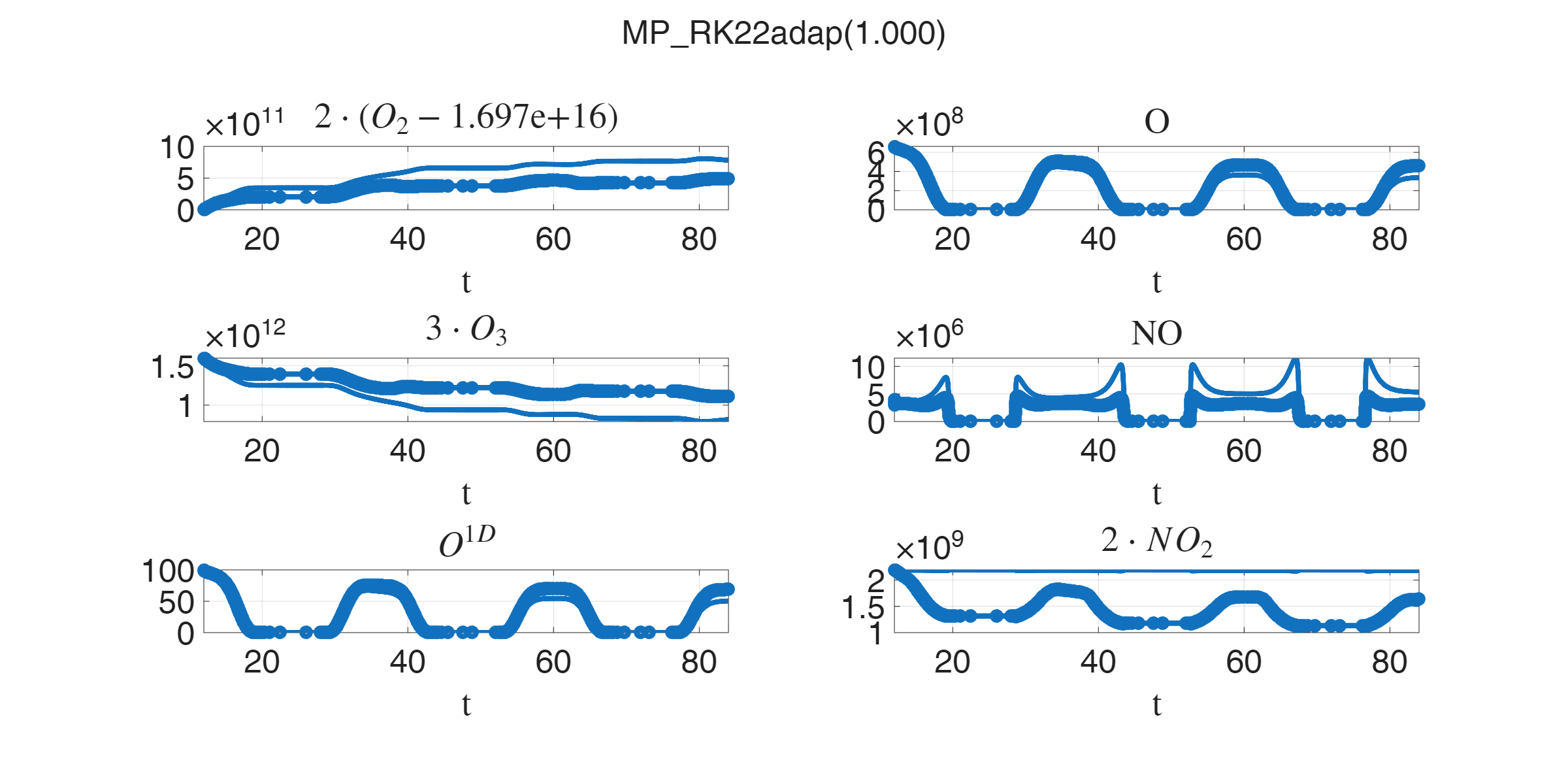}
			\end{subfigure}\hfill
			\begin{subfigure}[t]{0.52\textwidth}
				\includegraphics[width=\textwidth]{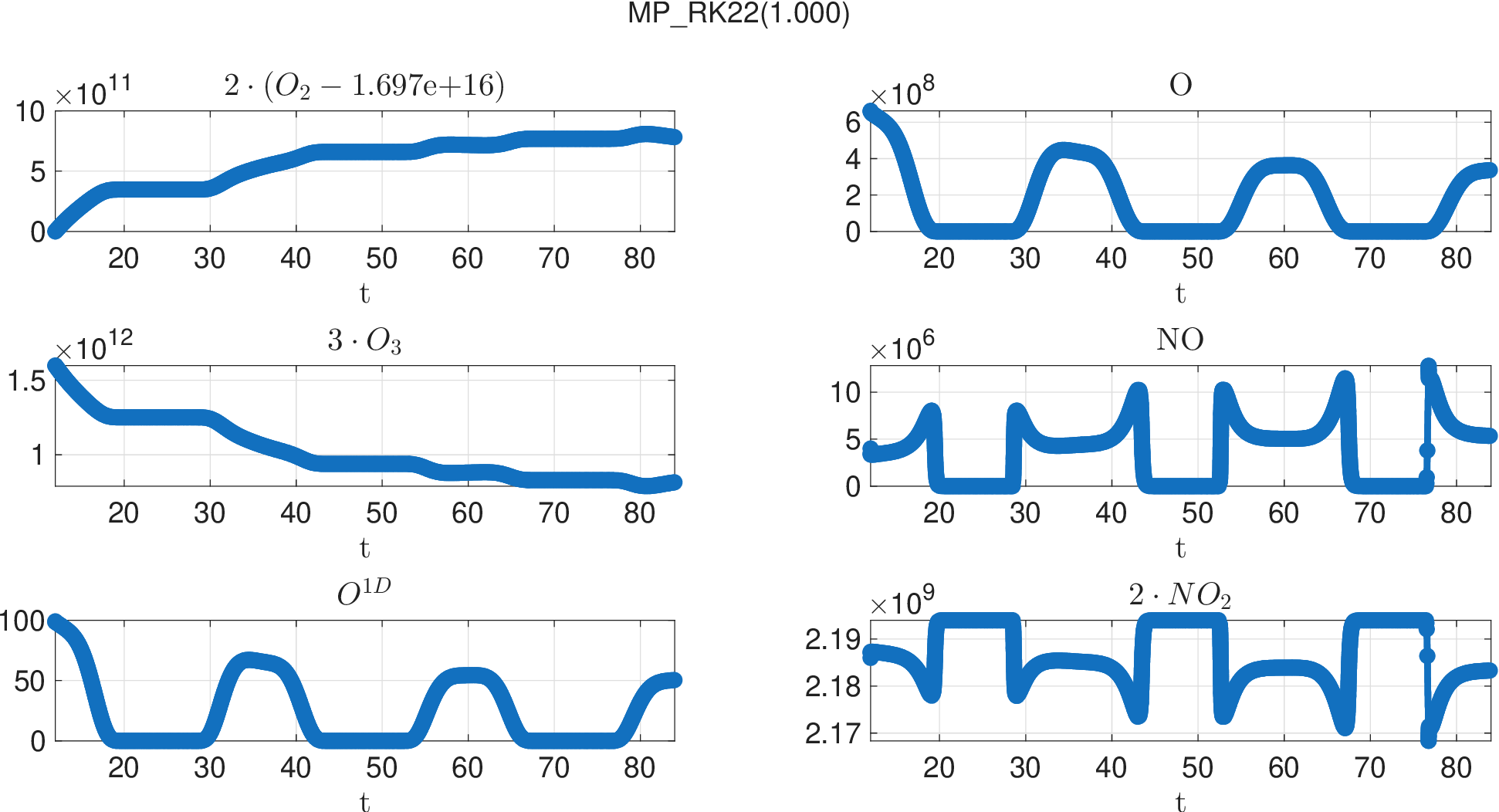}
			\end{subfigure}
			\vspace{1em} 
			\begin{subfigure}[t]{0.45\textwidth}
				\hspace{-0.4cm}	\includegraphics[width=\textwidth]{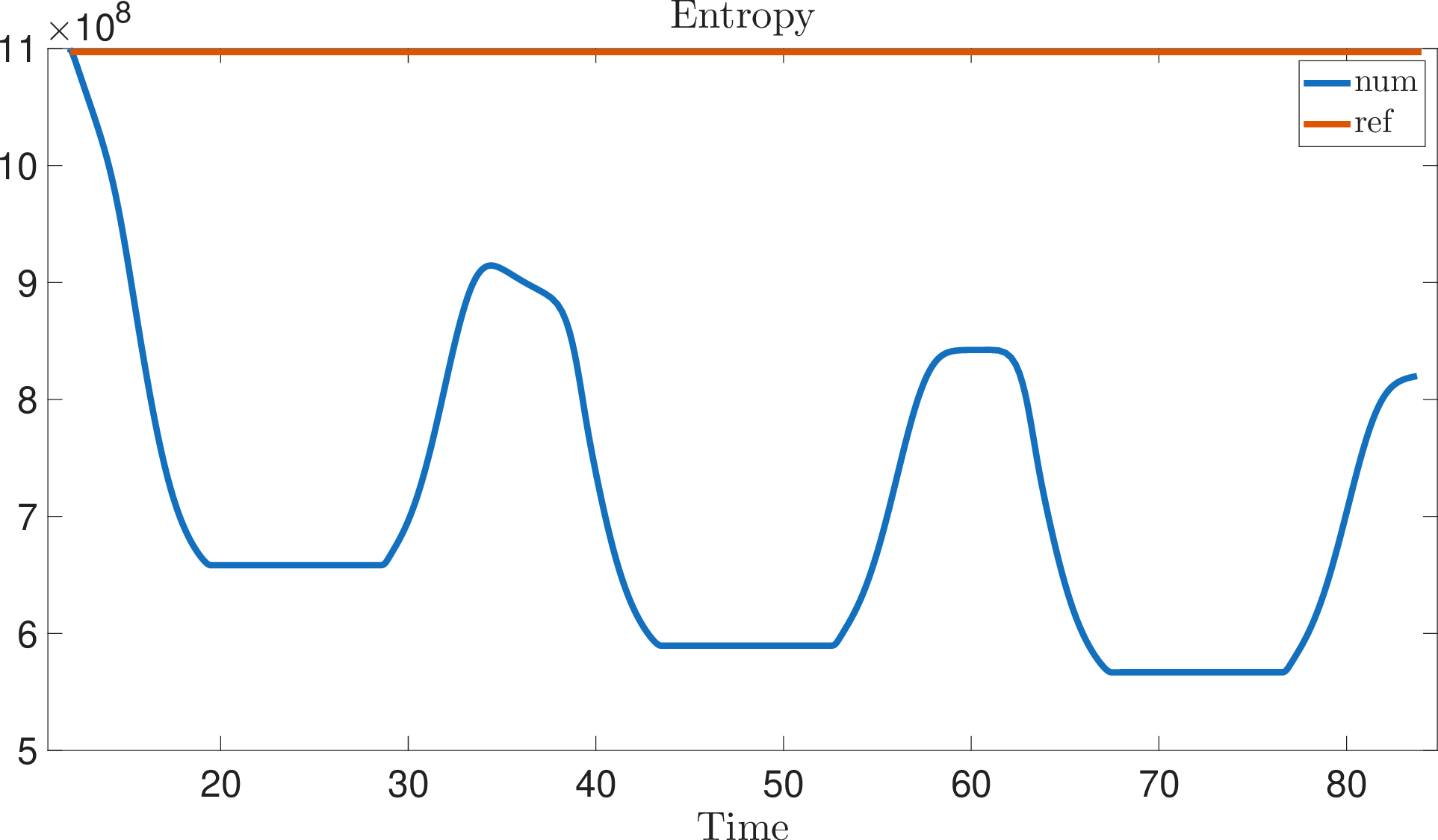}
			\end{subfigure}\hfill
			\begin{subfigure}[t]{0.5\textwidth}
				\includegraphics[width=\textwidth]{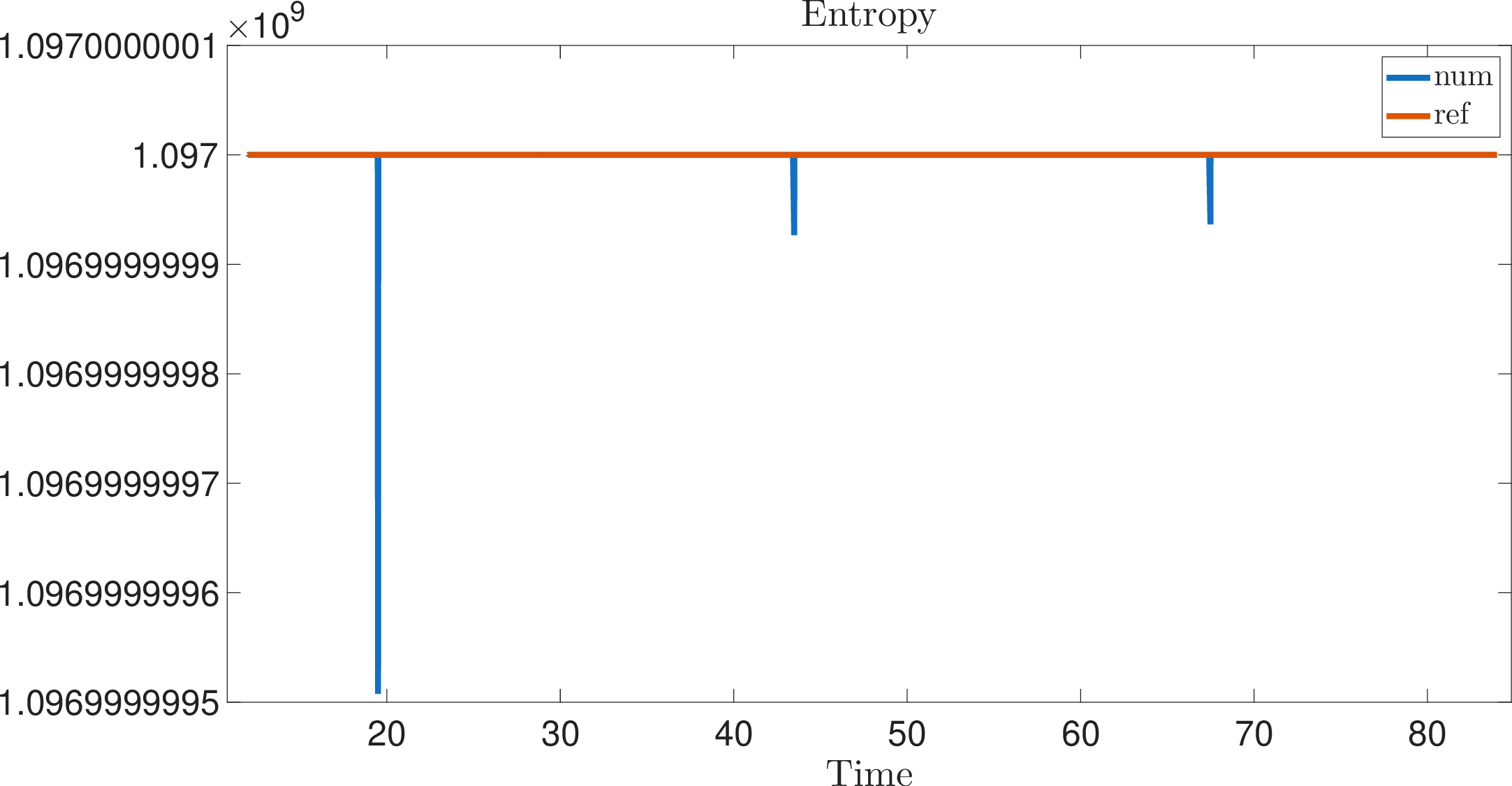}
			\end{subfigure}
			\vspace{1em} 
			\begin{subfigure}[t]{0.495\textwidth}
				\includegraphics[width=\textwidth]{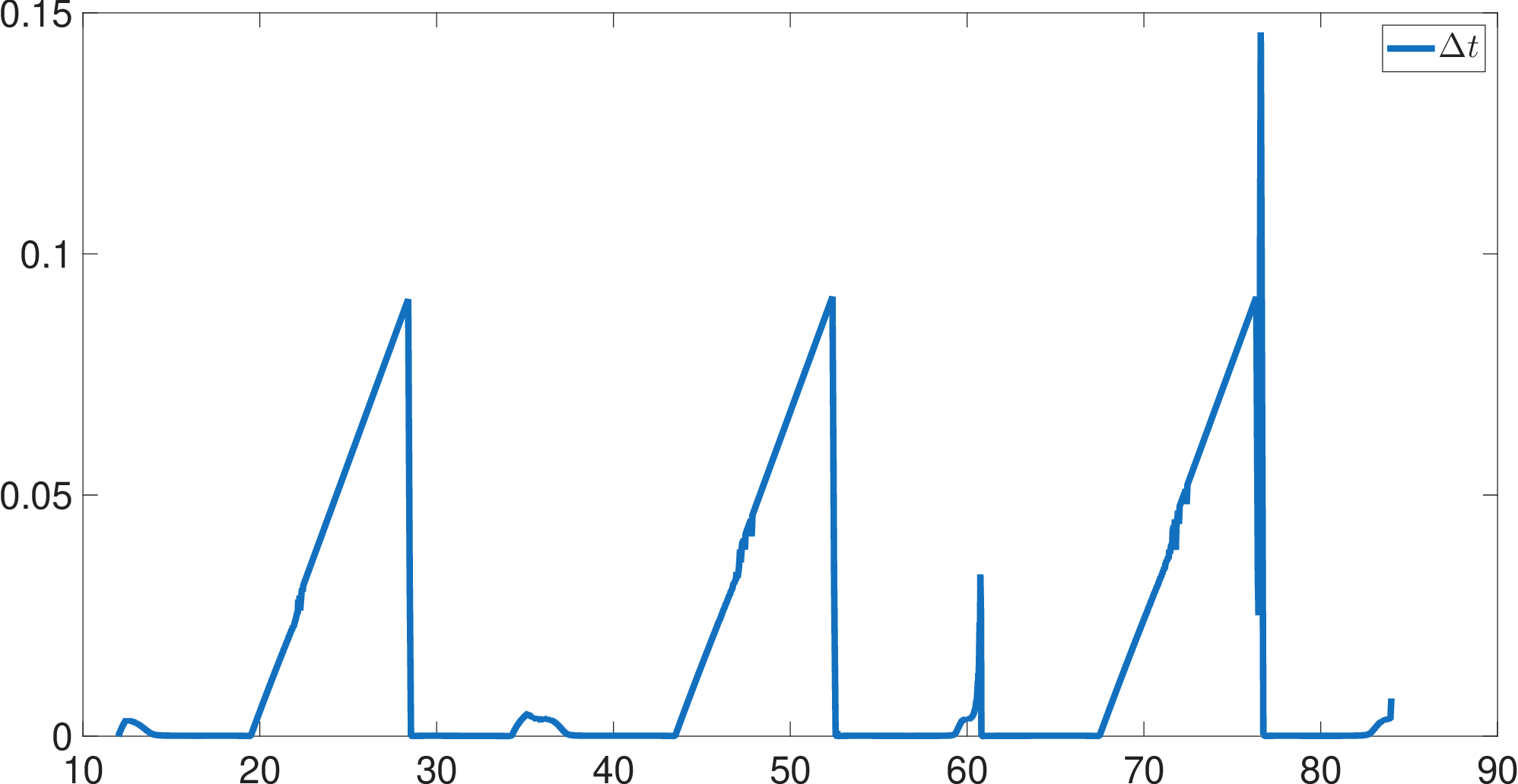}
			\end{subfigure}\hfill
			\begin{subfigure}[b]{0.495\textwidth}
				\includegraphics[width=\textwidth]{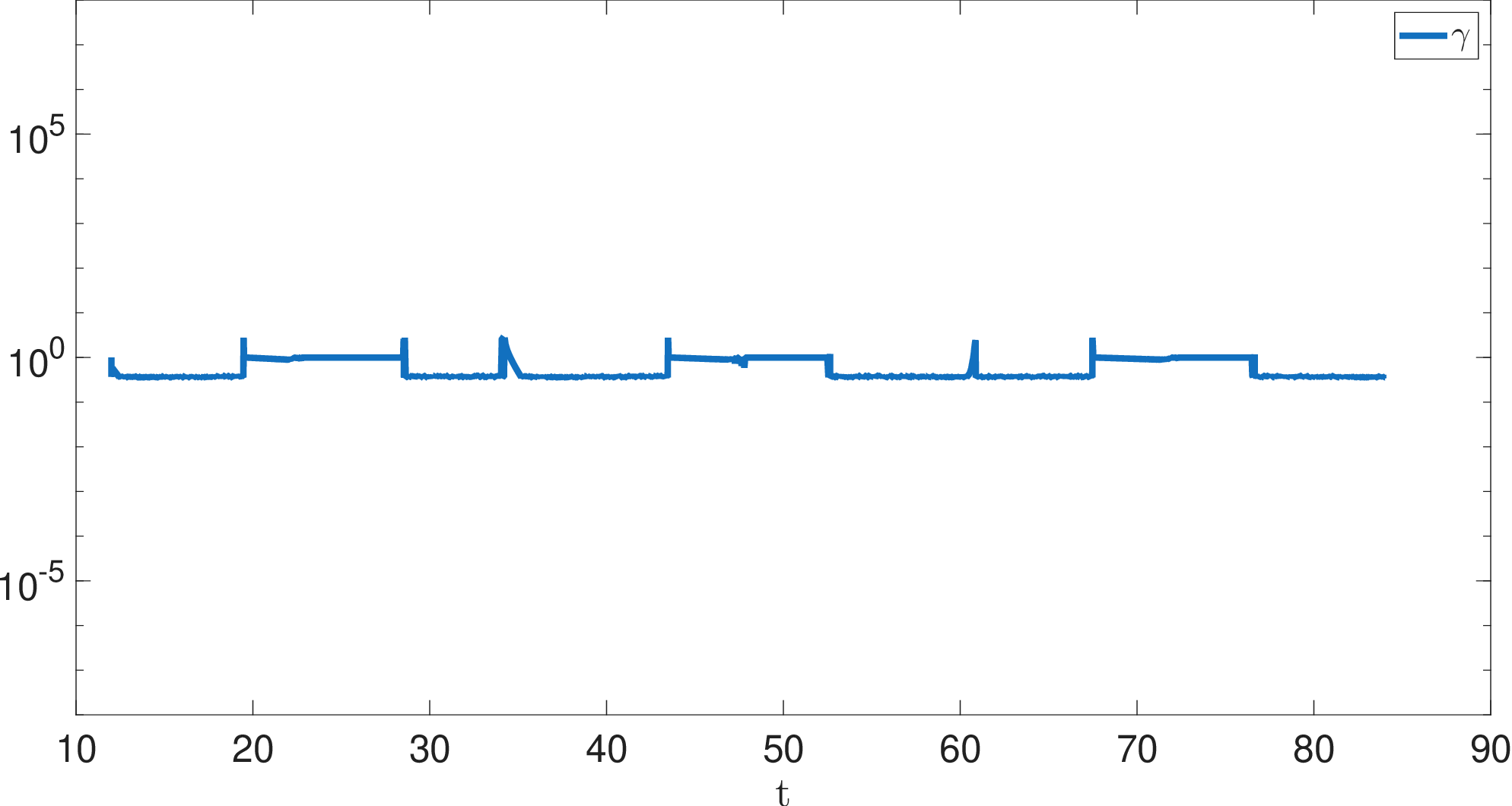}
			\end{subfigure}

			\caption{Numerical solution of stratospheric reaction problem using standard MPRK22adap(1) (left) with $\code{rtol}=\code{atol}=10^{-3}$ and $\dt_0=0.01h$. Top: without relaxation. Middle: with relaxation (Regula Falsi) and initial $\dt=0.01h$.
				Bottom: Plots of $\dt$ and $\gamma$ for the run with relaxation.}\label{Fig:MPRK22_strat}
		\end{figure}

		\subsection{Linear advection}

		Consider the linear advection equation
		\begin{equation}
			\partial_t u + \partial_x u = 0
		\end{equation}
		with periodic boundary conditions on $I=[0,2]$ and a positive initial condition $u^0 > 0$.
		Then, the solution $u(t, x) = u^0(x - t)$ stays positive. Moreover, every
		functional of the form
		\begin{equation}
			\eta(u(t,I)) = \int_I U\bigl(u(t,x)\bigr) \dif x
		\end{equation}
		for an entropy function $U$ is conserved with associated entropy flux
		$F(u) = U(u)$. Following Tadmor \cite{Tadmor03}, the entropy variables
		are $w = U'(u)$ and the flux potential is $\psi = w f - F = U'(u) u - U(u)$.
		The corresponding entropy-conservative numerical flux is
		\begin{equation}
			f^{\mathrm{num}}(u_-, u_+)
			=
			\frac{\psi(u_+) - \psi(u_-)}{w(u_+) - w(u_-)}
			=
			\frac{U'(u_+) u_+ - U(u_+) - U'(u_-) u_- + U(u_-)}{U'(u_+) - U'(u_-)}.
		\end{equation}
		If $U'(u) \to \infty$ faster than $U(u)$ as $u \to 0$, then the numerical
		flux goes to zero if one of the states goes to zero.
		Therefore, the resulting finite volume method
		\begin{equation}
			\partial_t u_i + \frac{f^{\mathrm{num}}(u_{i-1}, u_i) - f^{\mathrm{num}}(u_i, u_{i+1})}{\Delta x} = 0
		\end{equation}
		is positivity-preserving in this case. In particular, the numerical fluxes
		$f^{\mathrm{num}}(u_-, u_+)$ are always non-negative,
		resulting in a conservative production-destruction system.
		Next, we consider several examples.
		\begin{itemize}
			\item The entropy
			\begin{equation}
				U(u) = u \log(u) - u
			\end{equation}
			leads to the entropy-conservative numerical flux
			\begin{equation}
				f^{\mathrm{num}}(u_-, u_+)
				=
				\frac{u_+ - u_-}{\log(u_+) - \log(u_-)}
				=:
				\{\!\{ u \}\!\}_\mathrm{log}
			\end{equation}
			using the logarithmic mean, see \cite[Section~3.2]{ranocha2021preventing}.

			\item Similarly, the entropy
			\begin{equation}
				U(u) = -\sqrt{u}
			\end{equation}
			leads to the entropy variables $w = -1 / (2 \sqrt{u})$, the flux potential
			$\psi = \sqrt{u} / 2$, and the entropy-conservative numerical flux
			\begin{equation}
				f^{\mathrm{num}}(u_-, u_+)
				=
				\frac{\sqrt{u_+} - \sqrt{u_-}}{-1 / \sqrt{u_+} + 1 / \sqrt{u_-}}
				=
				\sqrt{u_- u_+}
				=:
				\{\!\{ u \}\!\}_\mathrm{geo}
			\end{equation}
			using the geometric mean.

			\item Analogously, the entropy
			\begin{equation}
				U(u) = 1 / u
			\end{equation}
			leads to the entropy variables $w = -1 / u^2$, the flux potential
			$\psi = -2 / u$, and the entropy-conservative numerical flux
			\begin{equation}
				f^{\mathrm{num}}(u_-, u_+)
				=
				\frac{-2 / u_+ + 2 / u_-}{-1 / u_+^2 + 1 / u_-^2}
				=
				\frac{2 u_- u_+}{u_+ + u_-}
				=:
				\{\!\{ u \}\!\}_\mathrm{harm}
			\end{equation}
			using the harmonic mean.
		\end{itemize}

		Please note that positivity preservation for an entropy-conservative method
		depends on the choice of the entropy function. For example, the standard $L^2$ entropy
		\begin{equation}
			U(u) = \frac{u^2}{2}
		\end{equation}
		leads to the numerical flux
		\begin{equation}
			f^{\mathrm{num}}(u_-, u_+)
			=
			\frac{1}{2} \frac{u_+^2 - u_-^2}{u_+ - u_-}
			=
			\frac{u_- + u_+}{2},
		\end{equation}
		i.e., the standard arithmetic mean. The resulting finite volume discretization
		\begin{equation}
			\partial_t u_i + \frac{u_{i+1} - u_{i-1}}{2 \Delta x} = 0
		\end{equation}
		is the classical second-order central discretization, which is not positivity-preserving.

		We use $N_x=100$ cells and the initial condition
		\[u(0,x)=1.9\sin(\pi x)+2, \quad x\in[0,2]\]
		and apply different iterative methods for solving for $\gamma$. The respective results are depicted in Figure~\ref{Fig:MPRK22_LA}. 
		\begin{figure}[!htbp]
			\begin{subfigure}[t]{0.4\textwidth}
				\includegraphics[width=\textwidth]{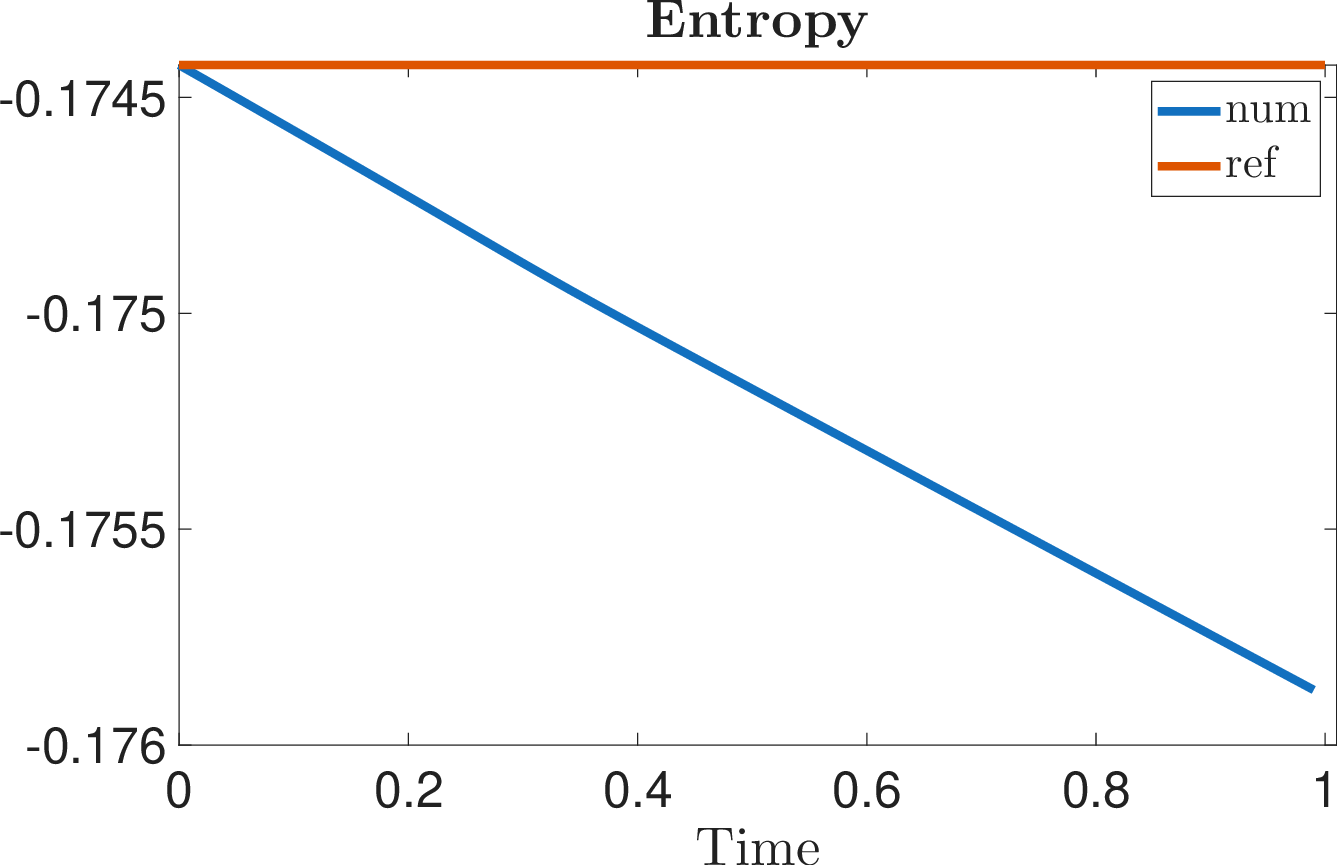}
				\subcaption{Logarithmic mean}
			\end{subfigure}
			\begin{subfigure}[t]{0.42\textwidth}
				\includegraphics[width=\textwidth]{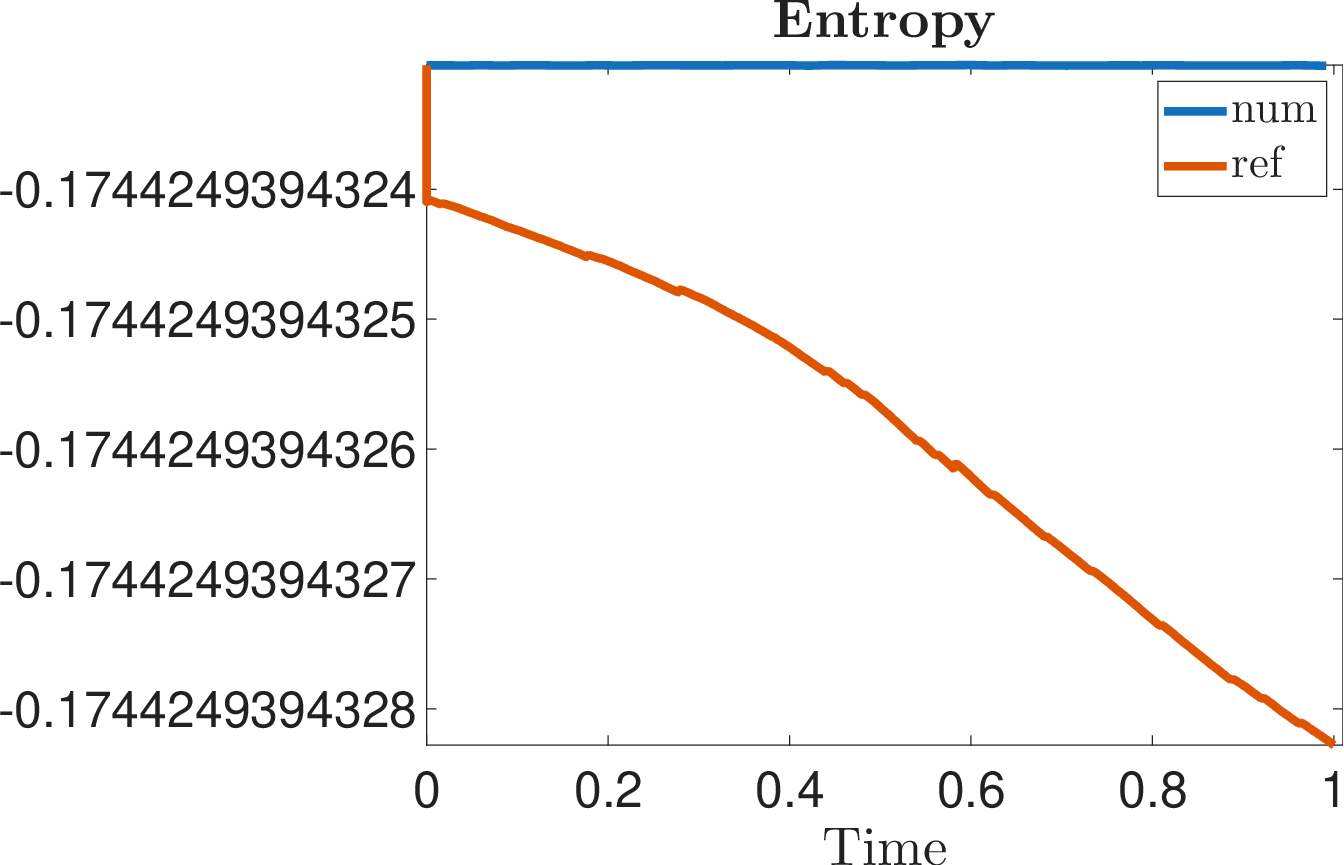}
				\subcaption{Logarithmic mean, secant method}
			\end{subfigure}
			\\
			\begin{subfigure}[t]{0.4\textwidth}
				\includegraphics[width=\textwidth]{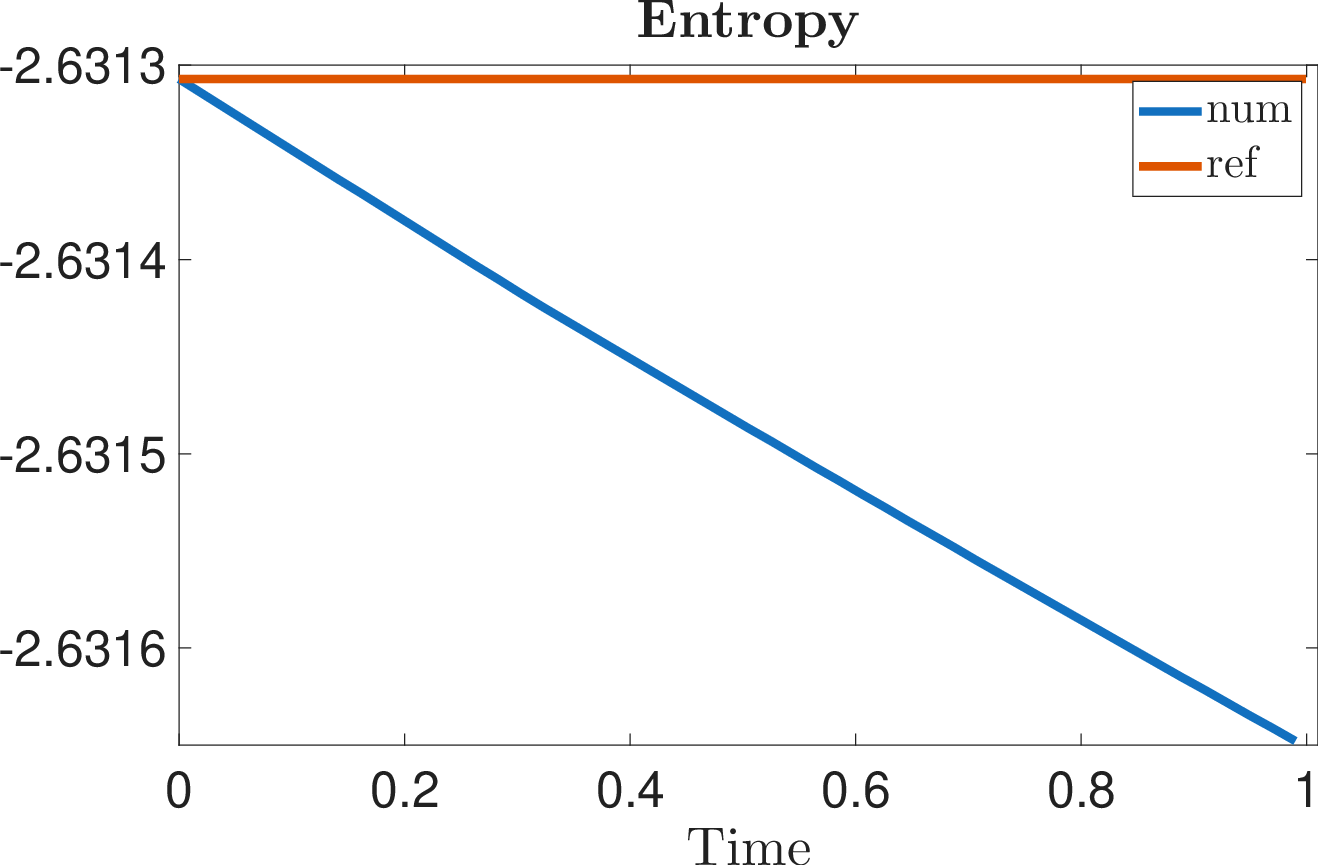}
				\subcaption{Geometric mean}
			\end{subfigure}
			\begin{subfigure}[t]{0.42\textwidth}
				\includegraphics[width=\textwidth]{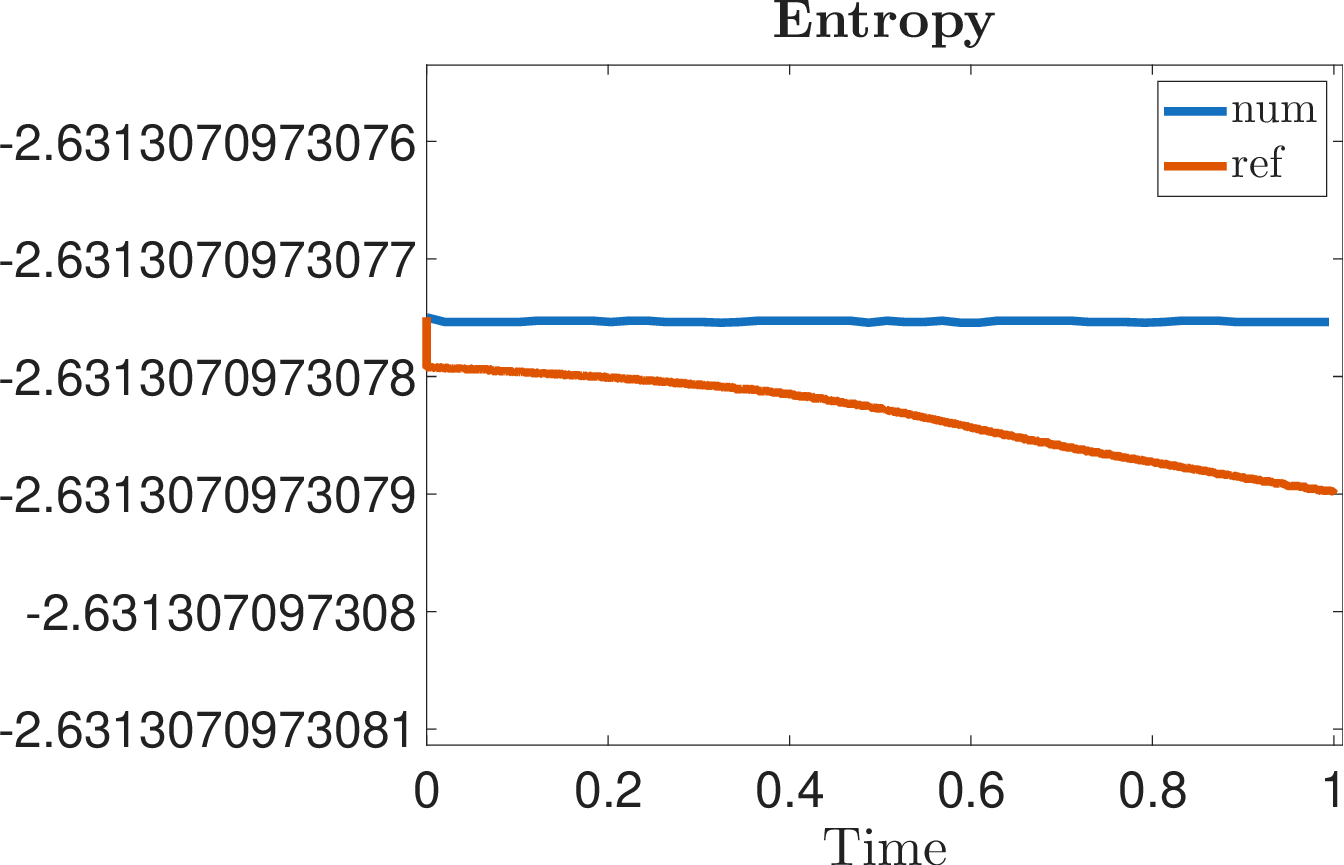}
				\subcaption{Geometric mean, regula falsi scheme}
			\end{subfigure}\\
			\begin{subfigure}[t]{0.4\textwidth}
				\includegraphics[width=\textwidth]{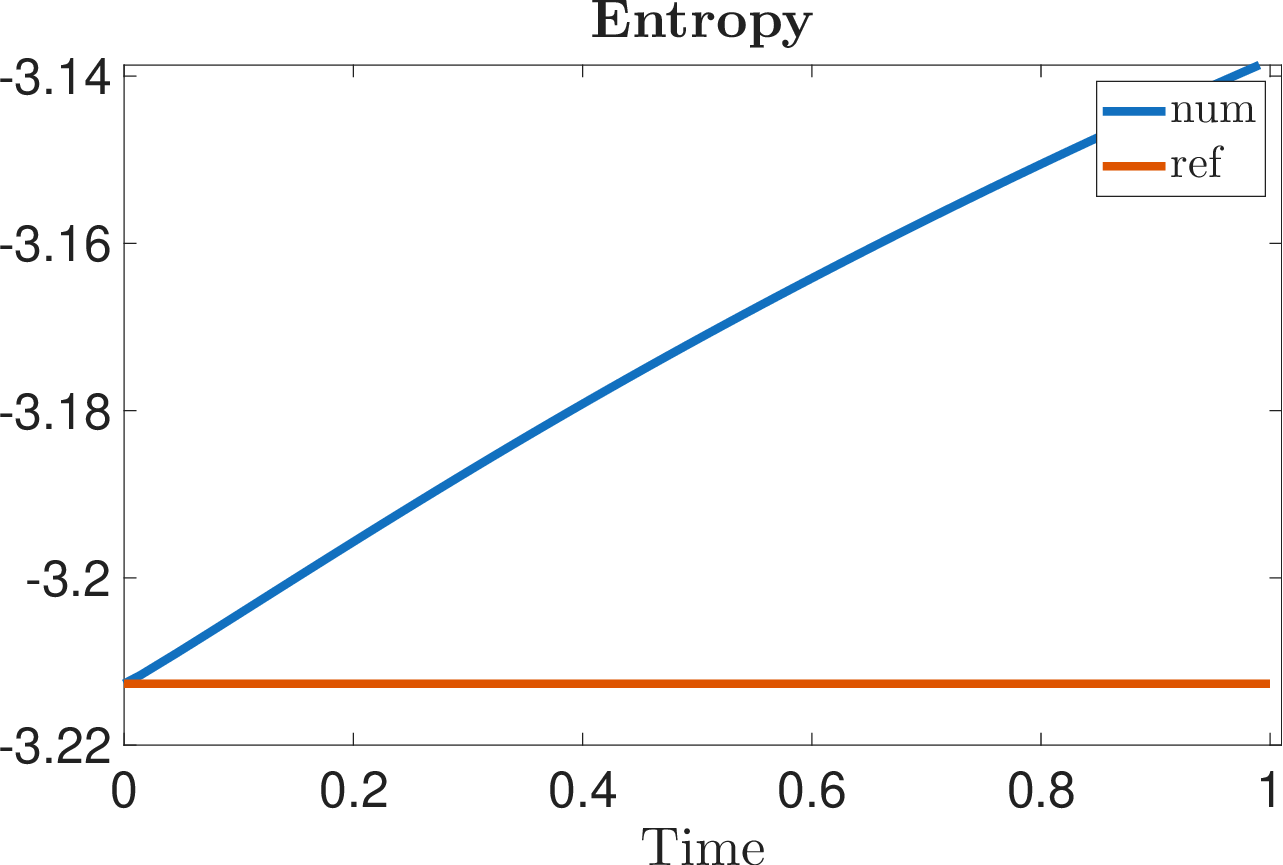}
				\subcaption{Harmonic mean}
			\end{subfigure}
			\begin{subfigure}[t]{0.42\textwidth}
				\includegraphics[width=\textwidth]{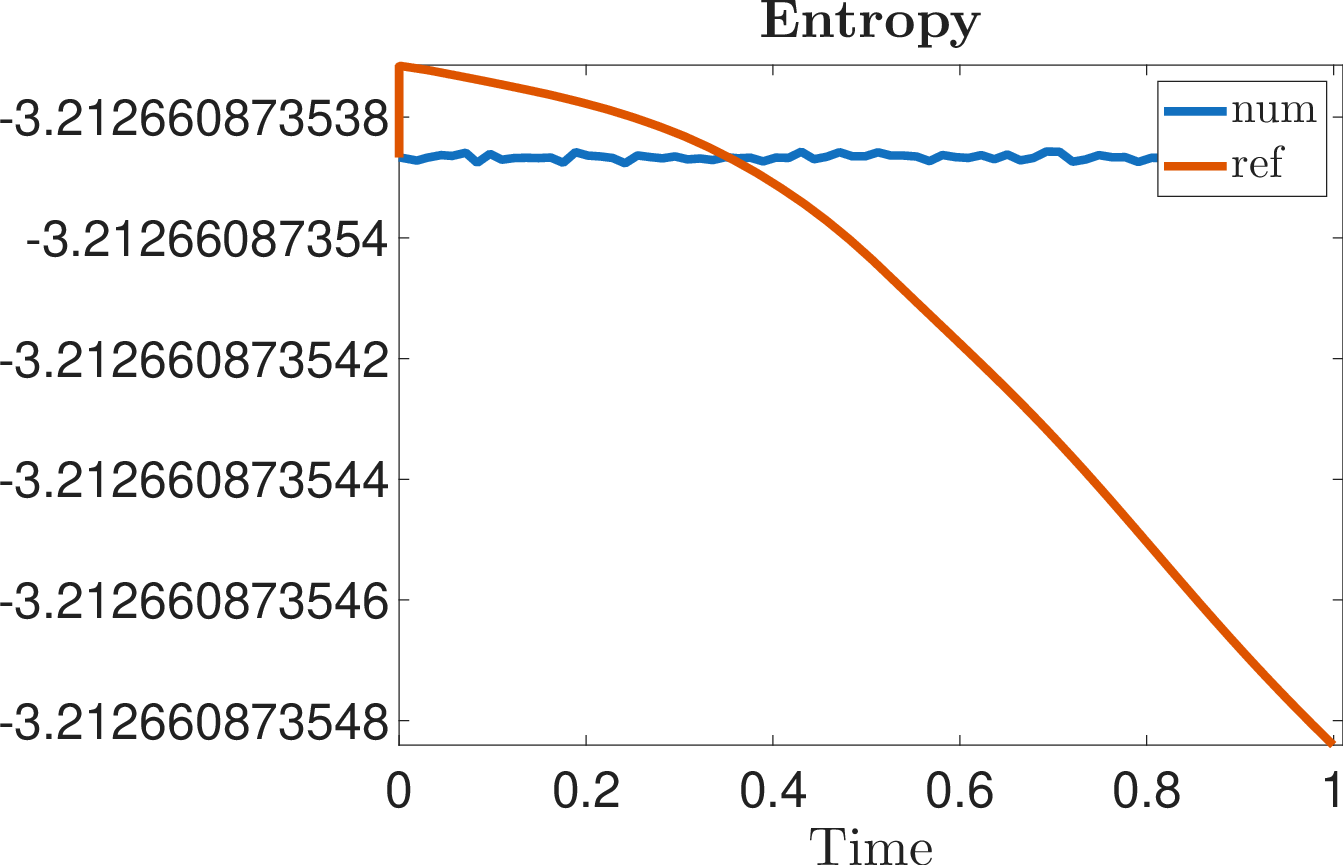}
				\subcaption{Harmonic mean, bisection method}
			\end{subfigure}
				\caption{Numerical solution of linear advection equation using MPSSPRK(0.5,1) ((a)-(b)), MPRK43adap(0.5,0.75) ((c)-(d)), and MPRK22adap(1) ((e)-(f)) with $\code{rtol}=\code{atol}=10^{-3}$, $N=100$, and $\dt_0=\dx$. Left: without relaxation. Right: with relaxation.}\label{Fig:MPRK22_LA}
			\end{figure}

			\subsection{Shallow water equations}

			The classical shallow water equations
			\begin{equation}
				\partial_t \underbrace{\begin{pmatrix} h \\ h v \end{pmatrix}}_{= u}
				+ \partial_x \underbrace{\begin{pmatrix} h v \\ h v^2 + \tfrac12 g h^2 \end{pmatrix}}_{= f(u)}
				=
				0
			\end{equation}
			have the total energy
			\begin{equation}
				U(u) = \tfrac12 h v^2 + \tfrac12 g h^2
			\end{equation}
			as entropy. The corresponding entropy variables are
			\begin{equation}
				w = \begin{pmatrix} g h - \tfrac12 v^2 \\ v \end{pmatrix}
			\end{equation}
			and the entropy flux potential is
			\begin{equation}
				\psi = \tfrac12 g h^2 v.
			\end{equation}
			For constant velocity $v$, the condition for an entropy-conservative numerical flux is
			\begin{equation}
				0
				=
				[\![ w ]\!] \cdot f^{\mathrm{num}} - [\![ \psi ]\!]
				=
				g [\![ h ]\!] f^{\mathrm{num}}_{h} - \tfrac12 g [\![ h^2 ]\!] v,
			\end{equation}
			where 	$[\![ w ]\!]\coloneqq w_{i+1}-w_i$.
			Thus, the numerical flux for the water height $h$ is
			\begin{equation}
				f^{\mathrm{num}}_{h}
				=
				\tfrac12 \frac{[\![ h^2 ]\!]}{[\![ h ]\!]} v
				=
				\{\!\{ h \}\!\} v
			\end{equation}
			with the arithmetic mean
			\begin{equation}
				\{\!\{ h \}\!\}
				=
				\tfrac12 (h_- + h_+).
			\end{equation}
			Similarly to the linear advection equations above, the arithmetic mean does not lead
			to a positivity-preserving semidiscretization. This proves
			\begin{theorem}
				An entropy-conservative semidiscretization of the shallow water equations
				is not positivity-preserving.
			\end{theorem}

			One can show a similar result for the polytropic Euler equations with
			pressure $p \propto \rho^\gamma$ and $\gamma > 1$. However, the limiting case of
			the isothermal Euler equations is different and discussed in the next subsection.

			\subsection{Isothermal Euler equations}

			The 1D isothermal Euler equations are
			\begin{equation}
				\partial_t \underbrace{\begin{pmatrix} \rho \\ \rho v \end{pmatrix}}_{=\b u}
				+ \partial_x \underbrace{\begin{pmatrix} \rho v \\ \rho v^2 + c^2 \rho \end{pmatrix}}_{=\b f(\b u)}
				=
				0,
			\end{equation}
			where $\rho$ is the density, $v$ is the velocity, and $c$ is the speed of sound.
			We take the total energy
			\begin{equation}
				U(\b u) = \tfrac12 \rho v^2 + \tfrac12 c^2 \rho \log(\rho)
			\end{equation}
			as (mathematical) entropy. An associated entropy-conservative numerical flux at the interface $i+\frac12$ is
			given by \cite{winters2020entropy}
			\begin{equation}
				f^{\mathrm{num}}_{\rho}
				=
				\{\!\{ \rho \}\!\}_\mathrm{log} \{\!\{ v \}\!\},
				\quad
				f^{\mathrm{num}}_{\rho v}
				=
				\{\!\{ \rho \}\!\}_\mathrm{log} \{\!\{ v \}\!\}^2 + \{\!\{ c^2 \rho \}\!\}.
			\end{equation}
			Since the logarithmic mean goes to zero if one of the states goes to zero, the
			resulting entropy-conservative finite volume method is unconditionally positive.
				Even more general, the flux differencing method \cite{tadmor1987numerical,Tadmor03,lefloch2002fully,fisher2013discretely,ranocha2018comparison,chen2017entropy} based on diagonal-norm SBP operators. In particular,
			high-order discontinuous Galerkin spectral element methods (DGSEMs) are positivity-preserving.
			While we apply the underlying explicit RK method to the second conserved variable together with the standard relaxation algorithm, we use MPRK22 for $\rho$, where we use the PDS
			\begin{equation*}
				p_{i+1,i}=d_{i,i+1}=\max\left\{0,	f^{\mathrm{num}}_{\rho}\right\}, \quad p_{i,i+1}=d_{i+1,i}=-\min\left\{0,	f^{\mathrm{num}}_{\rho}\right\}
			\end{equation*}
			for $i=1,\dotsc,N-1,$ and we take periodic boundary conditions into account for the terms if $i=N$. The study of flux-balanced MPRK schemes introduced in \cite{IMST2026} is left for future works.
			In Figure~\ref{Fig:MPRK22_IE} we solve the Riemann problem (RP)
			\[\b u_L=(0.8,10^{-3}),\quad \b u_R=(1,10^{-2})\]
			with periodic boundary conditions and final time $\tend=1$. We note that one should not use an entropy conservative flux for an RP, however, this is a good example that our time integrator maintains the entropy properties of the space discretization.
			\begin{figure}[!htbp]
				\begin{subfigure}[t]{0.5\textwidth}
					\includegraphics[width=\textwidth]{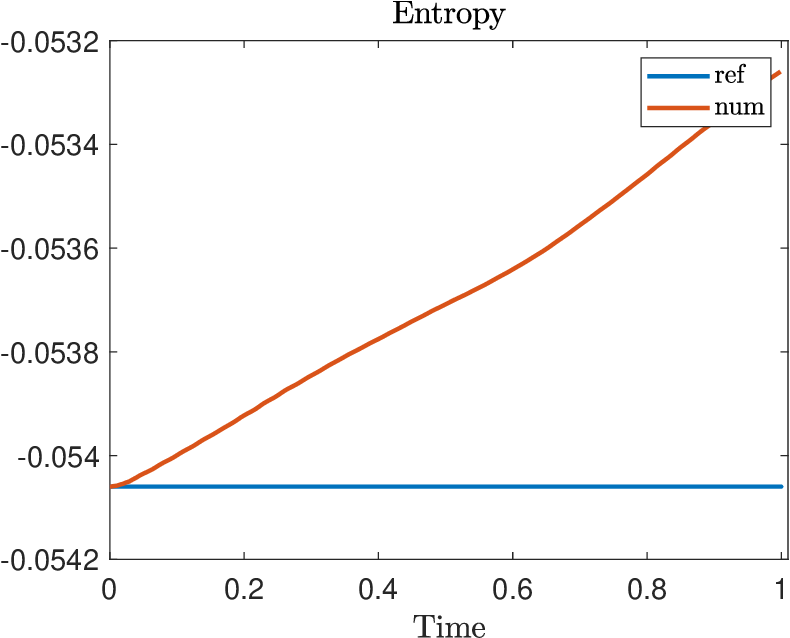}
				\end{subfigure}
				\begin{subfigure}[t]{0.5\textwidth}
					\includegraphics[width=\textwidth]{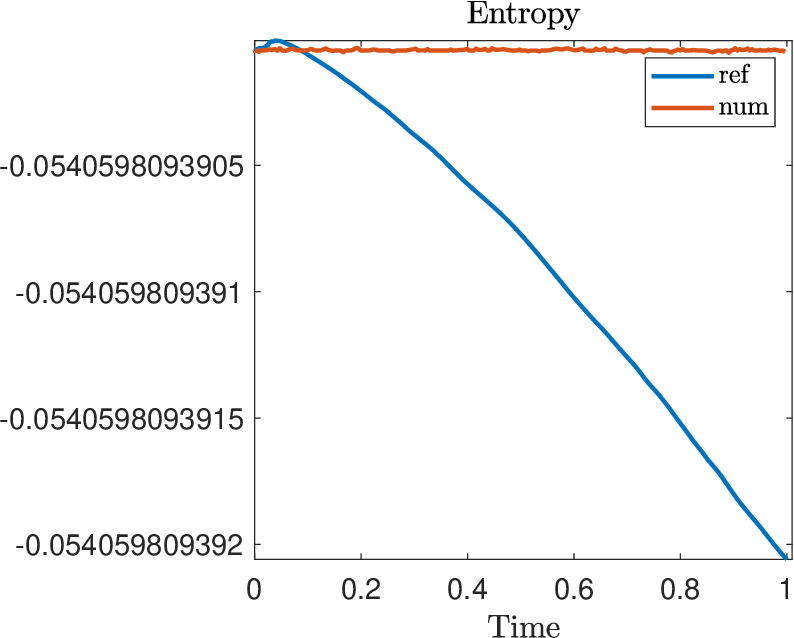}
				\end{subfigure}\\
				\begin{center}
					\begin{subfigure}[t]{0.5\textwidth}
						\includegraphics[width=\textwidth]{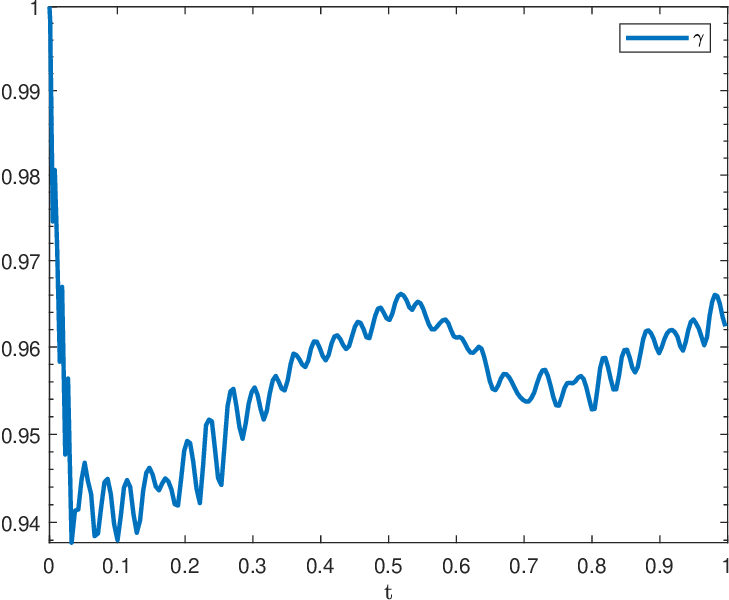}
					\end{subfigure}
				\end{center}
				\caption{Numerical solution of isothermal Euler equations with $N_x=100$ using MPRK22adap(1) with $\code{rtol}=\code{atol}=10^{-3}$ and $\dt=\dx$. Left: without relaxation. Right: with relaxation.} \label{Fig:MPRK22_IE}
			\end{figure}

			\subsection{Porous Medium Equation}
			The porous medium equation
			\[u_t = (u^m)_{xx} = (a(u) u_x)_x, \quad a(u) = m u^{m-1} \]
			with a free parameter $m>1$, see for instance \cite{Boscarino23}, admits a non-negative weak solution
			\[u^{(m)}(t,x)=t^{-k}\left[\max\left(1-\frac{k(m-1)}{2m}\frac{\lvert x\rvert^2}{t^{2k}},0\right)\right]^{\frac{1}{m-1}}\]
			with $k=\frac{1}{m+1}$, the so-called \emph{Barenblatt} solution \cite{Barenblatt52}. For every $t>0$, the solution has a compact support $[-\alpha_m(t),\alpha_m(t)]$ where
			\[\alpha_m(t)=\sqrt{\frac{2m}{k(m-1)}}t^k.\]
			We follow \cite{Boscarino23} using $u(0,x)=u^{(m)}(1,x)$ as an initial condition. We plot the numerical solution at time $t=2$ on the spatial domain $[-6,6]$ using the boundary conditions $u(t,\pm 6)=0$ for $t>1$.

			We use the second-order space discretization from \cite{Mattsson12,ranocha2019mimetic} given by
			\begin{align*}
				f_i(\b u(t))=&\frac{a(u_i(t))+a(u_{i+1}(t))}{2\dx^2} u_{i+1}(t)\\&-\frac{a(u_{i-1}(t))+2a(u_i(t))+a(u_{i+1}(t))}{2\dx^2}u_i(t)\\&+\frac{a(u_{i-1}(t))+a(u_i(t))}{2\dx^2} u_{i-1}(t)
			\end{align*}
			for $i=2,\dotsc,N$ and
			\[f_j(\b u(t))= \frac{a(u_j(t))}{2\dx^2}u_j(t),\quad \text{ for }\quad j\in \{1,N\}.\]
			Next, we consider the convex entropy
			\[\eta(\b u)=\frac{\dx^2}2\sum_{i=1}^{N_x}u_i^2, \]
			which satisfies
			\begin{equation*}
				\begin{aligned}
					\frac{\dd}{\dd t}\eta(\b u(t))\leq 0
				\end{aligned}
			\end{equation*}
			for the boundary conditions mentioned above, see \cite[Theorem~4.1]{ranocha2019mimetic}.
			This system of ODEs may be rewritten as a conservative PDS by setting
			\begin{equation*}
				\begin{aligned}
					p_{i,i+1}(\b u)&=\frac{a(u_i)+a(u_{i+1})}{2\dx^2} u_{i+1},&\quad p_{i,i-1}(\b u)&= \frac{a(u_{i-1})+a(u_i)}{2\dx^2} u_{i-1},&\quad i&=2,\dotsc,N,\\
					p_{1,2}(\b u)&= \frac{a(u_2)}{2\dx^2}u_2,&\quad p_{N,N-1}(\b u)&= \frac{a(u_{N-1})}{2\dx^2}u_{N-1},&\quad d_{i,j}&=p_{j,i}.
				\end{aligned}
			\end{equation*}

			According to \cite{Boscarino23}, the cases $m=3,5$ are particularly interesting as the numerical solution of the proposed third-order IMEX method in \cite[p.~10, eq.~(30)]{Boscarino23} generates negative approximations and which cannot happen with MPRK schemes. Indeed, we observe in Figure~\ref{Fig:MPRK22_PME} that we obtain positive approximations while the relaxation algorithm gives us an entropy estimate. Here, we do not plot $\gamma$ as it was constantly at $1$.
			\begin{figure}[!htbp]
				\begin{subfigure}[t]{0.5\textwidth}
					\includegraphics[width=\textwidth]{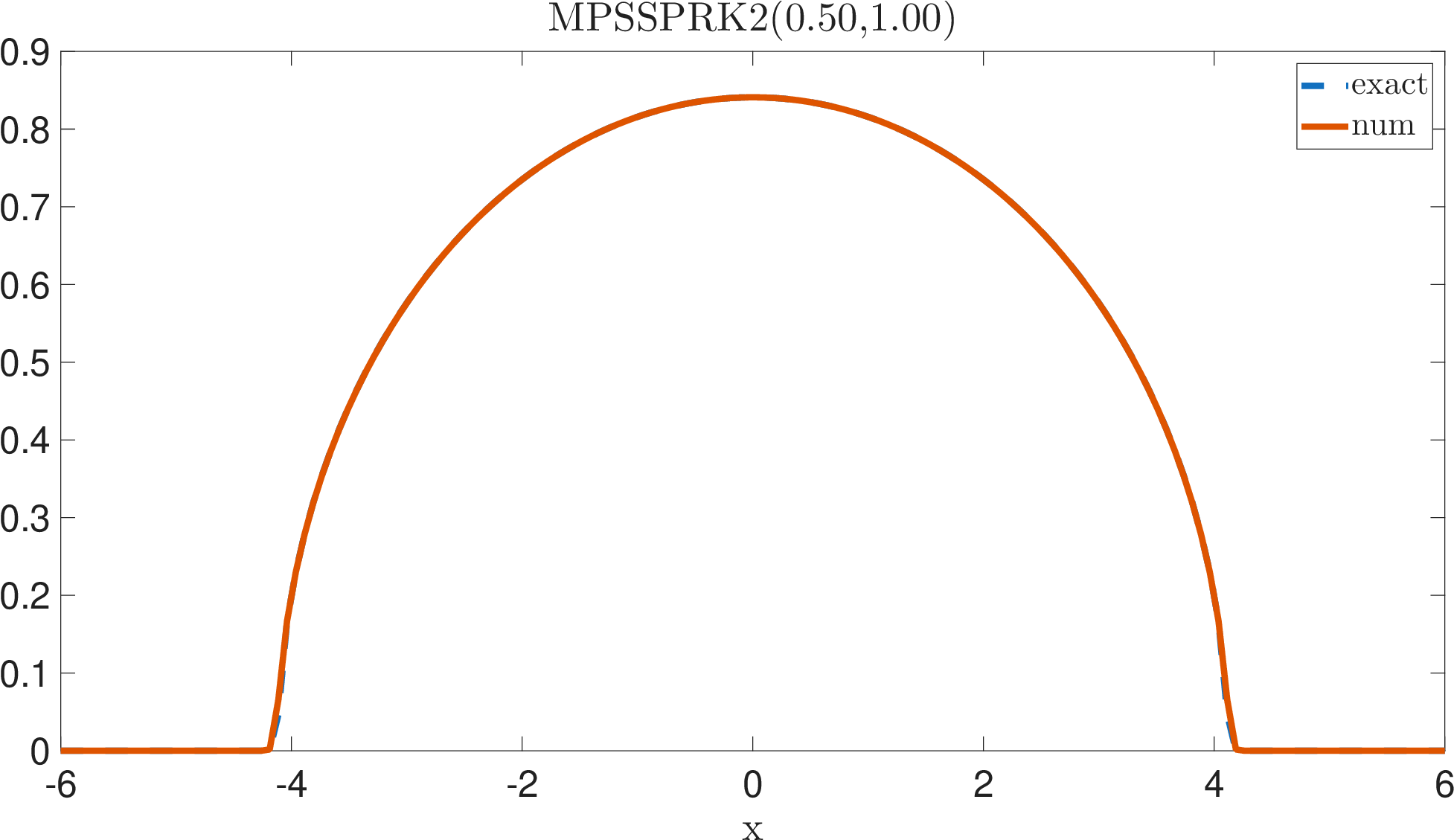}
					\subcaption{$m=3$}
				\end{subfigure}
				\begin{subfigure}[t]{0.5\textwidth}
					\includegraphics[width=\textwidth]{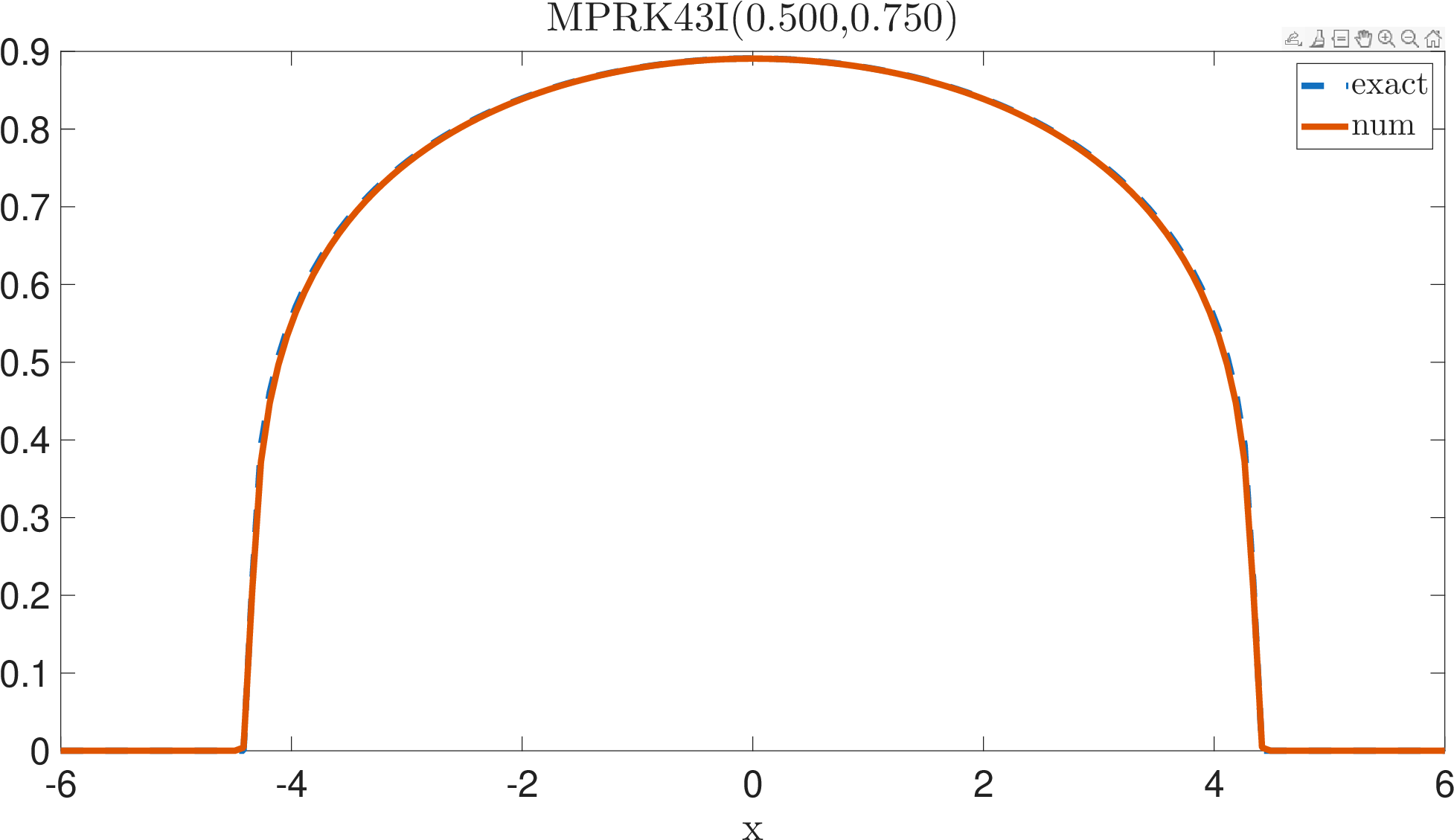}
					\subcaption{$m=5$}
				\end{subfigure}\\
				\begin{subfigure}[t]{0.5\textwidth}
					\includegraphics[width=\textwidth]{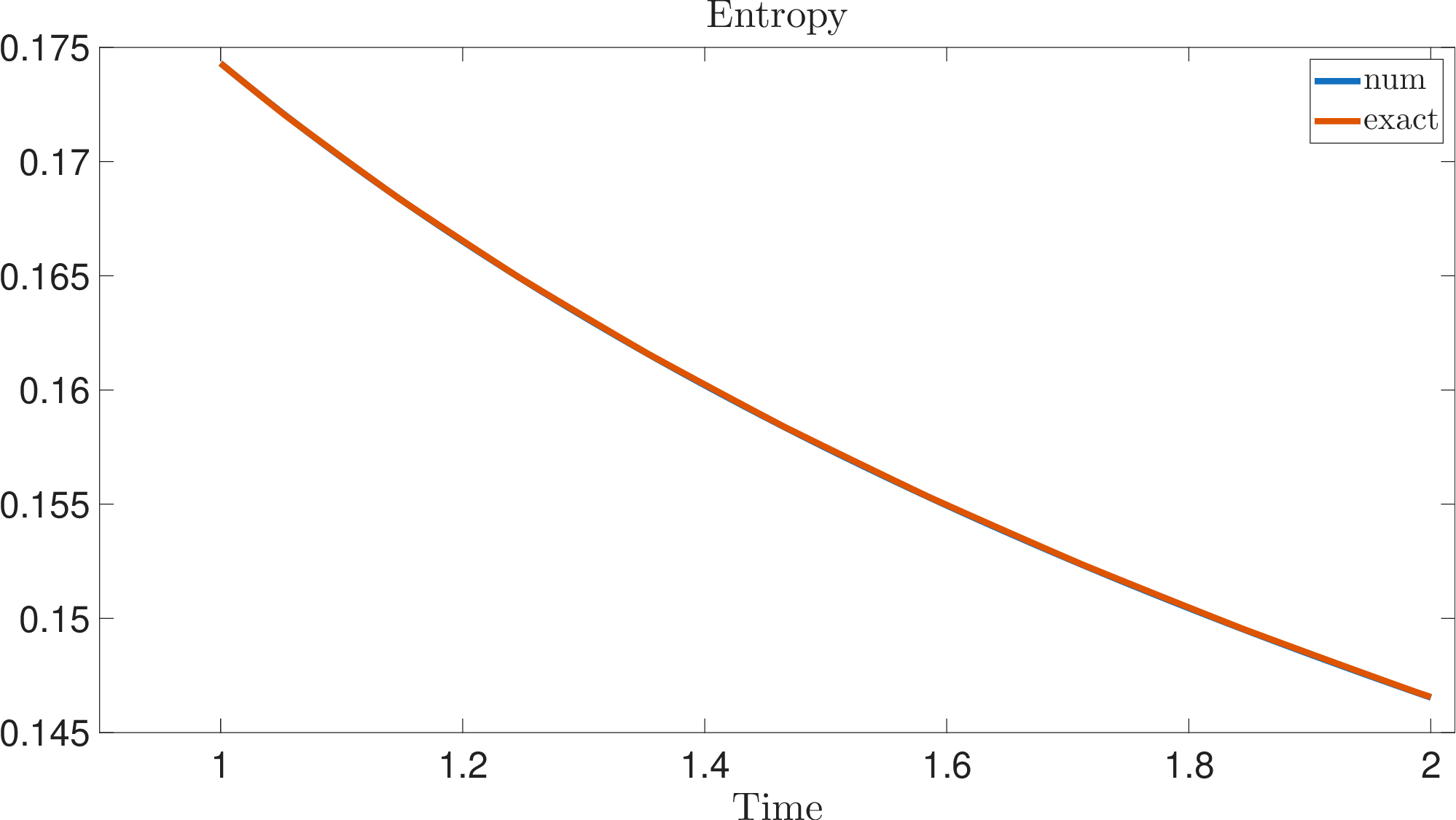}
					\subcaption{$m=3$}
				\end{subfigure}
				\begin{subfigure}[t]{0.495\textwidth}
					\includegraphics[width=\textwidth]{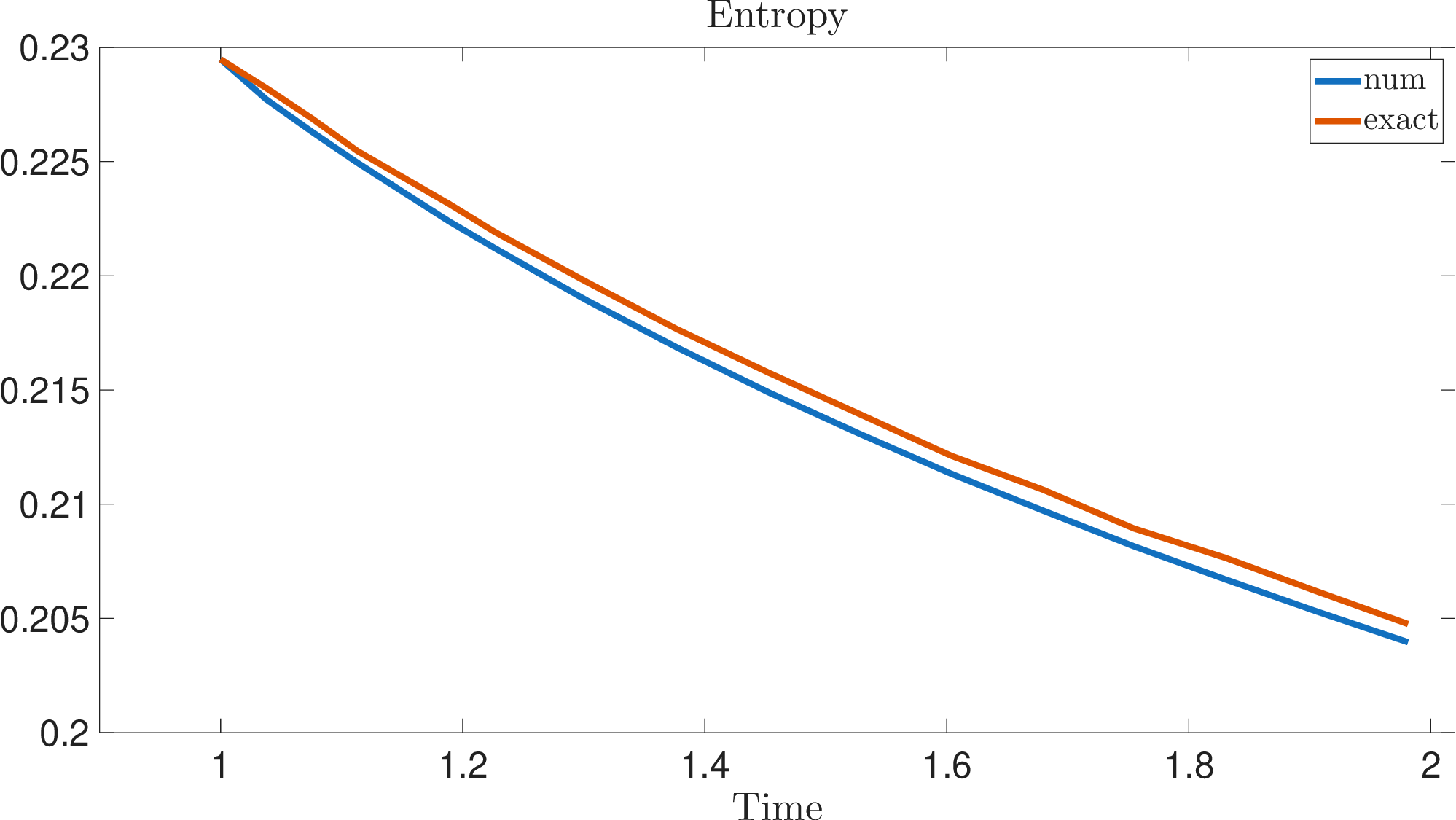}
					\subcaption{$m=5$}
				\end{subfigure}
				\caption{Numerical solution of PME equation with $N_x=160$ using MPSSPRK22(0.5,1) (left) and MPRK43(0.5,0.75) (right) with $\dt=\dx$ and relaxation.} \label{Fig:MPRK22_PME}
			\end{figure}
			\section{Summary and conclusions}
			In this work we investigated non-standard additive Runge--Kutta (NSARK) schemes, which include modified Patankar (MP) methods or Geometric Conservative (GeCo) to name a few. Being particularly interested in positivity-preserving methods that are also capable of conserving at least one linear invariant, we answered the question of whether these schemes can be equipped with a relaxation technique that preserves these properties while ensuring entropy stability. We point out that positivity preservation is easy to accomplish for entropy dissipative problems. For entropy conservative problems, where no linear invariant needs to be preserved, one can equip an unconditionally positive method with the geometric mean to compute the relaxation update. If the conservative problem has a linear invariant or one is interested in keeping a conservative PDS part also conservative within the relaxation procedure, we propose to use a linearly implicit formula for the relaxation update, which in turn results in a coupled linear-nonlinear system for the simultaneous computation of $\gamma$ and $\bunplusgamma$. All techniques can be used for any positivity-preserving method maintaining the order, however, the latter relaxation technique involves a bootstrapping algorithm to achieve higher-order entropy conservative methods preserving a linear invariant.

			We have tested our theoretical findings by means of multiple examples of ordinary and partial differential equations.
			Furthermore, interpreting a linear invariant as entropy, we were able to preserve both linear invariants of the stiff stratospheric reaction problem using MPRK. We have also tested several flux and entropy pairs for the linear advection equation testing the different iterative solvers for the computation of $\gamma$ and $\bunplusgamma$. Moreover, we applied our technique also in the context of the isothermal Euler equation guaranteeing the positivity of the density. Finally, we have also tested MPRK and MPSSPRK schemes with the entropy dissipative porous medium equation, where we are also able to avoid negative approximations.

			Future research topics include the testing of further NSARK schemes, including the recently developed flux-balanced versions, and the efficiency of the related methods. As some of the NSARK schemes are already proven to be conditionally stable, a thorough stability analysis for these methods is also part of ongoing research.

			\bmhead{Acknowledgements}

			\section*{Declarations}

			\paragraph{Funding}
			T.\ Izgin gratefully acknowledges the financial support by Fulbright Germany.
			H.\ Ranocha was supported by the Deutsche Forschungsgemeinschaft
			(DFG, German Research Foundation, project number 513301895)
			and the Daimler und Benz Stiftung (Daimler and Benz foundation,
			project number 32-10/22).  C.-W. Shu acknowledges partial support from NSF grant DMS-2309249.
			\paragraph{Conflict of interest}
			Not applicable.
			\paragraph{Ethics approval and consent to participate}
			Not applicable.

			\paragraph{Consent for publication}
			All authors consent to submit for publication.

			\paragraph{Data, Materials and Code availability}
			The source code used in this study is available at \cite{IRS2026repository}.

			\paragraph{Author contribution}
			All authors contributed to the study conception and design. Material preparation, data collection and analysis were performed by Thomas Izgin. The first draft of the manuscript was written by Thomas Izgin and all authors commented on as well as wrote on previous versions of the manuscript. All authors read and approved the final manuscript.


			\begin{appendices}

			\end{appendices}


			\bibliography{references}

@Article{KM18Order3,
  author     = {Kopecz, Stefan and Meister, Andreas},
  journal    = {BIT},
  title      = {Unconditionally positive and conservative third order modified {P}atankar-{R}unge-{K}utta discretizations of production-destruction systems},
  year       = {2018},
  issn       = {0006-3835},
  number     = {3},
  pages      = {691--728},
  volume     = {58},
  fjournal   = {BIT. Numerical Mathematics},
  mrclass    = {65L06 (65L05 65L20)},
  mrnumber   = {3855688},
  mrreviewer = {Luigi Brugnano},
  publisher  = {Springer},
  doi        = {10.1007/s10543-018-0705-1},
}

@article{BDM2003,
	title        = {A high-order conservative {P}atankar-type discretisation for stiff systems of production-destruction equations},
	author       = {Burchard, Hans and Deleersnijder, Eric and Meister, Andreas},
	year         = 2003,
	journal      = {Appl. Numer. Math.},
	volume       = 47,
	number       = 1,
	pages        = {1--30},
	doi          = {10.1016/S0168-9274(03)00101-6},
	issn         = {0168-9274,1873-5460},
	fjournal     = {Applied Numerical Mathematics. An IMACS Journal},
	mrclass      = {65L06 (65L05)},
	mrnumber     = 2003144,
	mrreviewer   = {Gabriela\ Schranz-Kirlinger}
}

@article{Tadmor03, title={Entropy stability theory for difference approximations of nonlinear conservation laws and related time-dependent problems}, volume={12}, DOI={10.1017/S0962492902000156}, journal={Acta Numerica}, author={Tadmor, Eitan}, year={2003}, pages={451–512}}

@article {MPDeC,
	AUTHOR = {{\"O}ffner, P. and Torlo, D.},
	TITLE = {Arbitrary high-order, conservative and positivity preserving
	{P}atankar-type deferred correction schemes},
	JOURNAL = {Appl. Numer. Math.},
	FJOURNAL = {Applied Numerical Mathematics. An IMACS Journal},
	VOLUME = {153},
	YEAR = {2020},
	PAGES = {15--34},
	ISSN = {0168-9274},
	MRCLASS = {65L06},
	MRNUMBER = {4064785},
	MRREVIEWER = {Helmut Podhaisky},
	DOI = {},
	URL = {https://doi.org/10.1016/j.apnum.2020.01.025},
}

@article {SSPMPRK3,
	AUTHOR = {Huang, J.  and Zhao, W. and Shu, C.-W.},
	TITLE = {A third-order unconditionally positivity-preserving scheme for
	production-destruction equations with applications to
	non-equilibrium flows},
	JOURNAL = {J. Sci. Comput.},
	FJOURNAL = {Journal of Scientific Computing},
	VOLUME = {79},
	YEAR = {2019},
	NUMBER = {2},
	PAGES = {1015--1056},
	ISSN = {0885-7474},
	MRCLASS = {65M06 (76Nxx 76V05)},
	MRNUMBER = {3969000},
	DOI = {},
	URL = {https://doi.org/10.1007/s10915-018-0881-9},
}

@article {SSPMPRK2,
	AUTHOR = {Huang, J.  and Shu, C.-W.},
	TITLE = {Positivity-preserving time discretizations for
	production-destruction equations with applications to
	non-equilibrium flows},
	JOURNAL = {J. Sci. Comput.},
	FJOURNAL = {Journal of Scientific Computing},
	VOLUME = {78},
	YEAR = {2019},
	NUMBER = {3},
	PAGES = {1811--1839},
	ISSN = {0885-7474},
	MRCLASS = {65M06 (65L06 65M20 76V05)},
	MRNUMBER = {3934688},
	MRREVIEWER = {Qifeng Zhang},
	DOI = {},
	URL = {https://doi.org/10.1007/s10915-018-0852-1},
}

@article {Boscarino23,
	AUTHOR = {Boscarino, Sebastiano},
	TITLE = {High-order semi-implicit schemes for evolutionary partial
	differential equations with higher order derivatives},
	JOURNAL = {J. Sci. Comput.},
	FJOURNAL = {Journal of Scientific Computing},
	VOLUME = {96},
	YEAR = {2023},
	NUMBER = {1},
	PAGES = {Paper No. 11, 31},
	ISSN = {0885-7474,1573-7691},
	MRCLASS = {65M20 (35G25 65L05 65L06 65M06)},
	MRNUMBER = {4593278},
	DOI = {10.1007/s10915-023-02235-0},
	URL = {https://doi.org/10.1007/s10915-023-02235-0},
}

@article {Barenblatt52,
	AUTHOR = {Barenblatt, G. I.},
	TITLE = {On self-similar motions of a compressible fluid in a porous
	medium},
	JOURNAL = {Akad. Nauk SSSR. Prikl. Mat. Meh.},
	FJOURNAL = {Akad. Nauk SSSR. Prikl. Mat. Meh.},
	VOLUME = {16},
	YEAR = {1952},
	PAGES = {679--698},
	MRCLASS = {76.1X},
	MRNUMBER = {52948},
	MRREVIEWER = {R.\ E.\ Gaskell},
}

@book{Patankar1980,
	title        = {Numerical heat transfer and fluid flow},
	author       = {Patankar, S. V.},
	year         = 1980,
	publisher    = {Hemisphere Pub. Corp. New York},
	address      = {Washington},
	series       = {{Series in computational methods in mechanics and thermal sciences}},
	isbn         = {0-07-048740-5},
	url          = {http://opac.inria.fr/record=b1085925}
}

@article{IMPV2025,
	title        = {Modified Patankar Linear Multistep methods for production-destruction systems},
	author       = {Izzo, Giuseppe and Messina, Eleonora and Pezzella, Mario and Vecchio, Antonia},
	year         = 2025,
	journal      = {Journal of Scientific Computing},
	publisher    = {Springer},
	volume       = 102,
	number       = 3,
	pages        = 87
}

@article{izgin2024,
	title={A Boot-Strapping Technique to Design Dense Output Formulae for Modified Patankar-Runge-Kutta Methods},
	author={Thomas Izgin},
	year={2024},
	eprint={2406.16718},
	archivePrefix={arXiv},
	primaryClass={math.NA},
	journal={https://arxiv.org/abs/2406.16718},
}

@article{IMST2026,
	title={Flux-Balanced Patankar-type Schemes for the Compressible Euler Equations},
	author={Izgin, Thomas and Meister, Andreas and Shu, Chi-Wang and Torlo, Davide},
	journal={arXiv preprint arXiv:2602.14392},
	year={2026},
	url={https://arxiv.org/abs/2602.14392},
	note={Available at arXiv:2602.14392 [math.NA]}
}

@book{boyd2004convex,
	title={Convex Optimization},
	author={Boyd, Stephen and Vandenberghe, Lieven},
	year={2004},
	publisher={Cambridge University Press},
	address={Cambridge, UK},
	isbn={978-0521833783}
}

@article{jungel2015entropy,
  title={Entropy dissipative one-leg multistep time approximations of nonlinear
         diffusive equations},
  author={J{\"u}ngel, Ansgar and Mili{\v{s}}i{\'c}, Josipa-Pina},
  journal={Numerical Methods for Partial Differential Equations},
  volume={31},
  number={4},
  pages={1119--1149},
  year={2015},
  publisher={Wiley Online Library},
  doi={10.1002/num.21938}
}

@article{jungel2017entropy,
  title={Entropy-dissipating semi-discrete {R}unge-{K}utta schemes for
         nonlinear diffusion equations},
  author={J{\"u}ngel, Ansgar and Schuchnigg, Stefan},
  journal={Communications in Mathematical Sciences},
  volume={15},
  number={1},
  pages={27--53},
  year={2017},
  publisher={International Press of Boston},
  doi={10.4310/CMS.2017.v15.n1.a2}
}

@article {sandu2001positive,
    AUTHOR = {Sandu, A.},
     TITLE = {Positive numerical integration methods for chemical kinetic
              systems},
   JOURNAL = {J. Comput. Phys.},
  FJOURNAL = {Journal of Computational Physics},
    VOLUME = {170},
      YEAR = {2001},
    NUMBER = {2},
     PAGES = {589--602},
      ISSN = {0021-9991},
   MRCLASS = {80M20 (65L99 80A30)},
  MRNUMBER = {1844904},
       DOI = {10.1006/jcph.2001.6750},
       URL = {https://doi.org/10.1006/jcph.2001.6750},
}

@article{BBKS2007,
	author = {Bruggeman, J. and Burchard, H. and Kooi, B. W. and Sommeijer, B.},
	doi = {10.1016/j.apnum.2005.12.001},
	issn = {0168-9274},
	journal = {Applied Numerical Mathematics},
	note = {},
	number = {1},
	pages = {36--58},
	title = {{A second-order, unconditionally positive, mass-conserving integration scheme for biochemical systems}},
	url = {http://www.sciencedirect.com/science/article/pii/S0168927405002242},
	volume = {57},
	year = {2007}
}

@article{HR2020,
	title        = {Space-time residual distribution on moving meshes},
	author       = {Hubbard, M. E. and {Ricchiuto}, M. and S\'arm\'any, D.},
	year         = 2020,
	journal      = {Comput. Math. Appl.},
	volume       = 79,
	number       = 5,
	pages        = {1561--1589},
	doi          = {10.1016/j.camwa.2019.09.019},
	issn         = {0898-1221,1873-7668},
	url          = {https://doi.org/10.1016/j.camwa.2019.09.019},
	fjournal     = {Computers \& Mathematics with Applications. An International Journal},
	mrclass      = {65M08 (65M50 76B15)},
	mrnumber     = 4065801,
	mrreviewer   = {Jean-Pierre\ Croisille}
}

@techreport{ricchiuto2011habilitation,
	title        = {Contributions to the development of residual discretizations for hyperbolic conservation laws with application to shallow water flows},
	author       = {{Ricchiuto}, Mario},
	year         = 2011,
	month        = 12,
	address      = {Bordeaux, France},
	institution  = {Universit{\'e} Sciences et Technologies-Bordeaux I},
	type         = {Habilitation thesis}
}

@article{BIM2022,
	title        = {Positivity-preserving methods for ordinary differential equations},
	author       = {{Blanes, Sergio} and {Iserles, Arieh} and {Macnamara, Shev}},
	year         = 2022,
	journal      = {ESAIM: M2AN},
	volume       = 56,
	number       = 6,
	pages        = {1843--1870},
	doi          = {10.1051/m2an/2022042},
	url          = {https://doi.org/10.1051/m2an/2022042}
}

@book {GKS2011,
	AUTHOR = {Gottlieb, Sigal and Ketcheson, David and Shu, Chi-Wang},
	TITLE = {Strong stability preserving {R}unge-{K}utta and multistep time
	discretizations},
	PUBLISHER = {World Scientific Publishing Co. Pte. Ltd.},
	address = {Hackensack, NJ},
	YEAR = {2011},
	DOI = {10.1142/7498}
}

@article{Shampine1986,
	title = {Conservation laws and the numerical solution of ODEs},
	journal = {Computers \& Mathematics with Applications},
	volume = {12},
	number = {5, Part 2},
	pages = {1287-1296},
	year = {1986},
	issn = {0898-1221},
	doi = {https://doi.org/10.1016/0898-1221(86)90253-1},
	url = {https://www.sciencedirect.com/science/article/pii/0898122186902531},
	author = {L.F. Shampine},
}

@article{STKB2005,
	title = {Non-negative solutions of ODEs},
	journal = {Applied Mathematics and Computation},
	volume = {170},
	number = {1},
	pages = {556-569},
	year = {2005},
	issn = {0096-3003},
	doi = {https://doi.org/10.1016/j.amc.2004.12.011},
	url = {https://www.sciencedirect.com/science/article/pii/S0096300304009683},
	author = {L.F. Shampine and S. Thompson and J.A. Kierzenka and G.D. Byrne}
}

@article{ketcheson2019relaxation,
  title={Relaxation {R}unge-{K}utta Methods: {C}onservation and
         Stability for Inner-Product Norms},
  author={Ketcheson, David I},
  journal={SIAM Journal on Numerical Analysis},
  volume={57},
  number={6},
  pages={2850--2870},
  year={2019},
  publisher={Society for Industrial and Applied Mathematics},
  doi={10.1137/19M1263662},
  eprint={1905.09847},
  eprinttype={arxiv},
  eprintclass={math.NA}
}

@article{ranocha2020relaxation,
  title={Relaxation {R}unge-{K}utta Methods: Fully-Discrete Explicit
         Entropy-Stable Schemes for the Compressible {E}uler and
         {N}avier-{S}tokes Equations},
  author={Ranocha, Hendrik and Sayyari, Mohammed and Dalcin, Lisandro and
          Parsani, Matteo and Ketcheson, David I.},
  journal={SIAM Journal on Scientific Computing},
  volume={42},
  number={2},
  pages={A612--A638},
  year={2020},
  month={03},
  publisher={Society for Industrial and Applied Mathematics},
  doi={10.1137/19M1263480},
  eprint={1905.09129},
  eprinttype={arxiv},
  eprintclass={math.NA}
}

@article{ranocha2020general,
    AUTHOR = {Ranocha, Hendrik and L\'oczi, Lajos and Ketcheson, David I.},
     TITLE = {General relaxation methods for initial-value problems with
              application to multistep schemes},
   JOURNAL = {Numer. Math.},
  FJOURNAL = {Numerische Mathematik},
    VOLUME = {146},
      YEAR = {2020},
    NUMBER = {4},
     PAGES = {875--906},
      ISSN = {0029-599X,0945-3245},
   MRCLASS = {65L06 (65L20 65M12 65M70 65P10)},
  MRNUMBER = {4182089},
MRREVIEWER = {Gabriela\ Schranz-Kirlinger},
       DOI = {10.1007/s00211-020-01158-4},
       URL = {https://doi.org/10.1007/s00211-020-01158-4},
}

@article {CLM2010,
	AUTHOR = {Calvo, M. and Laburta, M. P. and Montijano, J. I. and
	R\'andez, L.},
	TITLE = {Projection methods preserving {L}yapunov functions},
	JOURNAL = {BIT},
	FJOURNAL = {BIT. Numerical Mathematics},
	VOLUME = {50},
	YEAR = {2010},
	NUMBER = {2},
	PAGES = {223--241},
	ISSN = {0006-3835,1572-9125},
	MRCLASS = {65L06 (65L05 93B40 93D05)},
	MRNUMBER = {2640013},
	MRREVIEWER = {Vasile\ Dr\u agan},
	DOI = {10.1007/s10543-010-0259-3},
	URL = {https://doi.org/10.1007/s10543-010-0259-3},
}

@article{ranocha2019mimetic,
  title={Mimetic Properties of Difference Operators: Product and Chain Rules as
         for Functions of Bounded Variation and Entropy Stability of Second
         Derivatives},
  author={Ranocha, Hendrik},
  journal={BIT Numerical Mathematics},
  volume={59},
  number={2},
  pages={547--563},
  year={2019},
  month={06},
  publisher={Springer},
  doi={10.1007/s10543-018-0736-7},
  eprint={1805.09126},
  eprinttype={arxiv},
  eprintclass={math.NA}
}

@article{Mattsson12,
	AUTHOR = {Mattsson, Ken},
	TITLE = {Summation by parts operators for finite difference
	approximations of second-derivatives with variable
	coefficients},
	JOURNAL = {J. Sci. Comput.},
	FJOURNAL = {Journal of Scientific Computing},
	VOLUME = {51},
	YEAR = {2012},
	NUMBER = {3},
	PAGES = {650--682},
	ISSN = {0885-7474,1573-7691},
	MRCLASS = {65M06 (65M12)},
	MRNUMBER = {2914426},
	MRREVIEWER = {Tore\ Fl\aa tten},
	DOI = {10.1007/s10915-011-9525-z},
	URL = {https://doi.org/10.1007/s10915-011-9525-z},
}

@article {KM2019,
	AUTHOR = {Kopecz, Stefan and Meister, Andreas},
	TITLE = {On the existence of three-stage third-order modified
	{P}atankar-{R}unge-{K}utta schemes},
	JOURNAL = {Numer. Algorithms},
	FJOURNAL = {Numerical Algorithms},
	VOLUME = {81},
	YEAR = {2019},
	NUMBER = {4},
	PAGES = {1473--1484},
	ISSN = {1017-1398,1572-9265},
	MRCLASS = {65L06},
	MRNUMBER = {3987250},
	DOI = {10.1007/s11075-019-00680-3},
	URL = {https://doi.org/10.1007/s11075-019-00680-3},
}

@article {AKM2020,
	AUTHOR = {\'{A}vila, Andr\'{e}s I. and Kopecz, Stefan and Meister, Andreas},
	TITLE = {A comprehensive theory on generalized {BBKS} schemes},
	JOURNAL = {Appl. Numer. Math.},
	FJOURNAL = {Applied Numerical Mathematics. An IMACS Journal},
	VOLUME = {157},
	YEAR = {2020},
	PAGES = {19--37},
	ISSN = {0168-9274},
	MRCLASS = {65L06 (65L04 65L20)},
	MRNUMBER = {4109346},
	MRREVIEWER = {B\"{u}lent Karas\"{o}zen},
	DOI = {},
	URL = {https://doi.org/10.1016/j.apnum.2020.05.027},
}

@article {MCD2020,
	AUTHOR = {Martiradonna, Angela and Colonna, Gianpiero and Diele, Fasma},
	TITLE = {{{G}e{C}o}: {G}eometric {C}onservative nonstandard schemes
	for biochemical systems},
	JOURNAL = {Appl. Numer. Math.},
	FJOURNAL = {Applied Numerical Mathematics. An IMACS Journal},
	VOLUME = {155},
	YEAR = {2020},
	PAGES = {38--57},
	ISSN = {0168-9274},
	MRCLASS = {92C40 (65P10)},
	MRNUMBER = {4087156},
	DOI = {10.1016/j.apnum.2019.12.004},
	URL = {https://doi.org/10.1016/j.apnum.2019.12.004},
}

@article{NSARK,
	doi = {10.2140/camcos.2025.20.29},
	journal = {Communications in Applied Mathematics and Computational Science},
      title={Order conditions for Runge--Kutta-like methods with solution-dependent coefficients},
author={Thomas Izgin and David I. Ketcheson and Andreas Meister},
year={2025},
	VOLUME = {20-1},
pages = {29--66},
url={https://doi.org/10.2140/camcos.2025.20.29},
}

@phdthesis{IzginThesis,
	author={Izgin, Thomas},
	title={A Unifying Theory for Runge-Kutta-like Time Integrators: Convergence and Stability},
	school={University of Kassel},
	year={2024},
	DOI = {10.17170/kobra-202402059522},
}

@article {Crouzeix1980,
	AUTHOR = {Crouzeix, M.},
	TITLE = {Une m\'{e}thode multipas implicite-explicite pour l'approximation
	des \'{e}quations d'\'{e}volution paraboliques},
	JOURNAL = {Numer. Math.},
	FJOURNAL = {Numerische Mathematik},
	VOLUME = {35},
	YEAR = {1980},
	NUMBER = {3},
	PAGES = {257--276},
	ISSN = {0029-599X},
	MRCLASS = {65M10},
	MRNUMBER = {592157},
	MRREVIEWER = {Erwin Schechter},
	DOI = {10.1007/BF01396412},
	URL = {https://doi.org/10.1007/BF01396412},
}

@article{bolley1978conservation,
  title={Conservation de la positivit{\'e} lors de la discr{\'e}tisation
         des probl{\`e}mes d'{\'e}volution paraboliques},
  author={Bolley, Catherine and Crouzeix, Michel},
  journal={RAIRO. Analyse num{\'e}rique},
  volume={12},
  number={3},
  pages={237--245},
  year={1978},
  publisher={EDP Sciences}
}

@article{ARS1997,
	title = {Implicit-explicit Runge-Kutta methods for time-dependent partial differential equations},
	journal = {Applied Numerical Mathematics},
	volume = {25},
	number = {2},
	pages = {151-167},
	year = {1997},
	note = {Special Issue on Time Integration},
	issn = {0168-9274},
	doi = {https://doi.org/10.1016/S0168-9274(97)00056-1},
	url = {https://www.sciencedirect.com/science/article/pii/S0168927497000561},
	author = {U. M. Ascher and S. J. Ruuth and R. J. Spiteri}
}

@article {SG2015,
	AUTHOR = {Sandu, A. and G\"{u}nther, M.},
	TITLE = {A generalized-structure approach to additive {R}unge-{K}utta
	methods},
	JOURNAL = {SIAM J. Numer. Anal.},
	FJOURNAL = {SIAM Journal on Numerical Analysis},
	VOLUME = {53},
	YEAR = {2015},
	NUMBER = {1},
	PAGES = {17--42},
	ISSN = {0036-1429},
	MRCLASS = {65L06 (65L05 65L07 65L20)},
	MRNUMBER = {3296613},
	MRREVIEWER = {Justin S. C. Prentice},
	DOI = {10.1137/130943224},
	URL = {https://doi.org/10.1137/130943224},
}

@article{IssuesMPRK,
	AUTHOR = {Torlo, Davide and \"{O}ffner, Philipp and Ranocha, Hendrik},
	TITLE = {Issues with positivity-preserving {P}atankar-type schemes},
	JOURNAL = {Appl. Numer. Math.},
	FJOURNAL = {Applied Numerical Mathematics. An IMACS Journal},
	VOLUME = {182},
	YEAR = {2022},
	PAGES = {117--147},
	ISSN = {0168-9274},
	MRCLASS = {65L06 (65L04)},
	MRNUMBER = {4469089},
	DOI = {10.1016/j.apnum.2022.07.014},
	URL = {https://doi.org/10.1016/j.apnum.2022.07.014},
}

@Article{KM18,
	author     = {Kopecz, S. and Meister, A.},
	journal    = {Appl. Numer. Math.},
	title      = {On order conditions for modified {P}atankar-{R}unge-{K}utta schemes},
	year       = {2018},
	issn       = {0168-9274},
	pages      = {159--179},
	volume     = {123},
	fjournal   = {Applied Numerical Mathematics. An IMACS Journal},
	mrclass    = {65L06},
	mrnumber   = {3711996},
	mrreviewer = {D. Shirokoff},
	publisher  = {Elsevier},
	url        = {https://doi.org/10.1016/j.apnum.2017.09.004},
}

@article{IR2023,
	title={Using Bayesian Optimization to Design Time Step Size Controllers with Application to Modified Patankar--Runge--Kutta Methods},
	author={Thomas Izgin and Hendrik Ranocha},
	year={2023},
	journal = {https://arxiv.org/abs/2312.01796},
	eprint={2312.01796},
	archivePrefix={arXiv}
}

@article{ranocha2021preventing,
  title={Preventing pressure oscillations does not fix local linear stability
         issues of entropy-based split-form high-order schemes},
  author={Ranocha, Hendrik and Gassner, Gregor J},
  journal={Communications on Applied Mathematics and Computation},
  year={2021},
  month={08},
  doi={10.1007/s42967-021-00148-z},
  eprint={2009.13139},
  eprinttype={arxiv},
  eprintclass={math.NA}
}

@article{winters2020entropy,
  title={Entropy stable numerical approximations for the isothermal
         and polytropic {E}uler equations},
  author={Winters, Andrew R and Czernik, Christof and Schily, Moritz B
          and Gassner, Gregor J},
  journal={BIT Numerical Mathematics},
  volume={60},
  number={3},
  pages={791--824},
  year={2020},
  publisher={Springer},
  doi={10.1007/s10543-019-00789-w}
}

@article{cano1997error,
  title={Error growth in the numerical integration of periodic orbits,
         with application to {H}amiltonian and reversible systems},
  author={Cano, B and Sanz-Serna, Jesus Maria},
  journal={SIAM Journal on Numerical Analysis},
  volume={34},
  number={4},
  pages={1391--1417},
  year={1997},
  publisher={SIAM},
  doi={10.1137/S0036142995281152}
}

@article{cano1998error,
  title={Error growth in the numerical integration of periodic orbits by
         multistep methods, with application to reversible systems},
  author={Cano, B and Sanz-Serna, Jesus Maria},
  journal={IMA Journal of Numerical Analysis},
  volume={18},
  number={1},
  pages={57--75},
  year={1998},
  publisher={Oxford University Press},
  doi={10.1093/imanum/18.1.57}
}

@article{calvo2011error,
  title={Error growth in the numerical integration of periodic orbits},
  author={Calvo, Manuel and Laburta, MP and Montijano, Juan I and
          R{\'a}ndez, Luis},
  journal={Mathematics and Computers in Simulation},
  volume={81},
  number={12},
  pages={2646--2661},
  year={2011},
  publisher={Elsevier},
  doi={10.1016/j.matcom.2011.05.007}
}

@article{lefloch2002fully,
  title={Fully Discrete, Entropy Conservative Schemes of Arbitrary Order},
  author={LeFloch, Philippe G and Mercier, Jean-Marc and Rohde, Christian},
  journal={SIAM Journal on Numerical Analysis},
  volume={40},
  number={5},
  pages={1968--1992},
  year={2002},
  publisher={Society for Industrial and Applied Mathematics},
  doi={10.1137/S003614290240069X}
}

@article{fisher2013discretely,
  title={Discretely conservative finite-difference formulations for nonlinear
         conservation laws in split form: {T}heory and boundary conditions},
  author={Fisher, Travis C and Carpenter, Mark H and Nordstr{\"o}m, Jan and
          Yamaleev, Nail K and Swanson, Charles},
  journal={Journal of Computational Physics},
  volume={234},
  pages={353--375},
  year={2013},
  publisher={Elsevier},
  doi={10.1016/j.jcp.2012.09.026}
}

@article{ranocha2018comparison,
  title={Comparison of Some Entropy Conservative Numerical Fluxes for
         the {E}uler Equations},
  author={Ranocha, Hendrik},
  journal={Journal of Scientific Computing},
  volume={76},
  number={1},
  pages={216--242},
  year={2018},
  month={07},
  publisher={Springer},
  doi={10.1007/s10915-017-0618-1},
  eprint={1701.02264},
  eprinttype={arxiv},
  eprintclass={math.NA}
}

@article{chen2017entropy,
  title={Entropy stable high order discontinuous {G}alerkin methods with suitable
         quadrature rules for hyperbolic conservation laws},
  author={Chen, Tianheng and Shu, Chi-Wang},
  journal={Journal of Computational Physics},
  volume={345},
  pages={427--461},
  year={2017},
  publisher={Elsevier},
  doi={10.1016/j.jcp.2017.05.025}
}

@article{tadmor1987numerical,
  title={The numerical viscosity of entropy stable schemes for systems
         of conservation laws. {I}},
  author={Tadmor, Eitan},
  journal={Mathematics of Computation},
  volume={49},
  number={179},
  pages={91--103},
  year={1987},
  publisher={American Mathematical Society},
  doi={10.1090/S0025-5718-1987-0890255-3}
}

@article{nusslein2021positivity,
  title={Positivity-Preserving Adaptive {R}unge-{K}utta Methods},
  author={N\"u{\ss}lein, Stephan and Ranocha, Hendrik and Ketcheson, David I},
  journal={Communications in Applied Mathematics and Computational Science},
  year={2021},
  month={11},
  doi={10.2140/camcos.2021.16.155},
  volume={16},
  number={2},
  pages={155--179},
  eprint={2005.06268},
  eprinttype={arxiv},
  eprintclass={math.NA}
}

@article{horvath_positivity_1998,
  title={Positivity of {R}unge--{K}utta and diagonally split
         {R}unge--{K}utta methods},
  author={Horv{\'a}th, Zolt{\'a}n},
  journal={Applied Numerical Mathematics},
  volume={28},
  number={2-4},
  pages={309--326},
  year={1998},
  publisher={Elsevier},
  doi={10.1016/S0168-9274(98)00050-6}
}

@article{macdonald2007,
	Author = {Macdonald, Colin B. and Gottlieb, Sigal and Ruuth, Steven J.},
	Journal = {Journal of Scientific Computing},
	Month = dec,
	Pages = {89--112},
	Title = {{A Numerical Study of Diagonally Split Runge--Kutta Methods for PDEs with Discontinuities}},
	Volume = {35},
	Year = {2008}
}

@incollection{tadmor2002semidiscrete,
  title={From Semidiscrete to Fully Discrete: {S}tability of
         {R}unge-{K}utta Schemes by the Energy Method {II}},
  author={Tadmor, Eitan},
  booktitle={Collected Lectures on the Preservation of Stability under Discretization},
  editor={Estep, Donald J and Tavener, Simon},
  series={Proceedings in Applied Mathematics},
  volume={109},
  pages={25--49},
  year={2002},
  publisher={Society for Industrial and Applied Mathematics},
  address={Philadelphia}
}

@article{ranocha2018L2stability,
  title={{$L_2$} Stability of Explicit {R}unge-{K}utta Schemes},
  author={Ranocha, Hendrik and {\"O}ffner, Philipp},
  journal={Journal of Scientific Computing},
  volume={75},
  number={2},
  pages={1040--1056},
  year={2018},
  month={05},
  doi={10.1007/s10915-017-0595-4}
}

@article{sun2017stability,
  title={Stability of the fourth order {R}unge-{K}utta method for
         time-dependent partial differential equations},
  author={Sun, Zheng and Shu, Chi-Wang},
  journal={Annals of Mathematical Sciences and Applications},
  volume={2},
  number={2},
  pages={255--284},
  year={2017},
  doi={10.4310/AMSA.2017.v2.n2.a3}
}

@article{sun2019strong,
  title={Strong Stability of Explicit {R}unge-{K}utta Time
         Discretizations},
  author={Sun, Zheng and Shu, Chi-Wang},
  journal={SIAM Journal on Numerical Analysis},
  volume={57},
  number={3},
  pages={1158--1182},
  year={2019},
  publisher={SIAM},
  doi={10.1137/18M122892X},
  eprint={1811.10680},
  eprinttype={arxiv},
  eprintclass={math.NA}
}

@inproceedings{achleitner2024necessary,
  title={Necessary and Sufficient Conditions for Strong Stability of
         Explicit {R}unge-{K}utta Methods},
  author={Achleitner, Franz and Arnold, Anton and J{\"u}ngel, Ansgar},
  pages={1--21},
  year={2024},
  booktitle={From Particle Systems to Partial Differential Equations.
             {PSPDE X}, Braga, Portugal, June 2022},
  editor={Carlen, Eric and Gonçalves, Patrícia and Soares, Ana Jacinta},
  series={Springer Proceedings in Mathematics \& Statistics},
  volume={465},
  publisher={Springer},
  address={Cham},
  doi={10.1007/978-3-031-65195-3_1}
}

@article{tadmor2025stability,
  title={On the stability of {R}unge-{K}utta methods for arbitrarily large
         systems of {ODEs}},
  author={Tadmor, Eitan},
  journal={Communications on Pure and Applied Mathematics},
  volume={78},
  number={4},
  pages={821--855},
  year={2025},
  publisher={Wiley Online Library},
  doi={10.1002/cpa.22238}
}

@article{sun2022energy,
  title={On energy laws and stability of {R}unge-{K}utta methods for linear
         seminegative problems},
  author={Sun, Zheng and Wei, Yuanzhe and Wu, Kailiang},
  journal={SIAM Journal on Numerical Analysis},
  volume={60},
  number={5},
  pages={2448--2481},
  year={2022},
  doi={10.1137/22M1472218},
  eprint={2201.06501},
  eprinttype={arxiv},
  eprintclass={math.NA}
}

@article{friedrich2019entropy,
  title={Entropy Stable Space-Time Discontinuous {G}alerkin Schemes with
         Summation-by-Parts Property for Hyperbolic Conservation Laws},
  author={Friedrich, Lucas and Schn{\"u}cke, Gero and Winters, Andrew Ross
          and Fern{\'a}ndez, David C Del Rey and Gassner, Gregor Josef
          and Carpenter, Mark H},
  journal={Journal of Scientific Computing},
  volume={80},
  number={1},
  pages={175--222},
  year={2019},
  publisher={Springer},
  doi={10.1007/s10915-019-00933-2},
  eprint={1808.08218},
  eprinttype={arxiv},
  eprintclass={math.NA}
}

@article{burrage1979stability,
  title={Stability criteria for implicit {R}unge-{K}utta methods},
  author={Burrage, Kevin and Butcher, John Charles},
  journal={SIAM Journal on Numerical Analysis},
  volume={16},
  number={1},
  pages={46--57},
  year={1979},
  publisher={SIAM},
  doi={10.1137/0716004}
}

@article{burrage1980nonlinear,
  title={Non-linear stability of a general class of differential
         equation methods},
  author={Burrage, Kevin and Butcher, John Charles},
  journal={BIT Numerical Mathematics},
  volume={20},
  number={2},
  pages={185--203},
  year={1980},
  publisher={Springer},
  doi={10.1007/BF01933191}
}

@article{higueras2005monotonicity,
  title={Monotonicity for {R}unge-{K}utta Methods: Inner Product
         Norms},
  author={Higueras, Inmaculada},
  journal={Journal of Scientific Computing},
  volume={24},
  number={1},
  pages={97--117},
  year={2005},
  publisher={Springer},
  doi={10.1007/s10915-004-4789-1}
}

@article{ranocha2021strong,
  title={On Strong Stability of Explicit {R}unge-{K}utta Methods for
         Nonlinear Semibounded Operators},
  author={Ranocha, Hendrik},
  journal={IMA Journal of Numerical Analysis},
  year={2021},
  month={01},
  volume={41},
  number={1},
  pages={654--682},
  publisher={Oxford University Press},
  doi={10.1093/imanum/drz070},
  eprint={1811.11601},
  eprinttype={arxiv},
  eprintclass={math.NA}
}

@article{ranocha2020energy,
  title={Energy Stability of Explicit {R}unge-{K}utta Methods for
         Nonautonomous or Nonlinear Problems},
  author={Ranocha, Hendrik and Ketcheson, David I},
  journal={SIAM Journal on Numerical Analysis},
  year={2020},
  month={11},
  volume={58},
  number={6},
  pages={3382--3405},
  publisher={Society for Industrial and Applied Mathematics},
  doi={10.1137/19M1290346},
  eprint={1909.13215},
  eprinttype={arxiv},
  eprintclass={math.NA}
}

@article{grimm2005geometric,
  title={Geometric integration methods that preserve {L}yapunov
         functions},
  author={Grimm, V and Quispel, GRW},
  journal={BIT Numerical Mathematics},
  volume={45},
  number={4},
  pages={709--723},
  year={2005},
  publisher={Springer},
  doi={10.1007/s10543-005-0034-z}
}

@article{calvo2006preservation,
  title={On the Preservation of Invariants by Explicit {R}unge-{K}utta
         Methods},
  author={Calvo, Manuel and Hern{\'a}ndez-Abreu, D and Montijano, Juan I
          and R{\'a}ndez, Luis},
  journal={SIAM Journal on Scientific Computing},
  volume={28},
  number={3},
  pages={868--885},
  year={2006},
  publisher={SIAM},
  doi={10.1137/04061979X}
}

@article{calvo2010projection,
  title={Projection methods preserving {L}yapunov functions},
  author={Calvo, Manuel and Laburta, MP and Montijano, Juan I and
          R{\'a}ndez, Luis},
  journal={BIT Numerical Mathematics},
  volume={50},
  number={2},
  pages={223--241},
  year={2010},
  publisher={Springer},
  doi={10.1007/s10543-010-0259-3}
}

@article{laburta2015numerical,
  title={Numerical methods for non conservative perturbations of
         conservative problems},
  author={Laburta, MP and Montijano, Juan I and R{\'a}ndez, Luis
          and Calvo, Manuel},
  journal={Computer Physics Communications},
  volume={187},
  pages={72--82},
  year={2015},
  publisher={Elsevier},
  doi={10.1016/j.cpc.2014.10.012}
}

@article{dahlby2011preserving,
  title={Preserving multiple first integrals by discrete gradients},
  author={Dahlby, Morten and Owren, Brynjulf and Yaguchi, Takaharu},
  journal={Journal of Physics A: Mathematical and Theoretical},
  volume={44},
  number={30},
  pages={305205},
  year={2011},
  publisher={IOP Publishing},
  doi={10.1088/1751-8113/44/30/305205}
}

@article{sanzserna1982explicit,
  title={An explicit finite-difference scheme with exact conservation
         properties},
  author={Sanz-Serna, Jesus Maria},
  journal={Journal of Computational Physics},
  volume={47},
  number={2},
  pages={199--210},
  year={1982},
  publisher={Elsevier},
  doi={10.1016/0021-9991(82)90074-2}
}

@book{dekker1984stability,
  title={Stability of {R}unge-{K}utta methods for stiff nonlinear
         differential equations},
  author={Dekker, Kees and Verwer, Jan G},
  series={CWI Monographs},
  volume={2},
  year={1984},
  publisher={North-Holland},
  address={Amsterdam}
}

@misc{IRS2026repository,
  title={Reproducibility repository for
         "A Positivity-Preserving Relaxation Algorithm"},
  author={Izgin, Thomas and Ranocha, Hendrik and Shu, Chi-Wang},
  year={2026},
  howpublished={\url{https://github.com/IzginThomas/PositiveRelaxation}},
  doi={10.5281/zenodo.19386973}
}

		\end{document}